\documentclass[11pt]{article}
\usepackage{amsmath,amsthm,amsfonts,amssymb}
\usepackage{mathrsfs}
\usepackage{xcolor}
\usepackage{xspace}
\usepackage[most]{tcolorbox}
\usepackage[]{authblk}
\definecolor{darkgreen}{rgb}{0.00,0.5,0.00}
\definecolor{darkblue}{rgb}{0.00,0,0.6}
\usepackage[colorlinks,citecolor=darkgreen,urlcolor=blue,linkcolor=darkblue, backref=page,linktocpage=true]{hyperref}
\usepackage[round]{natbib}
\usepackage[margin=1.2in]{geometry}
\usepackage[indent]{parskip}
\allowdisplaybreaks
\numberwithin{equation}{section}
\usepackage{multirow}
\usepackage{makecell}

\newcommand{\E}{\mathbb{E}}
\renewcommand{\Pr}{P}
\renewcommand{\P}{\Pr}

\newcommand{\Q}{Q}
\newcommand{\Ber}{\text{Ber}}

\newcommand{\risk}{\mathcal{R}}
\newcommand{\deltahat}{\widehat{\delta}}
\newcommand{\vphat}{\widehat{\vp}}
\newcommand{\ind}[1]{\mathbf{1}\left\{#1\right\}}
\newcommand{\eps}{\epsilon}
\newcommand{\nptest}{\phi^{\textsc{np}}}

\newcommand{\lr}{\Lambda}
\newcommand{\bstar}{{b^*}}

\newcommand{\np}{{\textsc{np}}}
\renewcommand{\Re}{\mathbb{R}}
\def\ddefloop#1{\ifx\ddefloop#1\else\ddef{#1}\expandafter\ddefloop\fi}
\def\ddef#1{\expandafter\def\csname cal#1\endcsname{\ensuremath{\mathcal{#1}}}}
\ddefloop ABCDEFGHIJKLMNOPQRSTUVWXYZ\ddefloop

\newcommand{\vp}{\varphi}
\renewcommand{\d}{\text{d}}

\DeclareMathOperator*{\argmin}{\textnormal{arg min}}

\renewcommand{\subset}{\subseteq}
\renewcommand{\supset}{\supseteq}
\DeclareMathOperator{\supp}{\text{supp}}

\newcommand{\const}{{\textnormal{const}}}
\newcommand{\all}{{\textnormal{all}}}
\newcommand{\consta}{\mathbf{C}}
\newcommand{\alla}{\mathbf{U}}

\newcommand{\edit}[1]{{\color{black}#1\xspace}}

\newcommand{\remove}[1]{{\color{red} TODO removed some text here; check whether o.k. \xspace}}

\theoremstyle{plain}
\newtheorem{theorem}{Theorem}[section]
\newtheorem{proposition}[theorem]{Proposition}
\newtheorem{corollary}[theorem]{Corollary}
\newtheorem{lemma}[theorem]{Lemma}

\theoremstyle{definition}
\newtheorem{example}[theorem]{Example}
\newtheorem{remark}[theorem]{Remark}
\newtheorem{definition}[theorem]{Definition}
\newtheorem{observation}[theorem]{Observation}

\newif\ifarxiv
\arxivtrue      

\usepackage{tocbasic}
\DeclareTOCStyleEntry[
  beforeskip=.5em plus .6pt,
  pagenumberformat=\textbf
]{tocline}{section}

\title{On admissibility in post-hoc hypothesis testing\footnote{Accepted to the International Journal of Approximate Reasoning.}}
\author[1]{Ben Chugg} 
\author[3]{Tyron Lardy}
\author[1]{Aaditya Ramdas}
\author[2,3]{Peter Gr\"{u}nwald}
{ 
\affil[1]{\small Carnegie Mellon University, USA}
\affil[2]{Leiden University, The Netherlands}
\affil[3]{Centrum Wiskunde \& Informatica, The Netherlands}
}
\date{January 2026}

\begin{document}

\maketitle

\begin{abstract}
    The validity of classical hypothesis testing requires the significance level $\alpha$ be fixed before any statistical analysis takes place. 
    This is a stringent requirement. For instance, it prohibits updating $\alpha$ during (or after) an experiment due to changing concern about the cost of false positives, or to reflect unexpectedly strong evidence against the null. Perhaps most disturbingly, witnessing a p-value $p\ll\alpha$ vs $p= \alpha- \epsilon$ for tiny $\epsilon > 0$ has no (statistical) relevance for any downstream decision-making. 
    Following recent work of \citet{grunwald2024beyond}, we develop a theory of \emph{post-hoc} hypothesis testing,
    enabling $\alpha$ to be chosen after seeing and analyzing the data. 
    To study ``good'' post-hoc tests we introduce $\Gamma$-admissibility, where $\Gamma$ is a set of adversaries which map the data to a significance level. 
    We classify the set of $\Gamma$-admissible rules for various sets $\Gamma$, showing they must be based on e-values, and recover the Neyman-Pearson lemma when $\Gamma$ is the constant map. 
\end{abstract}

{
\setcounter{tocdepth}{1}
\small 
\tableofcontents
}
\newpage 

\section{Introduction}
An epidemiologist runs a clinical trial to test the efficacy of a new drug. She does not choose a significance level $\alpha$ beforehand, but waits to see the results. If the p-value is $p=0.01$ she rejects at level $\alpha=0.01$ and if $p=0.05$ she rejects at $\alpha=0.05$. 
Is this valid statistical practice? 
In the standard theory of hypothesis testing---dominant since the work of Neyman, Pearson, and Fisher in the 1920s and 30s---it is  not. Indeed, if the epidemiologist is prepared to reject at level $p$ for all $p\leq 0.10$ (say), then the true type-I error is 0.10, not $p$. 

In other words, the current paradigm of hypothesis testing requires that the significance level $\alpha$ be chosen independently of (thus without looking at) the data. This is well-known to statisticians, but continues to be a source of frustration and confusion for practitioners, and the epidemiologist's mistake above is unfortunately all too common~\citep{bakan1966test,dar1994misuse,gigerenzer2004mindless}. 
This fact is limiting: once the data has been gathered, if the selected value of $\alpha$ gives a vacuous or uninteresting result, nothing more can be done. The same data should not be used in any further analysis of the same question (or, technically, even a separate question motivated by the first\footnote{This---admittedly philosophically thorny and somewhat controversial---reliance of p-values on counterfactuals is one of their many restrictive features~\citep{pratt1964foundations,wagenmakers2007practical}.}).

The difficulties do not end there. Requiring that $\alpha$ be chosen independently of the data introduces a fundamental tension between evidence and decision-making. Namely, upon observing a p-value $p\leq \alpha$, it is   irrelevant whether $p$ is roughly the same as $\alpha$ or significantly smaller. For example, if $\alpha = 0.05$, then observing $p=0.04$ and $p=10^{-6}$ have the same decision-theoretic consequences. As pointed out recently by \cite{grunwald2024beyond}, this is a frustrating fact: the p-value is ostensibly a measure of evidence against the null hypothesis, yet any $p$-value below $\alpha$  has no formal relevance for any downstream decision task. 

This limitation serves as an uncomfortable reminder that current statistical practice consists of an odd combination of both Ronald Fisher's perspective and that of Jerzy Neyman and Egon Pearson. Whereas Fisher proposed the p-value as measure of evidence~\citep{fisher1925theory,fisher1935logic}, Neyman and Pearson were focused on decision making~\citep{neyman1928use}. The fact that we both report the p-value \emph{and} make a decision at a prespecified significance level is an awkward methodological hybrid that none of these pioneers promoted. 

This hybrid approach reflects (or perhaps introduces) confusion over the proper interpretation of p-values, significance levels, and hypothesis testing in general. And such confusion has substantial practical consequences. ``The problem of roving alphas'' as \citet{goodman1993p} puts it, i.e., choosing $\alpha$ as a function of the data, is well-documented in business and management science~\citep{hubbard2003p,hubbard2011widespread}, medicine~\citep{goodman1999toward}, and psychology~\citep{gigerenzer2014superego}. What are statisticians to do in the face of this mistake? 

One response to the problem of roving alphas is to continue to educate practitioners on the proper use of statistical tools. This, however, doesn't solve the fundamental tension inherent in the standard theory of hypothesis testing. The epidemiologist's mistake is understandable, after all---there are many situations where it is  natural to want to use the p-value to say more than is mathematically warranted. 
\ifarxiv Consider the following three examples where an analyst   might be tempted to incorrectly use p-values in a downstream decision-making task. 
\else 
Consider the following example where an analyst   might be tempted to incorrectly use p-values in a downstream decision-making task. 
\fi

\begin{tcolorbox}[
    breakable, 
  colback=gray!10,      
  colframe=gray!10,     
  sharp corners,         
]
\begin{example}[Investment]
\label{example:drug}
A funder is interested in investing in a company producing a new drug that looks promising based on an early clinical trial. 
The funder wants to buy-in at an amount inversely proportional to the risk that the drug is ineffective. The original trial was run at $\alpha=0.05$ and produced a p-value of less than 0.001. While both the company and the investor would like to use this information, the company can only give guarantees about the false positive rate at the 5\% level. 
\end{example}
\ifarxiv
\vspace{0.5cm}
\begin{example}[Data Exploration]
\label{example:exploration}
A researcher investigates whether a particular gene is associated with increased cancer risk. An analysis yields a p-value of 0.0002. Initially, the researcher had no specific $\alpha$ in mind, but now argues that the result is clearly significant. The practice of publishing the result as significant and reporting $p<0.001$ leads to increased false positive rates. 
\end{example}
\vspace{0.5cm}
\begin{example}[Trading]
\label{example:trading}
A quantitative analyst develops several new trading rules and back tests them on historical data. One  rule produces a p-value of 0.00001 when tested against the null of zero mean return. 
The quant chooses this strategy and, using this p-value, calculates the expected return of this strategy. 
If the analyst would have run the same procedure regardless of the p-value, this calculation dramatically inflates the expected return. 
\end{example}
\fi
\end{tcolorbox}

Faced with such examples, in this work we consider an alternative response to the problem of roving alphas. Instead of pointing out flaws in the epidemiologist's approach, we focus on building new statistical tools which function both as measures of evidence \emph{and} allow for mathematically sound downstream decision-making.  In particular, we develop the theory of \emph{post-hoc hypothesis testing}~\citep{grunwald2024beyond}, which allows for data-dependent significance levels. The catch is that the guarantees that an analyst can give on their procedure changes under this theory, moving from error probabilities to expected losses. If $\alpha$ is fixed beforehand, however, then the theory of post-hoc testing recovers the standard theory of hypothesis testing. Thus, in that sense, post-hoc hypothesis testing is a strict generation of the classical framework.  
Our focus in this paper is to define and classify ``admissible test families'', which are the analogue of uniformly most powerful tests in the post-hoc setting.

One can situate our work as belonging to a recent push to resolve several issues with traditional statistical tools. The inability to handle post-hoc significance levels is just one problem with much of modern statistics---others include the inability to handle optional stopping and optional continuation. The burgeoning area of sequential, anytime-valid inference (SAVI)~\citep{ramdas2023game} is focused on such issues. A fundamental tool in this line of work is the \emph{e-value}~\citep{grunwald2024safe,ramdas2024hypothesis}, an object which will also play a crucial role in this paper. 
Overall, we view our work here as an extension of SAVI: broadly, as an attempt to refine existing statistical technology to make life easier for practitioners and to mitigate the risk of statistical malpractice.  

\subsection{Contributions}
\label{sec:contributions}

Following \citet{grunwald2024safe}, in order to formulate a theory of post-hoc hypothesis testing we begin by following \citet{wald1939contributions} and recast the usual notions of type-I and type-II error probabilities in terms of expectations over data-dependent type-I and type-II loss functions. We require that a test has bounded type-I loss in expectation, where the type-I loss may be chosen adversarially. Such a test is called \emph{type-I risk safe}, a notion which generalizes the property of a test having a bounded type-I error probability.

At this point we depart from the setting studied by \citet{grunwald2024beyond} and introduce a different notion of admissibility. We summarize our contributions below. While the general setting and several of the results are presented for composite hypotheses, we are mainly focused on testing point nulls and alternatives.

Throughout the paper it will often be useful to remark on the differences or similarities between our post-hoc setting and the traditional, non-post-hoc setting. Henceforth, we will often refer to the latter as the ``traditional setting'' or sometimes as the ``Neyman-Pearson paradigm.'' 
\ifarxiv
We refer to Neyman and Pearson because it was they who introduced type-I and type-II errors and the notion of a uniformly most powerful test~\citep{neyman1928use,neyman1933ix}, concepts that we build upon and generalize in this work.   
\fi

\paragraph{Definition of $\Gamma$-admissibility\ifarxiv.\fi} Definition~\ref{def:Pi-admissible} introduces the notion of $\Gamma$-admissibility, where $\Gamma$ is a set of functions which map the data to a type-II loss function. This generalizes the definition of a uniformly most powerful test to the post-hoc setting (Observation~\ref{obs:generalization-of-np} makes this precise). 
Informally, a test is $\Gamma$-admissible if there is no test that has smaller expected type-II loss under each distribution in the alternative and across all mappings in $\Gamma$. 
Our definition of $\Gamma$-admissibility departs from the notion of admissibility studied by \citet{grunwald2024beyond}, which requires that a test be dominated with probability 1 under each distribution in the alternative to be inadmissible. We require that a test be dominated only in expectation, making fewer tests admissible. 

\paragraph{Extension of the likelihood ratio test to the post-hoc setting\ifarxiv.\fi} 
We give an extension of the likelihood ratio test---shown to be uniformly most powerful in the traditional paradigm~\citep{neyman1933ix}---to a post-hoc setting and prove that the extension is $\Gamma$-admissible for any $\Gamma$. This is in Section~\ref{sec:np-tests}. This shows that our definition of admissibility is not too strict; the set of $\Gamma$-admissible tests is non-empty. This investigation also sheds light on an interesting phenomenon in the post-hoc setting: Even for continuous distributions, if randomization is allowed then admissible tests will always be randomized. More specifically, there will be losses on which the test is randomized. Since randomized tests may not be kosher for all applications, this motivates explicitly defining and studying the admissibility of \emph{binary} post-hoc tests which reject or accept with probability one.

\paragraph{General properties of admissible tests\ifarxiv.\fi} 
Moving beyond the likelihood ratio test, we prove 
several general properties for all $\Gamma$-admissible tests for any $\Gamma$. We study both randomized and binary tests. 
Of particular note, we associate with each test an e-variable and show how any admissible test can be written as a function of the associated e-variable. Thus, the e-variable directly determines the behavior of the test. In the opposite direction, given an e-variable, we can associate with it a test obtained by thresholding. See Lemmas~\ref{lem:representation} and \ref{lem:representation-binary} and equations~\eqref{eq:delta-from-e} and~\eqref{eq:delta-from-E-binary} for the details of these relationships. 
A second result worth mentioning is a Rao-Blackwell type-argument~\citep{radhakrishna1945information,blackwell1947conditional}, whereby we demonstrate how to (weakly) improve  any given test by conditioning on a sufficient statistic. This can be found in Section~\ref{sec:rao-blackwellization}. 

\paragraph{$\alla$- and $\consta$-admissibility\ifarxiv.\fi} 
Finally, we study two specific families of maps $\Gamma$ and give necessary and sufficient conditions on tests (and e-variables) to be $\Gamma$-admissible. In particular, we study both $\Gamma = \alla$ and $\Gamma = \consta$ where $\alla$ is the set of all mappings from the data to the losses, and $\consta$ is the set of constant mappings. For a point null and alternative, we give a full characterization of all $\alla$-admissible tests in terms of e-variables for both randomized and binary tests---see Theorem~\ref{thm:M-admissibility} and Theorem~\ref{thm:M-admissibility-binary}. For $\consta$, we give several necessary conditions for a randomized test to be admissible, and provide a complete characterization in the binary case---see Theorem~\ref{thm:const-admissibility-binary}. This final result recovers the Neyman-Pearson lemma \citep{neyman1928use} when instantiated with a single loss function. The curious reader can jump ahead to Table~\ref{tab:admissibility} for a diagramatic overview of our admissibility results. 

To summarize, our main contributions are as follows: 
\begin{enumerate}
    \item We give a definition of admissibility in the post-hoc setting---$\Gamma$-admissibility, see Definition~\ref{def:Pi-admissible}---which generalizes the notion of admissibility in the traditional Neyman-Pearson setting; 
    \item We give a $\Gamma$-admissible generalization of the likelihood ratio test to the post-hoc setting (Section~\ref{sec:np-tests}) for all $\Gamma$; 
    \item For both randomized and binary tests, we classify the set of $\alla$-admissible tests (Theorems~\ref{thm:M-admissibility} and \ref{thm:M-admissibility-binary}). For binary tests, we classify the set of $\consta$-admissible tests (Theorem~\ref{thm:const-admissibility-binary}). 
\end{enumerate}

\subsection{Related Work}
\label{sec:related-work}

The post-hoc framework studied here 
is not the only proposal for how to resolve the issues with traditional hypothesis testing. Bayesian decision theory~\citep{bernardo1994bayesian,gelman1995bayesian} has a devoted following, where one replaces p-values and significance levels with posterior probabilities and Bayes factors, and makes decisions by minimizing posterior loss. A great benefit of the Bayesian paradigm is that it handles data-dependent loss functions for free, which we show formally in  Appendix~\ref{app:bayesian} where we explore the connection between our setting and the Bayesian one.   
The hefty price to pay is that one cannot get prior-free Type-I error probability guarantees.

\edit{Another avenue of related work is the subfield of imprecise probability. Here one tries to avoid making claims with false precision by demanding that decisions be defensible over a set of models and/or priors (as done recently in, e.g.,~\citet{chau2025credal} for two-sample testing). E-values may already be considered as a tool of imprecise probability, as they are closed under convex combinations of the null (sometimes called ``credal sets''). 
Moreover, the e-posterior~\citep{grunwald2023posterior} was recently connected to credal sets~\citep[Proposition 10]{caprio2025joys}. While we do not pursue further connections here, we believe this to be an interesting direction. }

\edit{Besides imprecise probability and Bayesian methods}, other approaches to navigate around the drawbacks of traditional hypothesis testing involve changing what researchers report, e.g., switching p-values with confidence intervals~\citep{gardner1986confidence} or likelihood ratios~\citep{royall2017statistical}. Post-hoc hypothesis testing has more in common with some of these than with others (it plays nicely with the law of likelihood for instance, since the likelihood ratio is an e-value), but importantly it is  the only approach to seek a strict generalization of the traditional Neyman-Pearson paradigm. Given that most hypothesis testing in practice still relies on this paradigm, we believe it worthwhile to study how to expand the scope of its tools instead of casting them aside.

Our work can be viewed as part of a more general focus on e-variables/e-values, which are simple yet surprisingly fruitful mathematical objects that have proven useful across a wide variety of statistical problems. While some authors had previously used them implicitly, it was not until recently (roughly 2020) that excitement grew concerning their applications and they were given a name. Since then, there has been growing interest in both developing the underlying theory of e-variables and deploying them in various problems. 
We refer to \citet{grunwald2024safe}, \citet{ramdas2024hypothesis}, and \citet{vovk2021values} for introductions to e-values and their use in hypothesis testing.

That e-values can play nicely with data-dependent parameters, thus overcoming some of the drawbacks of p-values, has been noted previously by several authors. 
The first work in this vein came from the multiple testing literature and worked with e-values only implicitly~\citep{katsevich2020simultaneous}. 
As far as we are  aware, the first result to explicitly link link e-values and data-dependent thresholds is Lemma 1 of \citet{wang2022false} in the context of false discovery rate control. Then \citet{grunwald2024beyond} put these connection on firm decision-theoretic ground and launched the original investigation into post-hoc hypothesis testing. In contradistinction to this paper, Gr\"{u}nwald considers multiple actions, whereas we consider only accept/reject decisions. Working under a different notion of admissibility, he gives a partial classification of all admissible tests in terms of e-variables. We further discuss his results and his notion of admissibility in Appendix~\ref{app:admissibility}.

As part of this work, \citet{grunwald2024beyond} notes that e-value based confidence intervals and confidence distributions (using the so-called e-posterior~\citep{grunwald2023posterior}) are safe under data-dependent selections of $\alpha$. 
Since then, a number of authors have noted the promise of e-values in enabling testing and estimating when significance levels are data-dependent~\citep{xu2024post,koning2023markov,hemerik2024choosing,xu2025bringing,gauthier2025values}. 
In particular, \citet{koning2023markov} also studies post-hoc hypothesis testing and defines post-hoc p-values. He further relates post-hoc p-values to e-values and shows that any post-hoc test defines an e-value.   
An overview of post-hoc testing and decision-making with e-values was recently given in \citet[Chapter 4]{ramdas2024hypothesis}. 

Besides \citet{grunwald2024beyond}, several other authors have studied questions concerning the admissibility of e-variables. These include \citet{wang2024only}, \citet{vovk2020combining}, and \citet{ramdas2020admissible}. However, their notions of admissibility either differ from ours or the settings are incomparable. \citet{wang2024only} and \citet{vovk2020combining} are concerned with designing e-values that are as large as possible under the alternative, and \citet{ramdas2020admissible} are concerned with admissible anytime-valid inference.

\subsection{Outline}
\label{sec:outline}

Section~\ref{sec:testing_as_decision_theory} introduces the general framework of post-hoc hypothesis testing, recasting the usual notions of type-I and type-II error probabilities as type-I and type-II risk. Here we formalize a post-hoc hypothesis test as a ``test family,'' introducing both randomized and binary test families (Definitions~\ref{def:decision-family} and~\ref{def:binary-decision-family}). Section~\ref{sec:admissibility} introduces $\Gamma$-admissibility and discusses its relationship with traditional hypothesis testing. Section~\ref{sec:assumptions} discusses various assumption we make in the remainder of the paper and gives relevant background on e-values. 
With this machinery in hand, Section~\ref{sec:np-tests} generalizes the likelihood ratio test to a post-hoc setting. Section~\ref{sec:properties} then gives general properties of $\Gamma$-admissible tests for any $\Gamma$. Section~\ref{sec:randomized} studies tests which allow randomization and Section~\ref{sec:binary} studies tests that are restricted to be binary.
Sections~\ref{sec:M-admissibility} and \ref{sec:const_admissibility} study $\alla$-admissibility and $\consta$-admissibility respectively. Section~\ref{sec:summary} concludes. 
All longer proofs are delegated to Appendix~\ref{sec:proofs}.

\section{Post-hoc hypothesis testing}
\label{sec:testing_as_decision_theory}

We begin with a standard hypothesis testing setup. 
Let $\calP$ and $\calQ$ be two disjoint families of probability measures on some measurable space $(\Omega, \calF)$. 
We observe data encoded by a random variable $X:\Omega\to \calX$ taking values in some  space $\calX$ whose behavior is governed by some unknown distribution $R\in \calP\cup \calQ$. 
We are  testing $\calP$ against $\calQ$, meaning that we are  tasked with deciding between the null hypothesis $H_0: R\in \calP$ and the alternative $H_1: R\in\calQ$. We either \emph{reject} or \emph{sustain} the null hypothesis $H_0$. If the former, we are  implicitly adopting the alternative. 

We will introduce the framework of post-hoc hypothesis testing in terms of possibly composite hypotheses $\calP$ and $\calQ$, but thereafter switch to considering only point hypothesis $\calP=\{P\}$ and $\calQ = \{Q\}$. While some of our results extend to the composite case, not all do, and considering point hypotheses keeps the presentation cleaner. 

In order to formalize the notion of post-hoc significance levels, we follow \citet{wald1939contributions} and \citet{grunwald2024beyond} and frame hypothesis testing in terms of expected losses instead of error probabilities. 
For an index set $\calB\subset\Re_{\geq 0}$ which we call a set of \emph{scenarios},  
let $\{ L_b\}_{b\in \calB}$ be a set of \emph{loss functions}, where $ L_b:\{0,1\}\times \{0,1\}\to\Re_{>0}$. The first argument denotes whether we are  considering the loss under the null or alternative: $ L_b(0,\cdot)$ is loss under the former and $ L_b(1,\cdot)$ that under the latter. The second argument denotes whether we sustain (0) or reject (1) the null. We will assume that $ L_b(0,0) =  L_b(1,1) = 0$ for each $b\in\calB$. In other words, making the correct decision (sustaining when $R\in\calP$ and rejecting when $R\in \calQ$) incurs zero loss.  

We choose the letters $\calB$ and $b$ to accord with Example~\ref{example:drug}. For instance, we might consider $ L_b(0,1)$ to be the buy in of the funding agency---that is, the amount they stand to lose if the drug turns out to be ineffective.  
We do not assume that $\calB$ is finite (or even countable). For instance, we might have $ L_b(0,1) = b$ for all $b\geq 1$. (Indeed, this is arguably the most natural choice; see the discussion at the end of this section). Finally, we assume without loss of generality that the losses $ L_b(0,1)$ are increasing in $b$ (if not, we may simply relabel accordingly). Note this does not imply any ordering of the losses under the alternative. 

Next we define the analogue of a hypothesis test in this post-hoc setting. Our decision to sustain or reject the null takes the form of a test family, which sees both the data and the loss, and then makes a (possibly randomized) decision. We emphasize that while we will allow the loss to be data-dependent, the test \edit{takes this loss as an argument.} If it were otherwise, a test would have to be needlessly conservative. This is also the case in the traditional setting: the test knows the parameter $\alpha$. 

In what follows, let $\Ber(p)$ denote a Bernoulli distribution with parameter $p\in[0,1]$. 

\begin{definition}[Test family]
\label{def:decision-family}
    A \emph{test family} $\delta$ is a map $\delta:\calX\times \calB\to [0,1]$. We call $\delta$ a \textit{family} because it represents a hypothesis test for each $b\in\calB$. 
    Given $X\in\calX$ and $b\in\calB$, the decision to reject is drawn as $D\sim \Ber(\delta(X,b))$ where $D=1$ means reject and $D=0$ means sustain.
\end{definition}
We also define \emph{binary} test families, which do not allow randomization.

\begin{definition}[Binary test family]
\label{def:binary-decision-family}    
A \emph{binary test family} $\vp$ is a map $\vp:\calX\times \calB\to \{0,1\}$. 
    Given $X\in\calX$ and $b\in\calB$, we reject if $\vp(X,b) = 1$ and sustain otherwise. 
\end{definition}

When studying binary test families, we will assume that $\calP$ and $\calQ$ consist of continuous distributions. Even in the traditional setting, we know that \edit{powerful} hypothesis tests must generally be randomized for discrete distributions. Further, in the case of discrete distributions and binary tests, there is no neat classification of UMP tests as far as we are aware. Thus, studying admissibility in this case becomes quite challenging.

We will typically denote binary test families with the Greek letters $\vp$ or $\psi$---see Table~\ref{tab:notation}.  
Unless binary test families are mentioned explicitly, it should be assumed in the sequel that we are  discussing test families that could be either randomized or binary. We will denote such families with letters $\delta$ and $\phi$.  We may sometimes refer to test families that allow randomization (Definition~\ref{def:decision-family}) as \emph{randomized} test families to emphasize the distinction from binary test families.

For notational simplicity, for a test family $\delta$ and $a\in\{0,1\}$ define 
\begin{equation}
    L_b(a,\delta(X,b)) \equiv \E_{D\sim \Ber(\delta(X,b))} L_b(a,D), 
\end{equation}
which is the expected loss (type-I or type-II) when playing $\delta(X,b)$ on loss $L_b$. Note that we can write $L_b(0,\delta(X,b)) = L_b(0,1) \delta(X,b)$ and $L_b(1,\delta(X,b)) = L_b(1,0)(1 - \delta(X,b))$, since we are assuming that $L_b(0,0) = L_b(1,1) = 0$.  These two identities will be used extensively. Next we define the risk of a test family.

\begin{definition}[Type-I Risk]
\label{def:risk}
For a test family $\delta$, set 
\begin{align}
\label{def:k-risk}
\risk_\calP(\delta) &\equiv  \sup_{\P\in\calP}\E_{X\sim \P} \sup_{b\in \calB} L_b(0,\delta(X,b)).
\end{align}
We call $\risk_\calP(\delta)$ the \emph{type-I risk} of $\delta$. If $\risk_\calP(\delta)\leq 1$, we say that $\delta$ is \emph{type-I risk safe}. 
\end{definition}

 Definition~\ref{def:risk} implicitly assumes that $\sup_{b} L_b(0,\delta(X,b))$ is measurable. In this work we will make certain modest assumptions that ensure this is the case (see the discussion in Section~\ref{sec:risk-as-adversary}).

Risk is the equivalent of type-I error probability in the standard Neyman-Pearson paradigm. Indeed, consider the case of a single loss $\calB=\{b_0\}$. Then  $\risk_\calP(\delta) = \sup_{P\in\calP} \E_P L_{b_0}(0,\delta(X,{b_0}))
 = L_{b_0}(0,1) \sup_{P\in\calP} \E_P[\delta(X,b_0)]$. A risk bounded by 1 is thus identical to a type-I error probability bounded by  $\alpha>0$ if we take $L_{b_0}(0,1) = 1/\alpha$. The upper bound of 1 on the risk is arbitrary. It can be replaced by any positive constant and our results will scale accordingly. See for instance \citet[Section 4]{ramdas2024hypothesis} which replaces 1 with a constant $\beta>0$. 

 While risk recovers type-I error probability when there is a single loss, risk is of course a different metric than an error probability in general. A framework built on replacing error probabilities with risk may therefore make some readers uncomfortable. We provide a longer defense of risk in Appendix~\ref{app:risk-vs-error}, but let us make two points here. First, risk is actually the norm in decision theory. \edit{Second, type-I risk seems to us the most straightforward generalization of type-I error control}.\footnote{For the interested reader, \citet{koning2023markov} explores two other options, so-called ``geometric'' and ``arithmetic'' post-hoc validity. }

In the definition of type-I risk, the supremum over losses is inside the expectation. This allows the loss to be data-dependent, which can be seen explicitly by writing type-I risk in terms of the supremum over mappings from the data to losses. In particular, under reasonable measure-theoretic assumptions on the data and the losses, we have the equality: 
\begin{equation}
\label{eq:risk-as-adversary}
    \sup_{P\in\calP}\E_{X\sim P}\sup_{b\in\calB} L_b(0,\delta(X,b))= \sup_{P\in\calP}\sup_{B:\calX \to \calB} \E_{X\sim \P} L_{B(X)}(0,\delta(X,B(X))),
\end{equation}  
Each function $B:\calX\to\calB$ in~\eqref{eq:risk-as-adversary} can be thought of as an adversary who chooses a loss based on the data. 
Thus, if a decision $\delta$ is type-I risk safe, it is  \emph{post-hoc} type-I risk safe in the sense that it guarantees a type-I loss of at most 1 on any data-dependent loss. 
While intuitively clear, making this fully mathematically precise requires some non-trivial work. The formal statement and its proof can be fund in Appendix~\ref{sec:risk-as-adversary}.

We define the following fundamental object:
\begin{equation}
\label{eq:E_delta}
    E_\delta(X) \equiv \sup_{b\in \calB} L_b(0,\delta(X,b)).
\end{equation}
Using the letter ``E'' is no coincidence: If $\delta$ is type-I risk safe, then $E_\delta$ is an e-value (see Section~\ref{sec:assumptions}). 
There is also a natural way to define test families from e-values which we will see in Section~\ref{sec:properties}. This creates a bidirectional mapping between e-values and tests, a relationship which will be convenient both for classifying the admissible test families (e.g., Theorem~\ref{thm:M-admissibility}) and in proving properties of admissible tests.

\subsection{Admissible test families}
\label{sec:admissibility}

\edit{In the standard theory of hypothesis testing, two classes of tests are of interest: uniformly most powerful (UMP) tests and admissible tests. A level-$\alpha$ test (i.e., having type-I error bounded by $\alpha$) is UMP if it has greater or equal power than every other level-$\alpha$ test for every $Q\in\calQ$. 
If $\calQ$ is composite, then UMP tests only exist in special settings. These include practically relevant cases, however, such as the one-sided t-test and $\chi^2$-tests, and more generally when $\calP \cup \calQ$ admit a monotone likelihood ratio.
An admissible test, meanwhile, cannot be uniformly improved across all $Q\in\calQ$ (meaning that it cannot be strictly improved for some $Q$ and not made worse on any other $Q$). In this work we generalize admissibility to a post-hoc setting. 
}

Let $\alla \equiv \alla(\calX, \calB)$ be the set of all measurable functions (``adversaries'')  $B:\calX\to\calB$. Intuitively, we call a test family $\delta$ $\Gamma$-inadmissible for some $\Gamma\subset\alla$ if there is another test which is type-I risk safe and never has higher, but sometimes lower, expected type-II loss under some adversary $B\in\Gamma$. More formally:

\begin{definition}[$\Gamma$-admissibility]
\label{def:Pi-admissible}
We say the decision family $\phi$ is \emph{weakly preferable} to $\delta$ relative to $\Gamma\subset\alla$ if it is  type-I risk safe and for all $B\in\Gamma$ and all $Q\in\calQ$,
\begin{equation}
\label{eq:Pi-admissible}
    \E_{X\sim Q} [ L_{B(X)}(1,\phi(X, B(X)))] \leq \E_{X\sim Q} [ L_{B(X)}(1,\delta(X, B(X)))]. 
\end{equation}
We call $\phi$ \emph{strictly preferable} to $\delta$ if the inequality is strict for at least one $B\in\Gamma$ and $Q\in\calQ$. If there exists some type-I risk safe $\phi$ that is strictly preferable to $\delta$ relatively to $\Gamma$ we say that $\delta$ is $\Gamma$-\emph{inadmissible}. If no decision family $\phi$ is strictly preferable to $\delta$ relative to $\Gamma$ then we say that $\delta$ is $\Gamma$-admissible.  (Here is an alternative definition: $\delta$ is $\Gamma$-admissible if for every $\phi$ that is weakly preferable to $\delta$,~\eqref{eq:Pi-admissible} actually holds with equality for all $B\in\Gamma$ and $Q\in\calQ$.)
\end{definition}

\edit{
\begin{remark}
\label{rem:safety-vs-admissibility}
    Unless we consider $\Gamma=\alla$ there is an asymmetry between the adversaries considered in the definition of type-I risk and the definition of $\Gamma$-admissibility. Lemma~\ref{lem:risk-as-adversary} shows that a type-I risk safe rule has expected type-I loss bounded by 1 across any mapping $B:\calX\to\calB$, while in $\Gamma$-admissibility we are  deliberately restricting the mappings under consideration, but only for the alternative (i.e.\ for type II risk). This is by design.
    In the worst case, the loss may be data-dependent---we want to ensure we remain type-I risk safe in this case. However, we might also be cautiously optimistic that it won't be data-dependent, and allow ourselves to choose a test that is $C$-admissible based on this hope, all the while knowing that if we are wrong then at least we are protected under the null. Therefore, our first goal is to find tests which have bounded type-I risk under even the worst case adversary. Only then do we consider  type-II risk. This is analogous to how Neyman and Pearson prioritize type-I error bounded by $\alpha$ and only then look for tests with higher power.     
\end{remark}
}

We leave implicit in Definition~\ref{def:Pi-admissible} whether \edit{we are} considering binary or randomized tests; the definition holds for both. We will always make clear in the text which kinds of tests are under consideration. 

In what follows we will always assume that any collection $\Gamma$ of adversaries includes the constant mappings. This will simplify many of the proofs. It is also intuitive: it would be odd if an adversary could play $L_b(1,0)$ for only a subset of the observations. To formalize this assumption, we will consider only those $\Gamma\supset\consta$ 
where 
\begin{equation}
\label{eq:constant_mappings}
    \consta \equiv \big\{ B\in \alla: B(X) = b^*\text{ for all }X\in\calX \text{ and some }b^*\big\}.
\end{equation}
In the sequel, we will often shorthand $B(X)$ to $B$ in equations to save space. 
\edit{Recall} that for any $b\in\calB$, $L_b(1,\delta(X,b)) =  L_b(1,0)(1 - \delta(X,b))$
(we are assuming that $ L_b(1,1) =  L_b(0,0) = 0$), so~\eqref{eq:Pi-admissible} can be equivalently written as $\E_{X\sim Q}  L_{B}(1,0)\delta(X,B)\leq \E_{X\sim Q}  L_{B}(1,0)\phi(X,B)$. In particular, this implies that if $\phi(X, b)\geq \delta(X,b)$ for some $b\in\calB$ and all $X\in F$ for some $F\subset\calX$, then $\E_\Q [L_b(1, \phi(X,b))\ind{X\in F}] \leq \E_Q [L_b(1, \delta(X,b)\ind{X\in F}]$, a fact we will use often.  

\edit{Despite the differences between being UMP and admissible, it's worth investigating when the two notions coincide.
Suppose that there is only one loss function $L_{b_0}$. If $\phi$ is UMP at level $L_{b_0}^{-1}(0,1)$ then it satisfies $\E_Q[\delta(X,b_0)] \leq \E_Q[\phi(X,b_0)]$ for all $Q\in\calQ$ and all other tests $\delta$.  Thus, in particular, these inequalities are satisfied when $\delta$ is admissible (which notion of admissibility does not matter, since there is only one loss), thus implying that the inequalities must be equalities. Hence $\delta$ has the same power as $\phi$ on each $Q$, and is also UMP.  We may therefore conclude: }

\begin{observation}
\label{obs:generalization-of-np}
    For a single loss $L$ (the traditional Neyman-Pearson setting) with $\alpha = L^{-1}(0,1)$, \edit{if a uniformly most powerful level-$\alpha$ test exists then  any admissible test family is such a test.} 
\end{observation}

One could  consider other notions of admissibility. \citet{grunwald2024beyond} considers a family $\phi$ to be preferable to $\delta$ if, for any adversary, $\phi$ rejects whenever $\delta$ rejects with probability 1. We compare our notion of admissibility with that of \citet{grunwald2024beyond} further in Appendix~\ref{app:admissibility}, but suffice it to say here that requiring that $\phi$ be better than $\delta$ with probability 1 is a strong requirement, and does not recover the notion of a uniformly most powerful test. 

A second natural idea is to move the supremum over $B\in\Gamma$ inside the expectation and replace~\eqref{eq:Pi-admissible} with 
$\E_\Q [\sup_{B\in \Gamma} L_B(1,\phi(X,B))] \leq \E_{X\sim Q}[\sup_{B\in \Gamma} L_B(1,\delta(X,B))]$. That is, we require that the worst case adversary for $\phi$ is no worse than the worst case of $\delta$, in expectation. However, such a definition of admissibility is too strong: the set of admissible families is empty. This is also shown in Appendix~\ref{app:admissibility}.

One might expect that if $\delta$ is $\Gamma_1$-admissible, then it is  $\Gamma_2$-admissible for any $\Gamma_1\subset\Gamma_2$. But this does not necessarily hold. Consider $\Gamma_1 = \{B: B(x) = b_0\}$ for some fixed $b_0$. The family $\delta(X,b) = \nptest_{b_0}(X)$ if $b=b_0$ and $\delta(x,b) = 0$ otherwise is $\Gamma_1$-admissible, but is not $\alla$-admissible by Lemma~\ref{lem:representation} if $\calB$ contains more than $b_0$. To put it another way, by introducing more adversaries we open ourselves to the possibility of a test family having strictly better type-II risk than $\delta$ on one of those adversaries.

\begin{remark}
\label{rem:type-I-risk-adversaries}
Related to the previous remark, one could also consider different sets of adversaries in the definition of type-I risk safety. In particular, we could consider the notion of $(\Gamma_1,\Gamma_2)$-admissibility, where $\Gamma_1$ is the set of mappings considered in type-I risk (i.e., the supremum on the right hand side of~\eqref{eq:risk-as-adversary}), and $\Gamma_2$ is the set of adversaries considered in Definition~\ref{def:Pi-admissible}. This paper would then be concerned with $(\alla,\Gamma)$-admissibility while standard Neyman-Pearson theory focuses on $(\consta, \consta)$-admissibility. While this nicely emphasizes the relationship between our notion of type-I risk safety and that of traditional hypothesis testing, it is burdensome to continuously write $(\alla, \Gamma)$-admissibility. Therefore, to spare ourselves notational pain, in this paper we shorten $(\alla,\Gamma)$-admissibility to $\Gamma$-admissibility.  
\end{remark}

Throughout this work, we will be primarily concerned with testing a  simple null $P$ against simple alternative $Q$. Section~\ref{sec:properties} will give general properties of admissible test families for all $\Gamma$, and then Sections~\ref{sec:M-admissibility} and Section~\ref{sec:const_admissibility} will provide full or partial classifications of $\Gamma$-admissible test families for two particular classes $\Gamma$:
\begin{enumerate}
    \item $\Gamma = \alla$. The first natural set of adversaries to consider is all mappings $B: \calX\to\calB$. Under such a liberal notion of admissibility, a great many test families are admissible. Indeed, under mild assumptions on the losses, every test family $\delta$ which satisfies $\risk_\calP(\delta) = 1$ is admissible. In other words, every test which is defined by a sharp e-value is admissible. 
    \item $\Gamma = \consta$ for $\consta$ as in~\eqref{eq:constant_mappings}. 
    We discuss the interpretation of $\consta$ further in Section~\ref{sec:const_admissibility}, \edit{and provide a concrete motivation in Example~\ref{example:drugB},} but let us mention here that it may be helpful to interpret $\consta$-admissibility in terms of counterfactuals. Namely, if $\delta$ is $\consta$-admissible then for every type-I risk safe $\phi$, there exists some loss $b$ such that if the adversary had played $L_b$, we would not regret playing $\delta$ instead of $\phi$.  
    We will give several necessary conditions for a randomized test family to be $\consta$-admissible and provide a full characterization of $\consta$-admissible binary test families. 
\end{enumerate}
Note that $\alla$ and $\consta$ are, respectively, the largest and smallest sets of adversaries we can consider under the assumption that $\Gamma\supset\consta$: we are studying the two extrema of Definition~\ref{def:Pi-admissible}. 
While we do not explicitly study other sets of adversaries here (apart from when our results easy generalize to handle all families $\Gamma$; cf.\ Section~\ref{sec:properties}), 
\edit{
we believe that other sets could be of interest, as illustrated by Example~\ref{example:drugB} below, which also provides further  motivation for studying $\consta$-admissibility.}

Table~\ref{tab:admissibility} gives an overview of our admissibility results. 

Finally, it is very natural to take $ L_b(0,1) = b$ for $b\geq 1$. In this case, 
\begin{equation}
\label{eq:all_losses}
    \sup_{P\in\calP} \E_\P\left[\sup_{b\in\calB} L_b(0,\delta(X,b))\right] = \sup_{P\in\calP} \E_\P\left[\sup_{\alpha\in(0,1]}\frac{\delta(X,1/\alpha)}{\alpha}\right].
\end{equation}
That is, letting the losses take on any value in $[1,\infty)$ lets us consider any post-hoc significance level $\alpha\in(0,1)$. 
Arguably, the right hand side of \eqref{eq:all_losses} is what one would expect of a definition of post-hoc hypothesis testing 
(see, e.g., \citet[Section 6.1]{wang2022false}, \citet{koning2023markov} and \citet{gauthier2025values}).

\begin{tcolorbox}[
    breakable, 
  colback=gray!10,      
  colframe=gray!10,     
  sharp corners,         
]
\begin{example}[Example~\ref{example:drug}, Continued]
\label{example:drugB}
\edit{
Consider the following idealized instantiation(s) of Example~\ref{example:drug}. Suppose $b$ represents the cost of investment expressed in some monetary unit; for concreteness say the cost is  $\$ b \cdot 10^5$. A funder is given the opportunity to invest $b$ in the drug's production. \\

If the drug turns out to be ineffective yet an investment is made then the investor loses their capital. We formalize this as $L_b(0,0)=0, L_b(0,1)=b$. 
\edit{
When the drug is effective, we consider two pay-off variations: 
\begin{enumerate}
    \item The pay-off is a fixed constant $C^*$. This reflects, for example, a situation in which the funder is the sole investor, and the investment is used to build a factory that can produce a fixed quantity of the drug. We view the loss as the regret incurred if the funder does not invest in the successful drug, so $L_b(1,1) = 0$ and $L_b(1,0) = C^*$ for all $b$. 
    \item The pay-off is a strictly increasing function $f(b)$ of $b$. This could for example be the case if larger $b$'s represent a larger factory (more productive capacity) or if there are many funders who each get a share of profits  proportional to their investment (in that case $f$ would be linear). 
    In this case, again viewing $L_b(1,\cdot)$ as the regret of not investing in a successful drug, we set $L_b(1,0) = f(b)$ and $L_b(1,1) = 0$. 
\end{enumerate}
}

In the first variation, $L_b(1,0)=C^*$ is independent of $b$, showing that $\Gamma = \consta$ is an important special case to consider. 
In the second variation, $L_b(1,0)$ is increasing in $b$. It then seems reasonable to assume that the investment $b$ is increasing in the amount of evidence for effectiveness---in practice the `adversaries' might be the data-analysts in a company, who suggest to management to make a larger investment the more evidence they see in the data. This reflects how p-hacking typically occurs in practice: an analyst is tempted to lower $\alpha$ (increase $L_b(0,1)$, which, in the second case also increases $L_b(1,0)$) after seeing promising results. This suggests an analysis in which $\Gamma$ is restricted to those $B(X)$ which are (monotonically) increasing in a sufficient statistic of the data (e.g., the likelihood ratio); thus $\consta \subsetneq \Gamma\subsetneq \alla$. On the other hand, if we allowed $f(b)$ to be all functions of $b$, not just increasing functions, then $\Gamma = \alla$. 
}
\end{example}
\end{tcolorbox}

\subsection{Notation, assumptions, and further background}
\label{sec:assumptions}

As discussed previously, our results rely heavily on e-variables/e-values. 
An e-variable for a set of distributions $\calP$ is a nonnegative random variable $E$ with expected value at most 1 under the null: $\sup_{\P\in\calP} \E_\P[E]\leq 1$. The realized value of an e-variable is called an e-value. E-variables are typically functions of the underlying data $X$ and we will often write $E = E(X)$. 
We call an e-variable $E$ \emph{sharp} if $\sup_{P\in\calP}\E_P[E]=1$.

We use $\ind{\cdot}$ as an indicator, i.e., $\ind{F}$ is 1 if $F$ occurs and 0 otherwise. 
Throughout, we synonymize ``increasing'' with ``nondecreasing,'' and not with ``strictly increasing.'' 
Given a function $h:\Re \to \Re$, let $h^-(y) = \inf\{z: h(z) \geq y\}$ denote its (lower) generalized inverse. If $h$ is strictly increasing and $h^{-1}$ is well-defined, then $h^- = h^{-1}$. If $h$ is increasing, then $h^-$ is left-continuous and $\ind{h(x) \geq y} = \ind{x\geq h^-(y)}$. 
The generalized inverse will appear mainly in Section~\ref{sec:const_admissibility_binary}.

For a point null $P$ and point alternative $Q$ with $Q\ll P$ we use $\lr$ to refer to the likelihood ratio $\lr(X) = \frac{\d\Q}{\d\P} (X)$ where $\d\Q/\d\P$ is the Radon-Nikodym derivative between $Q$ and $P$. 
We let $\calL = \{L_b(0,1):b\in\calB\}$ be the ``span'' of the type-I losses, to appear in Section~\ref{sec:properties}.

\begin{table}[t]
    \centering
    \textbf{Notation}\\
    \vspace{0.4cm}
    \begin{tabular}{rll}
    \hline 
         $\delta,\phi$& Test families, either binary or randomized  & Definition~\ref{def:decision-family} \\
          $\vp, \psi$ & Binary test families & Definition~\ref{def:binary-decision-family} \\ 
         $t_\vp$ & Decision curve of (admissible) binary test $\vp$ & Corollary~\ref{cor:decision-curve} \\ 
         $E_\delta$ & e-variable associated with $\delta$ & Equation~\eqref{eq:E_delta} \\ 
         $\delta_E,\phi_E$ & Randomized test families defined by $E$ & Equation~\eqref{eq:delta-from-e} \\ 
         $\vp_E, \psi_E$ & Binary test families defined by $E$ & Equation~\eqref{eq:delta-from-E-binary}\\
         $\calL$ & Set of all type-I losses, $\{L_b(0,1):b\in\calB\}$ & Equation~\eqref{eq:loss-space}\\
         $\alla$ & All measurable maps from $\calX$ to $\calB$ & --- \\ 
         $\consta$ & All constant maps from $\calX$ to $\calB$ & Equation~\eqref{eq:constant_mappings} \\
         \hline 
    \end{tabular}
    \caption{Some common notation used throughout the paper. }
    \label{tab:notation}
\end{table}

Let us now discuss some assumptions. 
Suppose we observe $X$ which has no support under any $P\in\calP$. Any reasonable test family must reject the null in this case. If not, then a test family which does so will have lower type-II risk on all losses. Similarly, a test family which observes $X$ with no support under any $Q\in\calQ$ must sustain. This leads to the following observation. 
\begin{observation}
\label{eq:Q<<P}
    If $\delta$ is a $\Gamma$-admissible test family for any $\Gamma$, then for any $A\subset \calX$ with $\sup_{\P\in\calP}\P(A)=0$, $\delta(X,b) = 1$ for all $X\in A$ $\calQ$-almost surely and all $b\in\calB$. Likewise, for any $C\subset\calX$ with $\sup_{Q\in\calQ} Q(C) = 0$, we have $\delta(X,b) = 0$ for all $X\in C$ $\calP$-almost surely and $b\in\calB$. 
\end{observation}

Thus, going forward, we may safely assume that $\bigcup_{\Q\in\calQ} \supp(\Q) = \bigcup_{\P\in \calP}\supp(\P)$ (where $\supp(P)$ is the support of $P$). In the case of testing a simple null $\P$ against a simple alternative $\Q$, this implies that both $\Q$ and $P$ are absolutely continuous with respect to one another ($\Q\ll \P$ and $P\ll Q$). 
Because of this assumption, we will usually say $P$-almost surely, instead of $P$- and $Q$-almost surely.

We will also assume that there exists some $b$ such that $L_b(0,1)>1$. Otherwise a test family may always reject with probability 1 and remain type-I risk safe, making the problem uninteresting. This assumption is akin to not studying $\alpha=1$ in the traditional Neyman-Pearson paradigm, which is uninteresting for the same reason. We also make the necessary (quite modest) measure-theoretic assumptions so that~\eqref{eq:risk-as-adversary} holds. See Appendix~\ref{sec:risk-as-adversary} for precise details. Finally, going forward, we assume that we are testing a point null $\calP=\{P\}$ against a point alternative $\calQ=\{Q\}$. 

\edit{We end this section by noting how one might modify the notation we use throughout the paper. 

\begin{remark}[{\em Koning representation\/} of post-hoc hypothesis tests]
For the choice $L_b(0,1)=b$, \citet{koning2024continuous} recently suggested merging the loss into the decision space of the test. Instead of a test mapping to $[0,1]$, he rescales the range to $[0,L_b(0,1)]\supseteq [0,1]$. One could do the same thing throughout this paper. Such a transformation amounts to rescaling our usual test family $\delta$ and defining a new test family $f$ given by  $f(X,b) = L_b(0,1)\delta(X,b)$. In this case, type-I risk safety is the condition $\E_P[\sup_b f(X,b)]\leq 1$. That is, the test itself becomes an e-value, as opposed to the quantity $\sup_b L_b(0,\delta(X,b))$. While 
Koning's representation definitely has a strong appeal, 
in 
this paper we opt to work with the standard range of $[0,1]$ in the hope that it makes the framework more familiar to those new to post-hoc testing and e-values.
\end{remark}
}

\section{Warmup: Likelihood ratio tests}
\label{sec:np-tests}

In the case of a single loss function, the Neyman-Pearson lemma~\citep{neyman1933ix} implies that the likelihood ratio test is  most powerful. What can we say about likelihood ratio-style tests in a setting with multiple losses? 

For each fixed $b\in\calB$, let $\nptest(X,b)$ denote the likelihood ratio test on loss $ L_b(0,1)$. Let $\lr$ be the likelihood ratio between $Q$ and $P$. The likelihood ratio test has the form 
\begin{equation}
    \label{eq:np-test}
    \nptest(X,b) = \ind{\lr(X) > \kappa(b)} + \gamma \ind{\lr(X) = \kappa(b)}, 
\end{equation}
where $\kappa(b)$ and $\gamma$ are chosen such that 
\begin{equation*}
\E_P[\nptest(X,b)] = L_b^{-1}(0,1). 
\end{equation*}
We use the superscript ``\textsc{NP}'' to refer to Neyman-Pearson. Note that the second term $\gamma\ind{\lr(X) = \kappa(b)}$ is only relevant for discrete distributions. We refer to \citet{shao2008mathematical} for a modern treatment of the Neyman-Pearson lemma and the optimality of the likelihood ratio test. 

\edit{Let $\delta$ be an arbitrary $\Gamma$-admissible test family and suppose that it acts as $\nptest(\cdot,b^*)$ on some $b^*$ (that is, $\delta(X,b^*) = \nptest(X,b^*)$)}. How does it behave for $b\neq b^*$? 
\edit{Naively, one would like to set $\delta(\cdot,b)=\nptest(\cdot,b)$ on each $b$.} Such a test family would not be type-I risk safe, however. This is easy to see, even in the case of two losses. \edit{Suppose that $\delta(X,b_1) = \nptest(X,b_1)$ and $\delta(X,b_2)=\nptest(X,b_2)$ with $b_1<b_2$. }
Since $L_{b_2}(0,1)>L_{b_1}(0,1)$ we have $\kappa({b_2})>\kappa({b_1})$. If an adversary $B$ plays loss $b_2$ for $X$ with $\lr(X)\geq\kappa({b_2})$ and $b_1$ for  $X$ with $\kappa({b_1})\leq \lr(X)<\kappa({b_2})$, then $\E_P[L_B(0,\delta(X,B))] = L_{b_2}(0,1)\P(\lr(X)\geq \kappa({b_2})) + L_{b_1}(0,1)\Pr(\kappa(b_1)\leq \lr(X)<\kappa({b_2}))>1$, so $\risk_P(\delta)>1$.

However, for $b\leq b^*$, $\delta(\cdot,b)$ \emph{can} act as $\nptest(\cdot,b^*)$ and remain type-I risk safe. (In fact, in the discrete case, for those $X$ such that $\lr(X) = \kappa(b)$ it can even raise the probability of rejection due to the extra freedom afforded by the fact that $L_b(0,1) \leq L_{b^*}(0,1)$.)  
For $b>b^*$ meanwhile, $\delta(\cdot,b)$ must lower the probability of rejection, but may still reject.  That is, even in the continuous case, in contrast to the NP setting with a single loss function, if we allow randomized test families then admissible test families will {\em always\/} be randomized. Formally, Proposition~\ref{prop:np-decision-rules} demonstrates that, for any $\Gamma\supset\consta$, the only admissible test family which plays as $\nptest(\cdot,b^*)$ on $b^*$ has the following form: 
\begin{equation}
\label{eq:np-rule}
    \delta(X,b) = \min\left\{1, \frac{L_{b^*}(0,1)}{L_b(0,1)}\nptest(X,b^*)\right\}. 
\end{equation}
The proof is in Appendix~\ref{sec:proof-np-decision-rules}. One may prove the result directly, but instead we leverage some of the properties of general admissible rules proven in Section~\ref{sec:properties}. Even so, the proof is more involved than one might initially expect.

\begin{proposition}
\label{prop:np-decision-rules}
The test family $\delta$ defined by~\eqref{eq:np-rule} is $\Gamma$-admissible for any $\Gamma\supset\consta$. Conversely, if $\phi$ is $\Gamma$-admissible for any $\Gamma\supset\consta$ and there exists some $b^*$ such that $\phi(X,b^*) = \nptest(X,b^*)$ $P$-almost surely then $\phi$ acts as~\eqref{eq:np-rule} $P$-almost surely. 
\end{proposition}

The representation of $\delta$ in~\eqref{eq:np-rule} foreshadows a more general pattern that we explore further in Section~\ref{sec:properties}. In particular, after noting that $E_\delta(X) = L_{b^*}(0,1) \nptest(X,b^*)$, we see that we can write $\delta(X,b) = \min\{1, E_\delta(X) / L_b(0,1)\}$ which is a representation that holds for all admissible tests (for any $\Gamma$) and illustrates a fundamental connection between e-values and test families.

It is worth noting that when $P$ and $Q$ are continuous, \eqref{eq:np-rule} can be rewritten as:
\begin{equation}
\label{eq:np-rule-continuous}
    \delta(X,b) = \begin{cases}
        \frac{L_{b^*}(0,1)}{L_b(0,1)}\nptest(X,b^*),& b>\bstar, \\ 
        \nptest(X,b^*), &b\leq \bstar.
    \end{cases}
\end{equation}
This representation is perhaps more intuitive to digest. 
If an adversary $B$ plays $B(X) = b^*$ for all $X$, the expected type-I risk $\E_P[L_B(1,\delta(X,B)]$ is precisely 1 by definition of the likelihood ratio test. 
For $b\leq b^*$ we have $L_b(0,1)\leq L_{b^*}(0,1)$ so an adversary playing such a $b$ can only lower the type-I risk and $\delta$ can play as $\nptest(X,b^*)$. For $b>b^*$, $\delta$ may still reject, but it must lower its rejection probability to make up for the increase in type-I risk. We set $\delta$ such that $L_b(0,1)\delta(X,b) = L_{b^*}(0,1) \nptest(X,b^*)$ so that the type-I risk stays the same. \edit{We note that \citet{koning2023markov} also studied~\eqref{eq:np-rule-continuous} (see his example 9). He cites it as an example of a post-hoc test, though he does not  prove that it is admissible.}

The test family in~\eqref{eq:np-rule-continuous} illustrates a general phenomenon we discussed previously, namely that admissible rules will always be randomized (or, more precisely, there will exist losses on which the test will be randomized), even for continuous distributions. Because randomization may not be desirable, it is  reasonable to search for binary test families. In this case, we can show that 
\begin{equation}
\label{eq:np-rule-binary}
    \vp(X,b) = 
    \begin{cases}
        0,& b>\bstar, \\ 
        \nptest(X,b^*), &b\leq \bstar,
    \end{cases}
\end{equation}
is admissible when $P$ and $Q$ are assumed to be continuous. When $P$ and $Q$ are discrete, then $\nptest(\cdot,b^*)$ is randomized---see~\eqref{eq:np-test}---so this case is not of interest. The proof of the following proposition is in Appendix~\ref{sec:proof-binary-np-rules}. 
\begin{proposition}
\label{prop:binary-np-rules}
Fix any $\Gamma\supset\consta$ and 
suppose $P$ and $Q$ are continuous. 
If we restrict our attention to binary test families, then \eqref{eq:np-rule-binary} is $\Gamma$-admissible. Moreover, if $\psi$ is any binary $\Gamma$-admissible test family and there exists some $b^*$ such that $\psi(X,b^*) = \nptest(X,b^*)$, then $\psi$ acts as~\eqref{eq:np-rule-binary} $P$-almost surely.  \end{proposition}

Just as~\eqref{eq:np-rule} was illustrative of a more general representation of test families when randomization is allowed, \eqref{eq:np-rule-binary} is illustrative of a general representation in the binary case. Indeed, Lemma~\ref{lem:representation-binary} in Section~\ref{sec:properties} will show that, for any $\Gamma\supset\consta$,  a $\Gamma$-admissible binary test $\vp$ can be written as $\vp(X,b) = \ind{E_\delta(X) \geq L_b(0,1)}$.

\section{Properties of admissible rules}
\label{sec:properties}

Here we prove properties of $\Gamma$-admissible families that  hold  for general  $\Gamma\supset\consta$, the most important results being  lemmas~\ref{lem:representation} and \ref{lem:representation-binary}, which provide explicit representations of any admissible test $\delta$ in terms of its associated e-value $E_\delta$ (defined in~\eqref{eq:E_delta}), generalizing equations~\eqref{eq:np-rule} and~\eqref{eq:np-rule-continuous}. 
All omitted proofs can be found in Appendix~\ref{sec:proofs}.

We begin by proving an intuitive property of any test family: If it can be modified to reject with strictly higher probability on some loss without sacrificing type-I risk safety, then it cannot be admissible. This fact is immediate in the traditional Neyman-Pearson framework for any uniformly most powerful test (indeed, it is the definition of being uniformly most powerful). But it requires proof in the post-hoc setting.  

\begin{lemma}
\label{lem:phi>delta}
    Let $\phi, \delta$ be two type-I risk safe test families such that $\phi(X,b) \geq \delta(X,b)$ $P$-almost surely for all $b\in\calB$. If $\phi$ is anywhere strictly greater than $\delta$, i.e., there exists some $b^*\in\calB$ such that $P(\phi(X,b^*)> \delta(X,b^*))>0$, then $\delta$ is $\Gamma$-inadmissible for any $\Gamma\supset\consta$. The same statement holds if $\phi$ and $\delta$ are binary test families. 
\end{lemma}

One property which follows from Lemma~\ref{lem:phi>delta} is that any admissible test must be decreasing in the loss. In other words, as the loss increases, a test should only become more conservative. \edit{This fact was also noted by \citet[Remark 4]{koning2023markov}, though not related to admissibility.}

\begin{lemma}
    \label{lem:decreasing_in_b}
    Fix any $\Gamma\supset\consta$ and let $\delta$ be $\Gamma$-admissible. Then $\delta(X,b_1)\geq  \delta(X,b_2)$ $P$-almost surely if and only if $b_1\leq b_2$. The same statement holds for binary test families. 
\end{lemma}

A second consequence of Lemma~\ref{lem:phi>delta} is that the e-value of an admissible test cannot be strictly dominated by the e-value of another admissible test. Consequently, admissible tests define sharp e-values (recall that $E_\delta(X) = \sup_b L_b(0,\delta(X,b))$). The proof is postponed to Appendix~\ref{sec:proof-Ephi>Edelta}. Note that the properties of $\Gamma$-admissible tests given by Lemmas~\ref{lem:decreasing_in_b} and \ref{lem:Ephi>Edelta} are both illustrated by the Neyman-Pearson style tests of Section~\ref{sec:np-tests}.

\begin{lemma}
\label{lem:Ephi>Edelta}
    Let $\phi, \delta$ be two type-I risk safe test families. If $E_\phi(X)\geq E_\delta(X)$ $P$-almost surely and $P(E_\phi(X) > E_\delta(X))>0$ then $\delta$ is $\Gamma$-inadmissible for any $\Gamma\supset\consta$. Consequently, if $E_\delta$ is not sharp then it is  $\Gamma$-inadmissible for any $\Gamma\supset\consta$. The same statement holds if $\phi$ and $\delta$ are binary test families and $P$ and $Q$ are continuous.  
\end{lemma}

Given a test family $\delta$ we can study the properties of $E_\delta$, sometimes giving necessary and sufficient conditions on $E_\delta$ for $\delta$ to be admissible. But we will also be interested in the other direction. That is, given an e-variable $E$ for $P$, the next sections will define an associated test family $\delta_E$---Section~\ref{sec:randomized} in the randomized case and Section~\ref{sec:binary} in the binary case. We may then ask about the properties of $E$ that lead to admissible tests. A helpful concept when reasoning about the behavior of e-variables is that of \emph{compatibility}. Recall that 
\begin{equation}
\label{eq:loss-space}
    \calL \equiv \{L_b(0,1):b\in\calB\}. 
\end{equation}

\begin{definition}
    \label{def:compatible}
    We say an e-variable $E$ for $P$ is \emph{compatible} with $\calL$ if $E(X) \in \calL\cup\{0\}$ $P$-almost surely. If $\calL$ is understood from context, we will often simply say that $E$ is compatible. 
\end{definition}

Lemma~\ref{lem:compatible} will show that if $E$ defines a $\Gamma$-admissible binary test family for any $\Gamma$ then $E$ is compatible. Such a general result does not hold in the randomized case, which can be seen by considering the e-values of Neyman-Pearson tests when the distributions are discrete. 
However, Lemma~\ref{lem:compatible-inc} proves that when $P$ and $Q$ are continuous and the randomized test $\delta_E$ is $\consta$-admissible, then $E$ is compatible. Interestingly, the issue of compatibility also arose for \citet{grunwald2024beyond}, even when studying a different notion of admissibility. He called the condition ``richness.'' 

Beyond this point, the properties of randomized test families and binary test families begin to diverge so we give a separate treatment of each.

\subsection{Randomized test families} 
\label{sec:randomized}

Throughout this subsection we assume that all test families allow for randomization. The distributions $P$ and $Q$ can be continuous or discrete, unless stated otherwise. As usual, we assume they have the same support. Our first result provides an explicit form for any test family in terms of its associated e-value. 

\begin{lemma}
\label{lem:representation}
    Fix any $\Gamma\supset\consta$. 
     If $\delta$ is $\Gamma$-admissible then for all $b\in\calB$,  
     \begin{equation}
     \label{eq:canonical}
         \delta(X,b) = \min\left\{1,\frac{E_\delta(X)}{ L_b(0,1)}\right\}, \quad \text{P-almost surely}. 
     \end{equation}
\end{lemma}

We will sometimes call~\eqref{eq:canonical} the \emph{canonical} representation of $\delta$, or simply say that $\delta$ is canonical if written as such. Lemma~\ref{lem:representation} will be particularly useful in the upcoming theorems on admissibility, since it lets us assume that all test families we deal with are canonical. \edit{We note that \citet[Corollary 2]{koning2023markov} shows that the test in~\eqref{eq:canonical} is type-I risk safe.}

Just as test families define e-variables, it will be convenient to enable e-variables to define test families. In view of Lemma~\ref{lem:representation}, given an e-variable $E$ for $\calP$, we define the corresponding test family as 
\begin{equation}
\label{eq:delta-from-e}
    \delta_E(X,b) \equiv \min\left\{1, \frac{E(X)}{L_b(0,1)}\right\}.
\end{equation}
It is easy to verify that this definition is consistent in the sense that $\delta_{E_\delta}(X,b)  = \delta(X,b)$ for a $\Gamma$-admissible test $\delta$ (this follows immediately from Lemma~\ref{lem:representation}). Further, if we begin with an e-variable $E$, define $\delta_E$ and then compute $E_{\delta_E}$, we recover $E$ under weak assumptions. To wit, 
\begin{align}
\label{eq:EdeltaE_vs_E}
    E_{\delta_E}(X) &= \sup_{b\in\calB} L_b(0,\delta_E(X,b)) 
    = \sup_{b\in\calB} \min\{L_b(0,1), E(X)\} = E(X) \wedge \sup_b L_b(0,1).
\end{align}
If $\sup L_b(0,1) = \infty$ then the final quantity is $E(X)$. In fact, we show next that if $\delta_E$ is admissible, then  $E(X) \wedge \sup_b L_b(0,1) = E(X)$ $P$-almost surely, which suffices for us to conclude that the final expression in \eqref{eq:EdeltaE_vs_E} equals $E$.

\begin{lemma}
    \label{lem:E<supb}
    Let $E$ be an e-variable for $P$. If $P(E(X) > \sup_bL_b(0,1)) > 0$ then $\delta_E$ is $\Gamma$-inadmissible for any $\Gamma\supset\consta$. Otherwise, $E_{\delta_E} = E$ $P$-almost surely, and in particular this holds if $\delta_E$ is $\Gamma$-admissible.
\end{lemma}

The proof can be found in Appendix~\ref{sec:proof-E<supb}. It is helpful to view Lemma~\ref{lem:E<supb} as a weaker form of compatibility.  It doesn't guarantee that $E$ takes values in $\calL\cup\{0\}$, but it does guarantee that it won't take values strictly larger than those in $\calL$. As we discussed earlier, Lemma~\ref{lem:E<supb} can be strengthened to bonafide compatibility when we consider admissibility with respect to constant adversaries. 
We now move onto studying binary test families.

\subsection{Binary test families}
\label{sec:binary}

Throughout this section we assume that $P$ and $Q$ are continuous. We begin with the analogue of Lemma~\ref{lem:representation} but for binary test families. The same representation was given recently in \citet[Chapter 4]{ramdas2024hypothesis}, though stated as a definition, not as the unique form of any admissible binary test family. \edit{Let us note that \citet[Theorem 2 and Corollary 2]{koning2023markov} proves a very similar result. In particular, after translating his result to our setting, he shows that given an e-variable $E$, we can obtain a type-I risk safe test via $(X,b)\mapsto \ind{L_b(0,1) \leq E(X)}$. Conversely, given a type-I risk safe test, then $E_\phi = \sup\{L_b(0,1) \leq E(X)\}$ is an e-value. We show that, starting with $\phi$, one can write $\phi(X,b) = \ind{L_b(0,1)\leq E_\phi(X)}$ if $\phi$ is admissible. This also follows from Koning's result after only slightly more work.    }

\begin{lemma}
\label{lem:representation-binary}
Fix any $\Gamma\supset\consta$.
    If $\vp$ is a $\Gamma$-admissible binary test family then, for all $b\in\calB$, 
    \begin{equation}
    \label{eq:canonical-binary}
        \vp(X,b) = \ind{ L_b(0,1) \leq E_{\vp}(X)} \quad P\text{-almost surely}.
    \end{equation}
\end{lemma}

Similarly to the randomized case, let us say that a test family with the representation in~\eqref{eq:canonical-binary} is \emph{canonical}, or has canonical representation. And again as we did in the randomized case, let us associate test families to e-variables. Given an e-variable $E$ for $P$, we define its associated binary test family as
\begin{equation}
\label{eq:delta-from-E-binary}
    \vp_E(X,b) \equiv \ind{ L_b(0,1) \leq E(X)}.
\end{equation}
As usual, we use $\vp_E$ or $\psi_E$ to denote the binary test family associated with $E$, as opposed to the randomized test defined by~\eqref{eq:delta-from-e}. We will always explicitly mention whether we are considering randomized or binary tests. 

Both~\eqref{eq:canonical-binary} and~\eqref{eq:delta-from-E-binary} are satisfying in that they are precisely what one would obtain from naively applying the results in the previous section and rounding all non-integer values taken by a test family to zero. Similarly to the randomized case, it is easy to check that $\vp_{E_\vp}(X,b) = \vp(X,b)$ $P$-almost surely. Moreover, given an e-variable $E$, we have 
\begin{equation}
    E_{\vp_E}(X) = \sup_b L_b(0,1)\ind{E(X) \geq L_b(0,1)},
\end{equation}
which is equal to $E(X)$ if $E(X) = L_b(0,1)$ for some $b$, or $E(X)=0$.
In other words, if $E(X)$ is compatible with $\calL$ then $E_{\vp_E} = E$ for any e-variable $E$.

If $E$ is not compatible with $\calL$, then the test family $\delta_E$ is ``wasting'' some of its risk budget for no reduction in type-II loss. Indeed, suppose that for some $X$  $L_{b_1}(0,1) < E(X) < L_{b_2}(0,1)$ where there is no other loss between $L_{b_1}(0,1)$ and $L_{b_2}(0,1)$. Then $E(X)$ defines the same test as taking $\widetilde{E}(X) = L_{b_1}(0,1)$. Moreover, $\widetilde{E}(X)$ can take advantage of the reduction in type-I risk in order to take larger values elsewhere and improve its type-II risk. The following result formalizes this argument. Its proof is provided in Appendix~\ref{sec:proof-compatible}.

\begin{lemma}
\label{lem:compatible}
    Let $E$ be an e-variable for $P$ and fix any $\Gamma\supset\consta$. If $\vp_E$ is $\Gamma$-admissible then $E$ is compatible with $\calL$. Consequently, if $\vp_E$ is $\Gamma$-admissible then $E = E_{\vp_E}$. 
\end{lemma}

\section{Classifying the admissible rules}
\label{sec:classification}

The following two sections study $\alla$-admissibility and $\consta$-admissibility. Table~\ref{tab:admissibility} provides a summary of the results. 

\begin{table}[t]
\centering
\textbf{Necessary and sufficient conditions for admissibility}\\ 
\vspace{0.5cm}
\ifarxiv \small \fi 
\begin{tabular}{l|c|c|c}
 & $\Gamma$ & $|B|>1$ & $|B|=1$  \\
\hline
\multirow{5}{*}{\makecell{Randomized test families \\ (\emph{Continuous or discrete} \\ \emph{distributions})}}     & $\alla$ & \makecell{If~\eqref{eq:C1} holds: \\ $\delta$ is $\alla$-admissible \\ $\Leftrightarrow$ \\ $\delta=\delta_E$ for $E$ sharp \\ and $E\leq \sup_b L_b(0,1)$}  & \multirow{14}{*}{\makecell{$\delta$ is $\alla$-admissible \\ $\Leftrightarrow$ \\ $\delta$ is $\consta$-admissible \\ $\Leftrightarrow$ \\ $\delta$ is the Neyman-Pearson \\ (likelihood-ratio) test}}  \\
\cline{2-3} 
& $\consta$ & \makecell{$\delta$ is $\consta$-admissible \\ ${\color{red}\Rightarrow}$ \\ 
$\delta=\delta_E$ for $E$ sharp,
increasing in $\lr$,\\ and compatible with $\calL$ } & \\
\cline{1-3}
\multirow{5}{*}{\makecell{Binary test families \\ (\emph{Continuous distributions)}}} & $\alla$ & \makecell{If \eqref{eq:C2} holds: \\ $\vp$ is $\alla$-admissible \\ $\Leftrightarrow$ \\ $\vp=\vp_E$ for $E$ sharp \\ and compatible with $\calL$} & \\ 
\cline{2-3} 
& $\consta$ & \makecell{$\vp$ is $\consta$-admissible \\ $\Leftrightarrow$ \\ $\vp=\vp_E$ for $E$ sharp,   increasing in $\lr$,\\ and compatible with $\calL$} & \\
\hline
\end{tabular}
\caption{An overview of our admissibility results and how they relate to the traditional Neyman-Pearson paradigm ($|B|=1$). 
$E$ always refers to an e-variable above. Note the one-way implication (highlighted in red) for $\consta$-admissibility for randomized test families, which differs from the biconditionals (if and only if) in the other cases. Recall that $E$ is sharp if $\E_P[E]=1$ and $E$ is compatible with $\calL$ if $E \in \calL\cup\{0\}$ $P$-almost surely (Definition~\ref{def:compatible}).}
\label{tab:admissibility}
\end{table}

\subsection{$\alla$-admissibility}
\label{sec:M-admissibility}

The first natural set of adversaries $\Gamma$ to consider in the definition of $\Gamma$-admissibility is all mappings from the data to the losses; i.e., $\Gamma = \alla$. This section classifies the set of $\alla$-admissible test families and binary test families for a point null $P$ versus a point alternative $Q$. For randomized test families we will let $P$ and $Q$ to be either discrete or continuous. As usual, for binary test families we assume that $P$ and $Q$ are continuous. 

When discussing randomized test families we will impose the following conditions on the losses, which restricts the behavior the ratio of type-II loss to type-I loss in the limit. In particular, we will assume that 
\begin{equation}
\label{eq:C1}
    \liminf_{b\to \sup(\calB)} \frac{L_b(1,0)}{L_b(0,1)} = 0. \tag{C1}
\end{equation}
Condition~\eqref{eq:C1} implies that an adversary can make type-II loss arbitrarily unimportant compared to type-I loss. 
In other words, as the cost of false positives increases ($L_b(0,1)$ grows) we care more about false positives than false negatives.

Intuitively, given two test families $\phi$ and $\delta$, if condition~\eqref{eq:C1} holds then an adversary $B$ can make $\phi$ arbitrarily small on regions where $\phi > \delta$, thus ensuring a lower type-II risk for $\delta$. Without such a condition, the set of admissible rules appear to depend in unsatisfactory ways on the precise null and alternative being tested. Condition~\eqref{eq:C1} can be replaced by the assumption that $\liminf_{b\to\inf(\calB)} L_b(1,0)/L_b(0,1) = 0$, but this requires that $L_b(1,0) \approx 0$ for small $b$, which is unsatisfying for practical applications. If the type-II loss is zero, then one should always accept the null. 

With that, we state the main result for $\alla$-admissibility of randomized test families. Recall that $\delta$ is canonical if it is written as in~\eqref{eq:canonical}.

\begin{theorem}[$\alla$-admissibility, randomized tests]
\label{thm:M-admissibility}
Let $\delta$ be a randomized test family and suppose condition~\eqref{eq:C1} holds. Then  $\delta$ is $\alla$-admissible if and only if $E_\delta$ is sharp and $\delta$ is canonical.  Furthermore, given an e-variable $E$ for $P$, $\delta_E$ is $\alla$-admissible if and only if $E$ is sharp and $E \leq \sup_bL_b(0,1)$ $P$-almost surely. 
\end{theorem}

The proof is in Appendix~\ref{sec:proof-thm-M-admissibility}. The idea for the first part of the theorem is straightforward. Given two type-I risk safe test families $\delta$ and $\phi$, we split $\calX$ into regions depending on whether $E_\delta$ is greater than $E_\phi$ or vice versa. When $E_\delta(X)> E_\phi(X)$, we can pick the adversary $B(X)$ such that $\phi(X,B) <  \delta(X,B)$, ensuring that $\delta$ has smaller type-II risk in such regions. When $E_\phi(X)>E_\delta(X)$, we choose $B(X) \to \sup(\calB)$ which drives the contribution of $\phi(X,B)$ on such regions to 0.

\begin{remark}
\label{rem:M-admissibility-NP-recovery}
    The invocation of Condition~\eqref{eq:C1} in Theorem~\ref{thm:M-admissibility} prohibits us from translating the results into the standard Neyman-Pearson setting. That is, Observation~\ref{obs:generalization-of-np} does not apply. Indeed, for a singleton $\calB = \{b\}$, \eqref{eq:C1} implies that $L_b(1,0) = 0$, so type-II risk loses its meaning and does not translate into the power of a test in the traditional paradigm.  
\end{remark}

An immediate consequence of Theorem~\ref{thm:M-admissibility} is that any $\alla$-admissible test family can be written as the test family corresponding to some
e-variable, highlighting their central role in post-hoc testing:
\begin{corollary}[$\alla$-admissibility, randomized tests]
\label{cor:M-admissibility-cor}
Suppose condition~\eqref{eq:C1} holds. 
    Then a randomized test family $\delta$ is $\alla$-admissible if and only if $\delta = \delta_E$ for some sharp e-variable $E$ such that $E \leq \sup_b L_b(0,1)$ $P$-almost surely. 
\end{corollary}

Intriguingly, Theorem~\ref{thm:M-admissibility} and Corollary~\ref{cor:M-admissibility-cor} recover the result of \citet{grunwald2024beyond} up to some assumptions, even though he studied a weaker form of admissibility. In particular, Theorem 1 of \citet{grunwald2024beyond} essentially shows that any G-admissible rule is defined by a sharp e-variable. 
Roughly speaking, he considers a test family to be G-admissible if it is  not dominated with probability 1 by another test. (We use G-admissibility to refer to Gr\"{u}nwald.) 

Upon reflection, the convergence of our two results is unsurprising.  Given two sharp but distinct e-variables, each must be greater than the other in some region of $\calX$. Neither will dominate the other with probability 1, hence neither will the corresponding tests. Thus both are G-admissible. In our setting meanwhile, Condition~\eqref{eq:C1} ensures that we can choose the adversary to take advantage of the difference between the two e-values.

Let us now provide the analogue of Theorem~\ref{thm:M-admissibility} for binary test families. In this case we replace condition~\eqref{eq:C1} with an assumption on the density of the losses in the reals. In particular: 
\begin{equation}
\label{eq:C2}
\text{$\calL$ is contained in and is dense in $\Re_{\geq M}$ for some $M>0$.} \tag{C2}
\end{equation}
\begin{theorem}[$\alla$-admissibility, binary tests]
\label{thm:M-admissibility-binary}
    Let $P$ and $Q$ be continuous and suppose that~\eqref{eq:C2} holds.  
    A binary test family  $\vp$ is $\alla$-admissible if and only if $E_{\vp}$ is sharp and $\vp$ is canonical. Moreover, given an e-variable $E$ for $P$, $\vp_E$ is $\alla$-admissible if and only if $E$ is sharp and compatible with $\calL$. 
\end{theorem}

The proof is in Appendix~\ref{sec:proof-M-admissibility-binary}. It follows a similar logic to the proof of Theorem~\ref{thm:M-admissibility}. Like Condition~\eqref{eq:C1}, the density assumption in Theorem~\ref{thm:M-admissibility-binary} prohibits us from applying the result to the standard Neyman-Pearson paradigm. That is, the conclusion of Remark~\ref{rem:M-admissibility-NP-recovery} applies in this case as well. 

To end this section, we state the 
analogue of Corollary~\ref{cor:M-admissibility-cor} for binary test families. 

\begin{corollary}[$\alla$-admissibility, binary tests]
    Let $P$ and $Q$ be continuous and suppose that $\calL$ is contained in and is dense in $\Re_{\geq 1}$.  A binary test family $\vp$ is $\alla$-admissible if and only if $\vp = \vp_E$ for some sharp e-variable $E$ for $P$ which is compatible with $\calL$. 
\end{corollary}

\subsection{Improving tests via Rao–Blackwellization}
\label{sec:rao-blackwellization}
What kinds of e-variables are available to us when testing point nulls and alternatives? 
For the Neyman-Pearson style tests in Section~\ref{sec:np-tests}, the corresponding e-variable is easily seen to be $E^\np(X) = L_{b^*}(0,1)\ind{\lr(X) \geq \kappa(b^*)}$ for some fixed $b^*$. Another e-variable is of course the likelihood ratio itself which, in the binary case, defines the satisfyingly intuitive test family $\vp_\lr(X,b) = \ind{\lr(X) \geq L_b(0,1)}$.   

These examples might make us wonder whether all ``good'' e-variables for post-hoc hypothesis testing for point hypotheses are functions of the likelihood ratio. In general, of course, there are e-variables which are not functions  of $\lr$. Indeed,  $E=\d R/\d\P$ for any distribution $R$ is a (sharp) e-variable. However,  unless $R = Q$, such an e-variable is not log-optimal~\citep{shafer2021testing}. We might therefore expect that it can be improved as a post-hoc test. 

Here we show that Rao-Blackwellization~\citep{radhakrishna1945information,blackwell1947conditional} can weakly improve the type-II risk of any test family based on an e-variable, with respect to a particular set of adversaries. Recall that the Rao-Blackwell theorem says that the expected loss of an estimator can be weakly improved (with respect to a convex loss) by conditioning on a sufficient statistic. See~\citet[Theorem 9.42]{wasserman2004all} for a modern statement of the result. 

Let $E$ be an arbitrary e-variable for $P$ and let $T$ be any sufficient statistic for testing $P$ vs $Q$. More formally, we assume that testing $P$ vs $Q$ can be written as $H_0: \theta=\theta_0$ and $H_1:\theta = \theta_1$ and that $T$ is sufficient for $\theta$. Define 
\begin{equation}
\label{eq:rao-e-value}
    S_T(X) = \E_\P [ E(X) | \sigma(T(X))], 
\end{equation}
where $\sigma(T(X))$ is the $\sigma$-field defined by $T(X)$. 
Note that $S$ is an e-variable by the law of total expectation: $\E_P[S] = \E_P \E_P[E|\sigma(T)] = \E_P[E]\leq 1$. Further note that $S_T$ is \emph{observable}, meaning that it does not depend on the unknown parameter. This is by definition of sufficiency. 

The following result shows that $S_T$ defines a test family that is weakly-preferable to that defined by $E$. However, we must restrict ourselves to adversaries that are $\sigma(T)$-measurable. 

\begin{proposition}
    \label{prop:rao-blackwell}
    Let $T$ be a sufficient statistic for $P$ vs $Q$. 
    Given any e-variable $E$ for $P$, the test family $\delta_{S_T}$ for $S_T$ in~\eqref{eq:rao-e-value} is weakly-preferable to $\delta_E$ with respect to any $\Gamma\subset\Gamma_{\sigma(T)}$ where $\Gamma_{\sigma(T)}$ is the set of all adversaries which are $\sigma(T)$-measurable. The same statement holds for the binary test families $\vp_{S_T}$ and $\vp_E$.  
\end{proposition}
\begin{proof}
Let 
\[f(E,B) = \min\left\{L_{B}(1,0), \frac{L_B(1,0)}{L_B(0,1)}E\right\}.\]
Note that $f$ is concave in $E$. Therefore, by Jensen's inequality, if $B$ is $\sigma(T)$-measurable, we have 
\begin{equation*}
    \E_P[f(E,B)|\sigma(T)] \leq  f(\E_P[E|\sigma(T)], B) = f(S_T,B). 
\end{equation*}
Since $\lr$ is \emph{minimally} sufficient, it follows that $\lr$ is $\sigma(T)$ measurable (apply the Fisher-Neyman characterization to write $\lr$ as a function of $T$). 
Therefore, 
\begin{align*}
    \E_Q[f(S_T,B)] &= \E_P[\lr f(S_T,B)] \geq \E_P[\lr \E_P [f(E,B)|\sigma(T)]] \\ 
    &= \E_P[\E_P [ \lr f(E,B)|\sigma(T)]] = \E_P[\lr f(E,B)] = \E_Q[f(E,B)].
\end{align*}
We have thus shown that 
\begin{equation*}
\E_Q[L_B(1,0)\delta_{S_T}(X,B)] = \E_Q[f(S_T,B)] \geq \E_Q[f(E,B)] = \E_Q[L_B(1,0) \delta_E(X,B)],
\end{equation*}
thus showing that $\delta_{S_T}$ is weakly-preferable to $\delta_E$ with respect to all adversaries $B$ that are $\sigma(T)$-measurable, as desired. When considering binary test families, we instead take $f(E,B) = L_B(1,0) \ind{E \geq L_B(0,1)}$, which is also concave in $E$. The rest of the argument remains unchanged. 
\end{proof}

For $\consta$-admissibility, Lemma~\ref{lem:inc-in-lr} will show that, not only is $\delta_E$ a function of the likelihood ratio, it must be an \emph{increasing} function of the likelihood ratio.

\subsection{$\consta$-admissibility}
\label{sec:const_admissibility}

The results of Section~\ref{sec:M-admissibility} were, in some sense, disappointingly liberal. As remark~\ref{rem:M-admissibility-NP-recovery} points out, Condition~\eqref{eq:C1} ensures that Theorems~\ref{thm:M-admissibility} and~\ref{thm:M-admissibility-binary} cannot be mapped back to the standard hypothesis testing setting. However, even if this were not the case, the result itself is unsatisfying: There are many sharp e-variables that do not correspond to the likelihood ratio test. Ideally, we would like a result which recovers the Neyman-Pearson lemma when instantiated with a single loss function. This motivates considering a different set $\Gamma$ of adversaries. In this section we study $\consta$-admissibility, where we recall that $\consta = \{ B: B(X) = b\text{ for some } b\text{ and all }X\}$. 

It is worth revisiting Remark~\ref{rem:safety-vs-admissibility} at this point. We stress that considering a restricted family $\Gamma$ of adversaries in the definition of admissibility does not change our notion of type-I risk safety, which continues to provide a guarantee with respect to all adversaries. In other words, the set of type-I risk safe test families stays the same, but the subset which are $\consta$-admissible is smaller (or so it will turn out) than the set which are $\alla$-admissible. 

For test families which are permitted to randomize, we give two necessary conditions to be $\consta$-admissible. The first is that the test families must be increasing functions of the likelihood ratio. The second is that, when $P$ and $Q$ are continuous, the e-variable associated to the test must be compatible with $\calL$. For binary test families we give a complete characterization of all $\consta$-admissible rules in terms of e-variables. As in Section~\ref{sec:M-admissibility} we consider a point null $P$ and a point alternative $Q$.

Studying $\consta$-admissibility has an interesting consequence: The type-II losses do not matter. Indeed, instantiating Definition~\ref{def:Pi-admissible} when $\Gamma=\consta$, we see that $\delta$ is $\consta$-inadmissible if, 
there exists a $\phi$ such that for all $b\in\calB$, $\E_Q[L_b(1,\delta(X,b))] \leq  \E_\Q[L_b(1,\phi(X,b))]$, i.e., $\E_Q[\delta(X,b))] \leq \E_Q[\phi(X,b)]$ for all $b\in\calB$, 
with a strict inequality for some $b$. That is, the precise values of $L_b(1,0)$, assuming they are not zero, do not affect the admissibility of a test. 

In other words, subject to type-I risk safety, we are  looking for tests which maximize their power on each fixed $b$. This leads us to the counterfactual interpretation of $\consta$-admissibility of $\delta$: For every other test $\phi$, there exists a $b\in\calB$ such that had we set our significance level to be $L_b^{-1}(0,1)$ at the beginning of the experiment, we would not regret using $\delta$ instead of $\phi$ (in terms of power, i.e., $\E_Q[\delta(X,b)]$).

We begin by proving that any randomized $\consta$-admissible test must be increasing in the likelihood ratio. The result is, of course, satisfyingly intuitive---a larger likelihood ratio indicates indicates more evidence against the null. A test that rejects the null at $\lr(X)$ but fails to reject when $\lr(Y)>\lr(X)$ is doing so despite increased odds in favor of the alternative. That such a result does not hold for richer classes of adversaries $\Gamma$ is yet another reason to consider $\consta$. Indeed, tests which don't satisfy this property are flaunting the law of likelihood~\citep{royall2017statistical}.

\begin{lemma}
\label{lem:inc-in-lr}
Let $\delta$ be a randomized test family. If a randomized test family $\delta$ is $\consta$-admissible then $E_\delta$ is an increasing function of the likelihood ratio $P$-almost surely. Consequently, for each fixed $b\in \calB$, $\delta(X,b)$ is an increasing function of the likelihood ratio $P$-almost surely. 
\end{lemma}

The proof may be found in Appendix~\ref{sec:proof-inc-in-lr}. 
\ifarxiv The idea is simple: Any test which is not increasing in $\lr$ can be perturbed such that it becomes strictly better. The details are intricate, however, requiring a careful analysis and strictly more edge cases than is healthy. 
\fi 
This result mirrors a result by \citet{koning2024continuous} (see Theorem 1) with a similar flavor: he shows that e-variables which maximize a fixed expected concave utility under a (simple) alternative must also be increasing in the likelihood ratio. Connecting his result with Lemma~\ref{lem:inc-in-lr} may be an interesting avenue for future research.

Our second result shows that when $P$ and $Q$ are continuous,  Lemma~\ref{lem:compatible} holds for randomized tests as well as binary ones. The proof may be found in Appendix~\ref{sec:proof-compatible-inc}.

\begin{lemma}
\label{lem:compatible-inc}
Let $P$ and $Q$ be continuous. If a randomized test family $\delta$ is $\consta$-admissible then $E_\delta$ is compatible with $\calL$. 
\end{lemma}

It is worth noting that the properties given in Lemmas~\ref{lem:inc-in-lr} and \ref{lem:compatible-inc} are both on display for the likelihood ratio test. For continuous $P$ and $Q$, its associated e-value is $E^\np(X) = L_{b^*}(0,1)$ if $\lr(X) \geq \kappa(b^*)$ and 0 otherwise (see Section~\ref{sec:np-tests}), which is clearly compatible with $\calL$. However, in the discrete case, the e-value obeys $E^\np(X) = L_{b^*}(0,1)(1 + \gamma)\ind{\lr(X)\geq \kappa(b^*)}$, which may not belong to $\calL$. Hence, the assumption of continuity in Lemma~\ref{lem:compatible-inc} cannot be removed.  

We leave it as an open question whether the conditions specified by Lemmas~\ref{lem:inc-in-lr} and~\ref{lem:compatible-inc} (in addition, of course, to the canonical representation in Lemma~\ref{lem:representation}) are also sufficient to specify $\consta$-admissible tests. We expect the answer is no. We do, however, conjecture that in addition to our extensions of the likelihood ratio test, \emph{mixtures} of likelihood ratio tests are $\consta$-admissible, subject to compatibility.  More formally, if $\nptest_k$ is the post-hoc likelihood ratio test on loss $b_k$ (i.e., the test defined by~\eqref{eq:np-rule}) then we suspect that the test defined by the (sharp) e-variable
\begin{equation}
    E_{\text{mix}} = \frac{1}{K}\sum_{k=1}^K E_{\phi_k},
\end{equation}
is $\consta$-admissible, as long as $\frac{1}{K}\sum_{k\leq j} L_{b_k}(0,1) \in \calB$ for each $1\leq j\leq K$. Note that $\nptest_k$ is $\consta$-admissible by Proposition~\ref{prop:np-decision-rules}.

\subsubsection{Binary test families}
\label{sec:const_admissibility_binary}

We now consider $\consta$-admissible binary test families. In this case we can provide a complete characterization of admissibility. As usual for binary tests, throughout this section we assume that $P$ and $Q$ are continuous. 
We begin with the analogue of Lemma~\ref{lem:inc-in-lr} in the binary case. The proof is in Appendix~\ref{sec:proof-inc-in-lr-binary}. 

\begin{lemma}
\label{lem:inc-in-lr-binary}
Let $P$ and $Q$ be continuous and let $\vp$ be a binary test family. If $\vp$ is $\consta$-admissible then $E_{\vp}$ is increasing in the likelihood ratio $P$-almost surely. Consequently, for each fixed $b\in\calB$, $\vp(X,b)$ is an increasing function of the likelihood ratio $P$-almost surely. 
\end{lemma}

One might initially suspect that Lemma~\ref{lem:inc-in-lr-binary} follows from Lemma~\ref{lem:inc-in-lr}. In that proof, however, we tweaked the given randomized test slightly, and there was no guarantee that the resulting test family was binary (in fact it almost certainly would not be). While the spirit of the proof of Lemma~\ref{lem:inc-in-lr-binary} is the same, the mechanics are quite different and use the fact that we begin with a binary test. 

Recall from Lemma~\ref{lem:representation-binary} that we can represent any $\consta$-admissible binary test as $\vp(X,b) = \ind{E_{\vp}(X) \geq L_b(0,1)}$. At a fixed $b$, this allows us to compare two binary rules by comparing their e-values. Lemma~\ref{lem:inc-in-lr-binary} implies that there exists some threshold $t(b)$ such that $\vp(X,b) = \ind{\lr(X) \geq t(b)}$ and, moreover, $t$ is increasing in $b$ by Lemma~\ref{lem:decreasing_in_b}. This allows us to compare two tests at the same $X$ by comparing these thresholds. Moving forward, our analysis will make heavy use of these thresholds. Let us state this formally. The proof is in Appendix~\ref{proof:cor-decision-curve}.  

\begin{corollary}
\label{cor:decision-curve}
    Let $P$ and $Q$ be continuous and let $\vp$ be a $\consta$-admissible binary test family. Then there exists an increasing function, henceforth called a \emph{decision curve}, $t:\calB\to\Re_{\geq 0}\cup\{\infty\}$ such that 
    \begin{equation}
    \label{eq:decision-curve}
        \vp(X,b) = \ind{\lr(X) \geq t(b)}.
    \end{equation}
    We write the decision curve associated with  $\vp$ as $t_\vp$.  
\end{corollary}

\begin{remark}
\label{rem:decision-curve-uniqueness}
There can be multiple decision curves for a given test, but they can differ only outside the range of $\lr$ under $P$ and $Q$. That is, if $t$ and $\ell$ are two decision curves for $\vp$, then $P(t(b)\leq \lr(X)<\ell(b)) = 0$ for all $b$, otherwise $\vp$ is not a well-defined function. What happens outside the support of $P$ and $Q$ does not affect admissibility, so going forward we will refer to the unique decision curve $t_\vp$ of a $\consta$-admissible test $\vp$.  
\end{remark}

As an example, if $\vp$ is the binary NP-type decision rule defined by~\eqref{eq:np-rule-binary} in Section~\ref{sec:np-tests}, then $t_\vp(b)=\infty$ for $b>b^*$ and $t_\vp(b) = \kappa(b^*)$ for $b\leq b^*$.

A test can be shown to be better than a second if the decision curve of the first is never above that of the second and is sometimes strictly lower. More formally, 
we say $t_\phi$ \emph{$Q$-strictly dominates} $t_\vp$ if $t_\phi(b) \leq t_\vp(b)$ for all $b\in\calB$ and there exists some $b^*$ such that $t_\phi(b^*) < t_\vp(b^*)$ and $Q(t_\phi(b^*) \leq \lr(X) < t_\vp(b^*)) > 0$. 

It turns out that we can characterize the $\consta$-admissible by appealing to decision curves. Lemma~\ref{lem:admissibility_by_t} in the appendix does precisely this. But such a classification is not practically useful. To move beyond decision curves and state $\consta$-admissibility in terms of e-variables, we require the following technical definition.

\begin{definition}
\label{def:minimal_decision_curve}
We say that a binary test family $\vp$ has a \emph{minimal decision curve} if for any $b^*\in\calB$, whenever $\sup_{b<b^*}L_b(0,1) = L_{b^*}(0,1)$, we either have $\sup_{b<b^*} t_\vp(b) = t_\vp(b^*)$ or $P(\sup_{b<b^*} t_\vp(b) \leq \lr(X) < t_\vp(b^*))=0$. 
\end{definition}

Definition~\ref{def:minimal_decision_curve} is vacuous whenever $\calL$ is discrete, meaning for every $L_b(0,1)\in\calL$ there exists some $\eps>0$ such that $(L_b(0,1) - \eps,L_b(0,1)+\eps)$ does not contain any other element of $\calL$. It is  relevant only for $\calL$ that is dense (or has a subset which is dense). 
When $P$ puts mass everywhere, the definition is equivalent to requiring that the decision curve be left continuous whenever the type-I losses are left continuous. 

Lemma~\ref{lem:minimal_decision_curve} shows that a $\consta$-admissible test must have a minimal decision curve. 
The intuition is as follows. Suppose $\sup_{b<b^*}L_b(0,1) = L_{b^*}(0,1)$ but $\sup_{b<b^*}t_\vp(b) < t_\vp(b^*)$. Then, for $X$ such that $\sup_{b<b^*}t_\vp(b) \leq \lr(X) \leq t_\vp(b^*)$, $\sup_b L_b(0,1) \ind{\lr(X) \geq t_\vp(b)} = L_{b^*}(0,1)$. 
However, $\vp(X,b^*) = 0$ for such $X$. That is, $\vp$ is suffering a type-I loss of $L_{b^*}(0,1)$ on such $X$ but not receiving the benefits of rejecting on loss $b^*$. In this case, one can define a second test family which rejects on $b^*$ and suffers the same type-I loss, thus making $\vp$ $\consta$-inadmissible.

We can now state the main result of this section. 
The proof is in Appendix~\ref*{sec:proof-const-admissibility-binary}.

\begin{theorem}
    \label{thm:const-admissibility-binary}
    Let $P$ and $Q$ be continuous.  
    A canonical binary test family $\vp$ is $\consta$-admissible if and only if it has a minimal decision curve and $E_{\vp}$ is sharp and an increasing function of the likelihood ratio. Moreover, given an e-variable $E$ for $P$, $\vp_E$ is $\consta$-admissible if and only if $E$ is sharp, compatible, and increasing in the likelihood ratio. 
\end{theorem}

\ifarxiv Note that there is no mention of minimal decision curves in the second half of the theorem. This is because $\vp_E$ will automatically have a minimal decision curve if $E$ is an increasing function of the likelihood ratio. Indeed, if $E(X) = h(\lr(X))$ for some increasing function $h$ with generalized inverse $h^-(y) = \inf\{z: h(z) \geq y\}$, then $\vp_E(X,b) = \ind{E(X) \geq L_b(0,1)} = \ind{\lr(X) \geq h^-(L_b(0,1))}$, implying that the decision curve of $\vp_E$ is $t_\vp(b) = h^-(L_b(0,1))$ (see also the proof of Corollary~\ref{cor:decision-curve}) which is minimal since $h^-$ is left-continuous. 
\fi

As was done in Section~\ref{sec:M-admissibility}, it is worth stating the following immediate corollary of Theorem~\ref{thm:const-admissibility-binary} which highlights the direct correspondence between admissible tests and e-variables.   

\begin{corollary}
\label{cor:const-admissibility-binary}
    Let $P$ and $Q$ be continuous. A binary test family $\vp$ is $\consta$-admissible if and only if $\vp = \vp_E$ for some sharp e-variable $E$ for $P$ which is compatible with $\calL$ and increasing in the likelihood ratio.
\end{corollary}

Theorem~\ref{thm:const-admissibility-binary} makes no assumptions on the losses, hence can recover the Neyman-Pearson lemma for continuous distributions. Indeed, as we discussed previously, for a single loss $L(0,1)=1/\alpha$, the e-variable associated to the likelihood ratio $\nptest(X) = \ind{\lr(X) \geq \kappa}$ is $E^\np(X) = \alpha^{-1}\ind{\lr(X) \geq \kappa}$ and is sharp, increasing in $\lr$, and compatible by construction. We can also see that $\phi^\np$ is the unique test which satisfies Corollary~\ref{cor:const-admissibility-binary}.  Consider any other $\consta$-admissible rule, which can be written as $\vp_E(X) = \ind{E(X)\geq 1/\alpha}$ for some $E$. Since $E$ is compatible with $\calL= \{1/\alpha\}$ and $\E_P[E]=1$, we have $P(E(X) = 1/\alpha) = \alpha$. Additionally, $E = h(\lr)$ for some increasing $h$ by Lemma~\ref{lem:inc-in-lr-binary}, so $\vp_E(X) = \ind{\lr(X) \geq h^-(1/\alpha)}$. Therefore $\alpha = P(E(X)\geq 1/\alpha) = P( \lr(X) \geq h^-(1/\alpha))$, meaning that $h^-(1/\alpha)= \kappa$ and $\vp_E = \nptest$.  

Now consider $L_b(0,1)=b$ for all $b>0$. In this case the test 
\begin{equation}
\label{eq:lr-test-binary}
    \vp_\lr(X,b) = \ind{\lr(X) \geq L_b(0,1)},
\end{equation}
is $\consta$-admissible. Indeed, $\lr$ is sharp, compatible with $\Re_{\geq 0}$ and obviously increasing in itself. 
This test is satisfying in that it accords well with calls to report the likelihood ratio instead of the p-value~\citep{royall1986effect,royall2000probability,berger1987testing,perneger2001sifting}.  
This nicely bridges the Fisherian and Neyman-Pearson perspectives on hypothesis testing: If a study summarizes its findings via $\lr(X)$ as a measure of evidence, downstream analysts can reject at whatever level they like, maintaining type-I risk safety.

\section{Summary}
\label{sec:summary}

We studied post-hoc hypothesis testing, a framework which allows the significance level to be chosen as a function of the data. The notions of type-I and type-II error are replaced by type-I and type-II risk, which take the form of expectations over data-dependent type-I and type-II loss functions. 
In order to generalize the notion of a uniformly most powerful test to this setting, we introduced the notion of $\Gamma$-admissibility,   where $\Gamma$ is a set of functions mapping the data to a type-II loss function. 

Different families $\Gamma$ result in different classes of admissible tests. For point hypotheses, when $\Gamma$ is all mappings from the data to losses we show that all admissible tests correspond to sharp e-variables. When $\Gamma$ is the set of constant mappings and the tests are binary, we give a classification of all admissible tests which recovers the Neyman-Pearson lemma~\citep{neyman1933ix}. 

This work suggests several open questions. For one, in the case of randomized tests, the set of necessary and sufficient conditions for $\consta$-admissibility do not match. Can we close the gap and give a classification of all admissible tests in this case? As discussed in Section~\ref{sec:const_admissibility}, we expect that mixtures of likelihood ratio tests are admissible but we have been unable to prove it. Second, while we have focused mainly on point hypotheses, hypothesis testing in practice typically involves composite hypotheses. 
While some of our results extend straightforwardly to composite settings,\footnote{In particular, Lemmas~\ref{lem:phi>delta}, \ref{lem:decreasing_in_b}, and~\ref{lem:Ephi>Edelta} can be extended to composite settings if handled correctly. For instance, a composite version of Lemma~\ref{lem:phi>delta} would posit that $\phi(X,b) \geq \delta(X,b)$ $P$-almost surely for all $P\in\calP$ and that there exists some $Q\in\calQ$ such that $Q(\phi(X,b)>\delta(X,b))>0$. The representations in Lemmas~\ref{lem:representation} and \ref{lem:representation-binary} likewise hold in composite settings.} can we obtain composite versions of our main classification theorems?
Third, we have adopted the typical hypothesis testing setup of two actions: either sustaining or rejecting the null. \citet{grunwald2024beyond}, however, studies multiple actions. How must our results be modified in this case?  

Overall, post-hoc hypothesis testing is in its nascency. We hope our work provides a foundation that will inspire further research. The connections we have established between admissible tests and e-variables connect post-hoc hypothesis testing to exciting developments in modern statistics, and our generalization of classical results suggest that post-hoc testing may serve as a principled bridge between traditional hypothesis testing and more adaptive approaches to statistical inference. While our work here has been primarily theoretical, we also hope that post-hoc hypothesis testing may offer solutions to some of the problems plaguing various empirical fields as they grapple with issues of statistical validity and reproducibility. 

\ifarxiv 
\subsection*{Acknowledgments}
We thank the anonymous referees for helpful comments which improved the manuscript and 
Ruodo Wang for pointers towards measurable selection theorems. BC was supported in part by the CWI internship program and in part by NSERC-PGSD, grant no.\ 567944. BC and AR acknowledge support from NSF grants IIS-2229881 and DMS-2310718. 
PG was co-funded by the European Union (ERC Advanced Grant FLEX 101142168). 
\fi

\bibliographystyle{plainnat}
\bibliography{main}

\newpage
\appendix

\section{Auxiliary results and omitted proofs}
\label{sec:proofs}

\subsection{Technical Lemmas}

\begin{lemma}
\label{lem:QA>QB}
    Suppose there exist sets $A_1$, $A_2$ such that  $\sup_{X_1\in A_1}\lr(X_1)<\inf_{X_2\in A_2}\lr(X_2)$. Then both of the following hold: (a)
    if $P(A_1)\leq P(A_2)$ then $Q(A_1)<Q(A_2)$, and (b) $Q(A_1)/Q(A_2)< P(A_1)/P(A_2)$. 
\end{lemma}
\begin{proof}
    Define  
    \[\overline{c} = \sup_{X\in A_1} \frac{\d\Q}{\d\P}(X), \text{~ and ~} \underline{c} = \inf_{X\in A_2} \frac{\d\Q}{\d\P}(X). \] 
    Observe that 
    \begin{equation*}
        \Q(A_1) = \int_{A_1} \d\Q(x) = \int_{A_1} \lr(x) \d\P(x)  \leq \overline{c} \int_{A_1} \d\P(x) = \overline{c}\P(A_1).
    \end{equation*}
    Similarly, 
    \begin{equation*}
        Q(A_2) = \int_{A_2} \lr(x) \d\P \geq \underline{c} \int_{A_2} \d P =\underline{c} P(A_2). 
    \end{equation*}
    To prove (b), use the above two displays and note that $\overline{c}<\underline{c}$. Then,  by assumption, we have 
    \[\frac{Q(A_1)}{Q(A_2)}\leq \frac{\overline{c}P(A_1)}{Q(A_2)} \leq \frac{\overline{c}P(A_1)}{\underline{c}P(A_2)} < \frac{P(A_1)}{P(A_2)}.\] 
    As for (a), suppose that $P(A_1)\leq P(A_2)$. Then, 
    \begin{equation*}
    \label{eq:pf-increasing-lr-Q}
        \Q(A_1) \leq \overline{c} \P(A_1) \leq  \overline{c}\P(A_2) \leq \underline{c} \P(A_2) \leq \Q(A_2), 
    \end{equation*}
    as desired. 
\end{proof}

\begin{lemma}
    \label{lem:test-evalue-inc}
    Let $\delta$ be $\Gamma$-admissible for any $\Gamma\supset \consta$. Then $\delta(\cdot,b)$ is increasing in the likelihood ratio for each $b\in\calB$ if and only if $E_\delta$ is increasing in the likelihood ratio. The same statement holds if $\delta=\vp$ is a binary test family. 
\end{lemma}
\begin{proof}
    Suppose $\lr(X_1)\leq \lr(X_2)$. If $E_\delta(X_1)\leq E_\delta(X_2)$ then $\min\{1, E_\delta(X_1)/L_b(0,1)\} \leq \min\{1, E_\delta(X_2)/L_b(0,1)\}$ for each $b$, so applying Lemma~\ref{lem:representation} demonstrates that $\delta(X_1,b)\leq \delta(X_2,b)$. Conversely, suppose that $E_\delta(X_1) > E_\delta(X_2)$. Consider any $b$ such that $L_b(0,1) \geq E_\delta(X_1)$ (this must exist, otherwise $E_\delta(X_1) > \sup_b L_b(0,1)$, which is impossible by definition). For such $b$, $ \min\{L_b(0,1), E_\delta(X_1)\} = E_\delta(X_1) > E_\delta(X_2) = \min\{ L_b(0,1), E_\delta(X_2)\}$, implying that $\phi(X_1,b) > \phi(X_2,b)$. This proves the result in the randomized case. The argument in the binary case is proved similarly. 
\end{proof}

\begin{lemma}
\label{lem:E>=Lb}
Let $P,Q$ be continuous and let $\delta$ be $\Gamma_\const$-admissible. 
Let $Y\subset\calX$ satisfy $Q(Y)>0$ and $E_\delta(X) >0$ for all $X\in Y$. 
Then there exists an adversary $B^*:\calX\to\calB$ such that $E_\delta(X) \geq L_{B^*(X)}(0,1)$ for all $X\in Y$ $Q$-almost surely. 
\end{lemma}
\begin{proof}
Suppose not. Then 
\[
\sup_{X\in Y}E_\delta(X) <\inf_{b\in \calB} L_b(0,1).
\]
Since $E_\delta>0$ on $Y$, we can find some $\gamma>0$ and $A\subset Y$ with $Q(A)>0$ such that
\[
\gamma\leq \inf_{X\in A} E_\delta(X). 
\]
Let $A = A_1\cup A_2$ where $P(A_1) = P(A_2)$ and $\lr(X_1)<\lr(X_2)$ for all $X_1\in A_1$ and $X_2\in A_2$. (This is possible by continuity.) 
Define a new decision rule $\deltahat$ such that, for all $b\in\calB$, $\deltahat(X,b)= \delta(X,b) - \eps/ L_b(0,1)$ for all $X\in A_1$ and $\deltahat(X,b) = \delta(X,b) + \eps/ L_b(0,1)$ for all $X\in A_2$. For $X\notin A$, $\deltahat$ is the same as $\delta$. Note that by Lemma~\ref{lem:representation}, $
\delta(X, b) = E_\delta(X) / L_b(0,1) \in (0,1)$ for $X\in A$ and $b\in\calB$. Therefore, 
\[
\deltahat(X, b) = \frac{E_\delta(X) - \eps}{L_b(0,1)} \;\text{ for }X\in A_1 \text{~ and ~} \deltahat(X, b) = \frac{E_\delta(X) + \eps}{L_b(0,1)}\;\text{ for } X\in A_2.
\]
Therefore, we may find $\eps$ small enough such that $0<\deltahat(X,b)\leq 1$ for all $X\in A$. Indeed, to ensure that $\deltahat(X,b) > 0$ we need to ensure that $\eps < \gamma$ and for $\deltahat(X,b) \leq 1$ we need $\eps < \inf_{b\in \calB} L_b(0,1) - \sup_{X\in A} E_\delta(X)$, where the final quantity is greater than zero by assumption. Fix any such $\eps>0$. 

To see that $\deltahat$ is type-I risk safe, write 
\begin{align*}
    \int_{A} E_{\deltahat}(x) \d\P &= \int_{A_1}\sup_b L_b(0,1)\left(\delta(x,b) - \frac{\eps}{L_b(0,1)}\right)\d\P \\
    &\qquad + \int_{A_1}\sup_b L_b(0,1)\left(\delta(x,b) + \frac{\eps}{L_b(0,1)}\right)\d\P \\ 
    &= \int_{A_1} E_\delta(x) \d\P - \eps \P(A_1) + \int_{A_2} E_\delta(x) \d\P + \eps \P(A_2)  \\ 
    &= \int_{A} \sup_b  L_b(0,1)\delta(x,b) \d\P,
\end{align*}
since $P(A_1) = P(A_2)$ by assumption. From this it follows that $\risk_P(\deltahat) \leq 1$ since 
\[
\E_P[E_{\deltahat}(X)] = \int_{A}E_{\deltahat}(x) \d\P + \int_{\calX\setminus A}E_{\delta}(x) \d\P = \int_{\calX} E_\delta(x) \d\P \leq 1.
\]
For type-II risk, fix any $b\in\calB$ and notice that 
\begin{align*}
    \int_{A}   \deltahat(x,b) \d\Q &= \int_{A_1\cup A_2}  \delta(x,b) \d Q- \frac{\eps}{L_b(0,1)} Q(A_1) + \frac{\eps}{L_b(0,1)} Q(A_2).
\end{align*}
Lemma~\ref{lem:QA>QB} implies that $Q(A_2)>Q(A_1)$ by our choice of $A_1$ and $A_2$. Therefore $-\eps Q(A_1) + \eps Q(A_2) > 0$ and we conclude 
    \begin{align*}
        \int_{A} \deltahat(x,b)\d\Q > \int_{A}  \delta(x,b)\d Q, 
    \end{align*}
implying that $\E_\Q [ L_B(1, \delta(X,B))]< \E_\Q[ L_B(1, \deltahat(X,B))]$. Hence $\delta$ is $\Gamma_\all$-inadmissible. 
\end{proof}

Recall that a decision curve $t_\phi$ \emph{$Q$-strictly dominates} $t_\vp$ if $t_\phi(b) \leq t_\vp(b)$ for all $b\in\calB$ and there exists some $b^*$ such that $t_\phi(b^*) < t_\vp(b^*)$ and $Q(t_\phi(b^*) \leq \lr(X) < t_\vp(b^*)) > 0$.

\begin{lemma}
\label{lem:admissibility_by_t}
Let $P$ and $Q$ be continuous and let $\vp$ be a type-I risk safe binary test family. Then $\vp$ is $\consta$-inadmissible if and only if there exists some binary type-I risk safe $\psi$ such that $t_\psi$ $Q$-strictly dominates $t_\vp$. 
\end{lemma}
\begin{proof}
   Let $\psi$ be a binary test family which is type-I risk safe and satisfies $t_\psi(b) \leq t_\vp(b)$ for all $b\in\calB$ and $t_\psi(b^*) < t_\vp(b^*)$ where $Q(t_{\psi}(b^*) < \lr(X) < t_\vp(b^*))>0$. Then for all $b\in\calB$, 
\[\vp(X,b) = \ind{\lr(X) \geq t_\vp(b)} \leq \ind{\lr(X) \geq t_\psi(b)} = \psi(X,b),\]
and the inequality is strict for $b^*$ on a set of positive measure under $Q$. Apply Lemma~\ref{lem:phi>delta} to see that $\vp$ is $\consta$-inadmissible. Conversely, suppose that $\vp$ is $\consta$-inadmissible, so there exists some $\psi$ satisfying $\E_Q[\psi(X,b)] \geq \E_Q[\vp(X,b)]$ for all $b$, with a strict inequality for at least one $b^*$. Using the representation~\eqref{eq:decision-curve} of both test families, this implies that 
\[Q(\lr(X) \geq t_\vp(b)) \leq Q(\lr(X) \geq t_\psi(b)),\]
for all $b$, so $t_\vp(b) \geq t_\psi(b)$. For $b^*$, $0< \E_Q[\ind{\lr(X) \geq t_\psi(b^*)} - \ind{\lr(X) \geq t_\vp(b^*)}] = Q(t_\psi(b^*) \leq \lr(X) < t_\vp(b^*))$, completing the proof. 
\end{proof}

\begin{lemma}
\label{lem:minimal_decision_curve}
    Let $P$ and $Q$ be continuous and let $\vp$ be a type-I risk safe binary test family. 
    If $\vp$ is $\consta$-admissible then it has a minimal decision curve. 
\end{lemma}
\begin{proof}
    We prove the contrapositive. Suppose that $\vp$ does have not have a minimal decision curve, so there exists some $b^*$
    such that $\sup_{b<b^*}t_\delta(b) < t_\delta(b^*)$ and $P(\sup_{b<b^*}t_\delta(b) \leq \lr(X) < t_\delta(b^*)) > 0$. 
    Define a new functions $\ell:\calB\to \Re_{\geq 0}\cup\{\infty\}$ as 
    \begin{equation*}
        \ell:b\mapsto \begin{cases}
            t_\delta(b),& b\neq b^*,\\
            \sup_{b<b^*}t_\delta(b),& b=b^*. 
        \end{cases}
    \end{equation*}
    Observe that $\ell$ $Q$-strictly dominates $t_\delta$. We claim that it also defines a type-I risk safe test family, which will show that $\delta$ is $\Gamma_\const$-inadmissible by Lemma~\ref{lem:admissibility_by_t}. 
    Note that since $t_\delta$ is increasing in $b$ by Corollary~\ref{cor:decision-curve}, so too is $\ell$.

    Partition $\calX$ into $C_1 = \{X : \lr(X) < \ell(b^*)\}$, $C_2 = G(b^*) = \{X: \ell(b^*)\leq \lr(X) < t_\delta(b^*)\}$, $C_3 = \{ X: \lr(X) \geq t_\delta(b^*)\}$. On $C_1$ and $C_3$, we have $\sup_b L_b(0,1)\ind{\lr(X) \geq t_\delta(b)} = \sup_bL_b(0,1)\ind{\lr(X) \geq \ell(b)}$ since, in both cases, the supremum does not involve $b^*$. Now, for $X\in C_2$, $\ind{\lr(X) \geq t_\delta(b)}=1$ if and only if $b<b^*$, so 
    \[\sup_bL_b(0,\vp(X,b)) =\sup_b L_b(0,1) \ind{\lr(X) \geq t_\delta(b)} = \sup_{b<b^*} L_b(0,1) = L_{b^*}(0,1).\]
    Moreover, again for $X\in C_2$, $\ind{\lr(X) \geq \ell(b)}=1$ if and only if $b\leq b^*$, giving that $\sup_b L_b(0,1) \ind{\lr(X) \geq \ell(b)} = \sup_{b\leq b^*}L_b(0,1) = L_{b^*}(0,1)$. Therefore, 
    \begin{equation*}
        \int_{C_2} \sup_b L_b(0,1)\ind{\lr(x) \geq \ell(b)}\d\P = \int_{C_2} \sup_b L_b(0,\vp(x,b))\d\P,
    \end{equation*}
    which, combined with equality on $C_1$ and $C_3$, implies that the risk of the test family defined by $\ell$ is equal to $\risk(\vp)$, completing the argument. 
\end{proof}

\subsection{Proof of Proposition~\ref{prop:np-decision-rules}}
\label{sec:proof-np-decision-rules}
First let us show that $\delta$ is type-I risk safe. Since $\delta(X,b) = 0$ on $\{\lr(X)<\kappa(b^*)\}$, 
\begin{align*}
    \risk_P(\delta) &= \int_{\{\lr(x) \geq  \kappa(b^*)\}} \sup_b L_b(0,\delta(x,b)) \d\P + \int_{\{\lr(x) < \kappa(b*)\}} \sup_b L_b(0,\delta(x,b)) \d\P \\ 
    &= \int_{\{\lr(x) > \kappa(b^*)\}} \sup_b L_b(0,1) \frac{L_{b^*}(0,1)}{L_b(0,1)} \d\P + \int_{\{\lr(x) = \kappa(b^*)\}} \sup_b L_b(0,1) \frac{L_{b^*}(0,1)}{L_b(0,1)} \gamma\d\P \\ 
    &= \int_{\{\lr(x) > \kappa(b^*)\}} L_{b^*}(0,1) \d\P + \int_{\{\lr(x) = \kappa(b^*)\}}  \gamma L_{b^*}(0,1) \d\P \\ 
    &= L_{b^*}(0,1)\E_\P[\nptest(X,b^*)] = 1. 
\end{align*}
Now we show it is admissible. 
Suppose there is a type-I risk safe test family $\phi$ (distinct from $\phi^\np$) which is weakly preferable to $\delta$ (with respect to $\Gamma$). 
That is, $\E_\Q [L_B(1,\delta(X,b))]\leq \E_\Q [L_B(1,\phi(X,b))]$ for all $B\in\Gamma$. If we consider the constant map $B(X) = b^*$, which is in $\Gamma$ by assumption, 
we obtain
$\E_\Q[\nptest(X,b^*)] \leq \E_\Q[\phi(X,b^*)]$. We claim this implies that $\phi(\cdot,b^*) = \nptest(\cdot,b^*)$ $P$-almost-surely. Indeed, this follows from the uniqueness in the Neyman-Pearson lemma, but let us prove it directly since it may not be a priori obvious that the same logic transfers to the post-hoc setting. Consider the integral 
\begin{equation*}
    I = \int_\calX (\underbrace{\nptest(x,b^*) - \phi(x,b^*)}_{:=t_1})(\underbrace{\d\Q - \kappa(b^*)\d\P}_{:=t_2}).
\end{equation*}
We claim that $I = 0$. We begin by showing that the integrand is nonnegative. This is clear if $\nptest(x,b^*) = \phi(x,b^*)$. If $\nptest(x,b^*)=1$ then $Q(x)/P(x) > \kappa(b^*)$ so both $t_1$ and $t_2$ are nonnegative. If $\nptest(x,b^*)=0$ then $Q(x)/P(x) < \kappa(b^*)$ so both terms are nonpositive. If $0<\nptest(x,b^*)<1$ then $Q(x) = \kappa P(x)$ so $t_1$ and $t_2$ multiply to zero. This shows that $I\geq 0$. Hence, 
\begin{align*}
    \int_\calX (\nptest(x,b^*) - \phi(x,b^*)) \d\Q & \geq \kappa(b^*) \int_\calX (\nptest(x,b^*) - \phi(x,b^*))\d\P \\ 
    &= \kappa(b^*)(L_{b^*}^{-1}(0,1) - \E_P[\phi(X,b^*)]. 
\end{align*}
Now, we must have $\E_P[\phi(X,b^*)] \leq L_{b^*}(0,1)$, otherwise $\risk_P(\phi) \geq \E_P[L_{b^*}(0,1)\phi(X,b^*)] >1$, contradicting that $\phi$ is type-I risk safe. The last term in the above display is thus nonnegative, and we have that 
\begin{equation*}
    \E_\Q[\nptest(X,b^*)] - \E_\Q[\phi(X,b^*)] = \int_\calX (\nptest(x,b^*) - \phi(x,b^*))\d\Q \geq 0. 
\end{equation*}
Since we know from above that $\E_\Q[\nptest(X,b^*)]\leq \E_\Q[\phi(X,b^*)]$, we have shown that $\E_\Q[\nptest(X,b^*)]=\E_\Q[\phi(X,b^*)]$. Finally, we can rewrite $I$ as $I = -\kappa(b^*) (L_{b^*}(0,1) - \E_\P[\phi(X,b^*)])$ which we now see must equal 0, since $I\geq 0$ but we argued above that $\E_\P[\phi(X,b^*)] \leq L_{b^*}(0,1)$. Since $I=0$ with a positive integrand, the integrand must be 0 almost surely. This implies that $\nptest(X,b^*) = \phi(X,b^*)$ except on the set $\{x: Q(x) = \kappa(b^*)P(x)\}$. If this set has positive measure under $Q$ (as it does in the discrete case) and $\nptest(X,b^*)\neq \phi(X,b^*)$, then one has higher power than the other, implying that $\E_\Q[\nptest(X,b^*)] \neq \E_\Q[\phi(X,b^*)]$, a contradiction. We conclude that $\phi(X,b^*) = \nptest(X,b^*)$ $P$-almost surely. 

Returning to the main proof, write 
\[E_\phi(X) = \sup_b L_b(0,\phi(X,b)) \geq L_{b^*}(0,\phi(X,b^*)) = L_{b^*}(0,\nptest(X,b^*)).\]
Hence $\E_P[E_\phi(X)]\geq \E_P[L_{b^*}(0,\nptest(X,b^*))] =1$ and we may conclude that 
\[E_\phi(X) = L_{b^*}(0,\nptest(X,b^*)),\] 
since $\phi$ is assumed to be type-I risk safe. By Lemma~\ref{lem:representation} we can write 
\begin{equation}
\label{eq:pf-np-decision-1}
    \phi(X,b) = \min\left\{1, \frac{E_\phi(X)}{L_b(0,1)}\right\} = \min\left\{1, \frac{L_{b^*}(0,1)}{L_b(0,1)}\nptest(X,b^*)\right\},
\end{equation}
as desired. Now we prove the second half of the proposition. Suppose $\delta$ is $\Gamma$-admissible and satisfies $\delta(X,b^*) = \nptest(X,b^*)$ for some $b^*$. We are tasked with showing that $\delta = \widehat{\phi}$ $P$-almost surely, where $\widehat{\phi}$ is defined as the right hand side of~\eqref{eq:pf-np-decision-1}. First we show that $\delta$ cannot be greater than $\widehat{\phi}$. Suppose that $\delta(X,b_0) > \widehat{\phi}(X,b_0)$ for all $X\in A\subset\calX$ and some $b_0\in\calB$ where $P(A)>0$. We will show that $\delta$ is not type-I risk safe. Observe that since $\widehat{\phi}(X,b_0)<1$ for $X\in A$ (since $1\geq \delta(X,b_0) > \widehat{\phi}(X,b_0)$), we have $\widehat{\phi}(X,b_0) = L_{b^*}(0,1) \nptest(X,b^*) / L_{b_0}(0,1)$. Consider an adversary $B$ defined as $B(X) = b^*$ for $X\in \calX\setminus A$ and $B(X) = b_0$ for $X\in A$. Then 
\begin{align*}
    \risk_P(\delta) &\geq \int_A L_{B(x)}(0,\delta(x,B(x))\d\P + \int_{\calX\setminus A} L_{B(x)}(0,\delta(x,B(x))\d\P \\ 
    & > \int_A L_{b_0}(0,1) \widehat{\phi}(X,b_0) \d\P + \int_{\calX\setminus \calA} L_{b^*}(0,1)\delta(x,b^*)\d\P \\ 
    &= \int_A L_{b^*}(0,1) \nptest(X,b^*)\d\P + \int_{\calX\setminus A} L_{b^*}(0,1) \nptest(X,b^*) \d\P \\ 
    &= L_{b^*}(0,1)\E_P[\nptest(X,b^*)] = 1,
\end{align*}
so $\delta$ is not type-I risk safe. We conclude that we must have $\delta(X,b) \leq \widehat{\phi}(X,b)$ $P$-almost surely for all $b\in\calB$. Now suppose that for some $b_0$, $\delta(X,b_0) < \widehat{\phi}(X,b_0)$ for all $X\in A$ where again $P(A)>0$. Consider defining a new test family $\deltahat$ such that $\deltahat = \delta$, except on $A$ and $b_0$ in which case we set $\deltahat(X,b_0) = \widehat{\phi}(X,b_0)$. Since we've already shown that $\delta \leq \widehat{\phi}$ $P$-almost surely, it follows that $\risk_P(\deltahat) \leq \risk_P(\delta)\leq 1$, so $\deltahat$ is type-I risk safe. Moreover, it has greater type-II risk than $\delta$. To see this, consider any adversary $B$ and let $A_0\subset A$ be the (possibly empty) set of $X$ such that $B(X) = b_0$. Then  
\begin{align*}
    \E_\Q[L_B(1,0) \deltahat(X,B)] &= \int_{A_0} L_{b_0}(1,0) \deltahat(x,b_0)\d\Q + \int_{\calX\setminus A_0} L_{B(x)}(1,0) \delta(x,B(x))\d\Q \\ 
    &\geq \int_{A_0} L_{b_0}(1,0) \deltahat(x,b_0)\d\Q + \int_{\calX\setminus A_0} L_{B(x)}(1,0) \delta(x,B(x))\d\Q \\ 
    &= \E_\Q[L_B(1,0) \delta(X,B))],
\end{align*}
hence $\E_\Q[L_B(1,\deltahat(X,B))] \leq \E_\Q[L_B(1,\delta(X,B))]$. 
If $A_0$ has positive measure under $Q$ then the inequality above becomes a strictly inequality, showing that $\deltahat$ sometimes has strictly lower type-II risk than $\delta$ (such an adversary exists in $\Gamma$ by assumption) making it strictly preferable to $\delta$ with respect to $\Gamma$, completing the proof.  

\subsection{Proof of Proposition~\ref{prop:binary-np-rules}}
\label{sec:proof-binary-np-rules}
The proof has similar mechanics to that of Proposition~\ref{prop:binary-np-rules}. 
    First note that $\vp$ is indeed type-I risk safe. By definition, $\vp(X,b) =0 $ whenever $\lr(X) < \kappa(b^*)$ and $\vp(X,b)$ iff $b\leq b^*$ otherwise. Therefore, 
    \begin{align*}
        \risk_P(\delta) &= \int_{\lr \geq \kappa(b^*)} \sup_bL_b(0,\vp(X,b)) \d\P + \int_{\lr(X) < \kappa(b^*)} \sup_b L_b(0,0)\d\P \\ 
        &= \int_{\lr \geq \kappa(b^*)} L_{b^*}(0,1)  \d\P = L_{b^*}(0,1) P(\lr(X) \geq \kappa(b^*) = 1,
    \end{align*}
    by definition of $\kappa(b^*)$. Now suppose that $\psi$ is any other type-I risk safe binary test family that is weakly preferable to $\vp$ with respect to any $\Gamma$. For the adversary $B(X) = b^*$ we have $\E_\Q[\psi(X,b^*)] \geq \E_\Q[\vp(X,b^*)] = \E_\Q[\nptest(X,b^*)]$. As in Appendix~\ref{sec:proof-np-decision-rules}, this implies that $\psi(X,b^*) = \nptest(X,b^*)$ $Q$-almost surely, which in turn implies that $E_{\psi}(X)= \sup_b L_b(0,\psi(X,b)) = L_{b^*}(0,1)\nptest(X,b^*)$ $P$-almost surely. Therefore, by the canonical representation given in Lemma~\ref{lem:representation-binary}, we have 
    \begin{equation}
        \psi(X,b) = \ind{E_{\psi}(X) \geq L_b(0,1)} = \ind{L_{b^*}(0,1) \nptest(X,b^*) \geq L_b(0,1)},
    \end{equation}
    $P$-almost surely, which is precisely $\vp(X,b)$. Next, suppose that $\vp$ is $\Gamma$-admissible and there exists some $b^*$ such that $\vp(X,b^*) = \nptest(X,b^*)$. Then we cannot have $\vp(X,b) > 0$ for any $b>b^*$ on any set $A$ of positive measure under $P$,   otherwise an adversary can ensure the type-I risk is higher than 1 by playing $b$ on $A$ and $b^*$ everywhere else. (We omit the precise calculation because it is similar to the proof in Appendix~\ref{sec:proof-np-decision-rules}). Similarly, we cannot have $\delta(X,b) < \nptest(X,b^*)$ for $b\leq b^*$ and $\lr(X) \geq \kappa(b^*)$ otherwise $\vp$ is $\Gamma$-inadmissible by Lemma~\ref{lem:phi>delta}. This completes the proof.

\subsection{Proof of Lemma~\ref{lem:phi>delta}}
\label{proof:lem-phi>delta}
By assumption we may assume that anything with positive probability under $P$ has positive probability under $Q$ and vice versa.   
For any $b\in\calB$, $L_b(1,\phi(X,b)) = L_b(1,0)(1 - \phi(X,b)) \leq L_b(1,0)(1 - \delta(X,b)) = L_b(1,\delta(X,b))$.
Therefore, for any $B\in\Gamma$, 
\[\E_\Q [L_B(1,\phi(X,B)] \leq \E_\Q [L_B(1,\delta(X,B))].\] 
Suppose $A\subset \calX$ is such that  $\phi(X,b^*) > \delta(X,b^*)$ for all $X\in A$ and $Q(A)>0$. 
Then 
$\phi$ has strictly lower type-II risk on the map $B(X) = b^*$ since
\[\E_Q [ L_{b^*}(1, \phi(X,b^*)\ind{X \in A} ] < \E_\Q[ L_{b^*}(1,\delta(X,b^*)\ind{ X \in A}],\] making $\phi$ strictly preferable to $\delta$ with respect to any $\Gamma$ ($B\in\Gamma$ by assumption). 

\subsection{Proof of Lemma~\ref{lem:decreasing_in_b}}
\label{proof:lem-decreasing_in_b}
Let $\delta$ be any test family, binary or otherwise, 
such that $\delta(X,b_1) < \delta(X,b_2)$ for some $b_1<b_2$ and all $X$ belonging to some set $A$ of positive measure on $Q$. Define a new test family $\deltahat$ such that $\deltahat(X,b_1) = \delta(X,b_2)$ for all $X\in A$. Otherwise $\deltahat$ is the same as $\delta$. Since $L_{b_1}(0,1)\deltahat(X,b_1) = L_{b_1}(0,1)\delta(X,b_2) \leq L_{b_2}(0,1)\delta(X,b_2)$, it follows that $\sup_b L_b(0,1) \deltahat(X,b) = \sup_b L_b(0,1)\delta(X,b)$. This implies that $\deltahat$ is type-I risk safe and we conclude that $\delta$ is $\Gamma$-inadmissible by Lemma~\ref{lem:phi>delta}.

\subsection{Proof of Lemma~\ref{lem:Ephi>Edelta}}
\label{sec:proof-Ephi>Edelta}
    If $E_\delta(X) < E_\phi(X)$ then there exists some $b$ such that $\delta(X,b) < \phi(X,b)$. If not, then $L_{b_0}(0,1)\phi(X,b_0) \leq L_{b_0}(0,1)\delta(X,b_0) \leq \sup_b L_b(0,1) \delta(X,b) = E_\delta(X)$ for all $b_0$, so taking the supremum over $b_0$ gives $E_\phi(X) \leq E_\delta(X)$, a contradiction. Now apply Lemma~\ref{lem:phi>delta} to see that $\delta$ is $\Gamma$-inadmissible for any $\Gamma$. As for the second part of the result, suppose that $\E_P[E_\delta]<1$. We will define a new test family $\deltahat$ such that $E_{\deltahat} > E_\delta$ 
    in some region $A$ while remaining type-I risk safe. First note that there exists some region $A\subset\calX$ with $P(A)>0$ and some $b_0$ such that $\delta(X,b_0)$ for all $X\in A$. If not, then $\delta(X,b)=1$ for all $b\in\calB$ $P$-almost surely. Since $\delta$ is type-I risk safe, we have  $\E_P[\sup_b L_b(0,\delta(X,b))] = \E_P[\sup_b L_b(0,1)] \leq 1$, hence $L_b(0,1) \leq 1$ for all $b$, contradicting our assumption on the losses. For $X\in A$, set 
    \begin{equation*}
        \deltahat(X,b_0) = \delta(X,b_0) + \frac{\eps}{L_{b_0}(0,1)},
    \end{equation*}
    for some $\eps>0$ to be determined. For $X\not\in A$ or any $b\neq b_0$, let $\deltahat(X,b) = \delta(X,b)$. 
    In order to ensure that $\deltahat(X,b_0) \leq 1$, we require that 
    \begin{equation}
    \label{eq:pf-Ephi-1}
        \frac{\eps}{L_{b_0}(0,1)} \leq 1 - \delta(X,b_0),
    \end{equation}
    for all $X\in A$. We can choose $\eps>0$ small enough such that~\eqref{eq:pf-Ephi-1} is satisfied since $\delta(X,b_0)<1$. 
    If $\deltahat$ is type-I risk safe then it is strongly preferable to $\delta$ for any $\Gamma$ by Lemma~\ref{lem:phi>delta}, thus rendering $\delta$ $\Gamma$-inadmissible. To see that $\eps$ can be chosen to make $\deltahat$ type-I risk safe, first write 
    \begin{align*}
        \int_A \sup_b L_b(0,\deltahat(x,b))\d\P &\leq \int_A \sup_b L_b(0,1)\left(\delta(x,b) + \frac{\eps}{L_b(0,1)}\right)\d\P \\
        &= \int_A \sup_bL_b(0,\delta(x,b))\d\P + \eps P(A).
    \end{align*}
    Therefore, 
    \begin{align*}
        \risk_P(\deltahat) &= \int_A \sup_bL_b(0,\deltahat(x,b))\d\P + \int_{\calX\setminus A} \sup_b L_b(0,\delta(x,b)) \d\P \\ 
        &\leq \int_\calX \sup_b L_b(0,\delta(x,b)) \d\P + \eps P(A)  
        = \risk_P(\delta) + \eps P(A).
    \end{align*}
    Since $\risk_P(\delta) = \E_P[E_\delta]<1$ by assumption, choosing any 
    \begin{equation}
    \label{eq:pf-Ephi-2}
        0<\eps \leq \frac{1 - \E_P[E_\delta]}{P(A)},
    \end{equation}
    ensures that $\deltahat$ remains type-I risk safe. Thus, choosing any $\eps$ that satisfies both~\eqref{eq:pf-Ephi-2} and~\eqref{eq:pf-Ephi-1} completes the proof when the test families are allowed to be randomized. Now let us consider the case of binary test families. Let $\vp$ be such that $\E_P[E_{\vp}]<1$. As before, this implies the existence of some $A\subset\calX$ and $b_0$ such that $\vp(X,b_0)=0$. Otherwise $\vp$ rejects with probability 1 everywhere and on all losses, implying that $L_b(0,1)\leq 1$ for all $b$. We define $\vphat$ which acts as $\delta$ everywhere except that $\vphat(X,b_0) = 1 > 0 = \vp(X,b_0)$ for all $X$ in some $A_0\subset A$ with $Q(A_0)>0$.  In particular, we choose $A_0$ to satisfy 
    \begin{equation}
    \label{eq:pf-Ephi-3}
        P(A_0) = \frac{ 1 - \E_P[E_\delta]}{L_{b_0}(0,1)},
    \end{equation}
    which is possible by continuity of $P$. Next, notice that by Lemma~\ref{lem:decreasing_in_b}, $\delta(X,b) = 0$ for all $b\geq b_0$ and $X\in A_0$, otherwise $\delta$ is already inadmissible. Therefore, $\sup_b L_b(0,\vphat(X,b)) = L_{b_0}(0,1)$ in $A_0$. 
    This gives
    \begin{align*}
        \risk_P(\deltahat) &= \int_{\calX\setminus A_0} \sup_b L_b(0,\vp(x,b)) \d\P + \int_{A_0} \sup_b L_b(0,\vphat(x,b))\d\P \\ 
        &= \int_{\calX\setminus A_0} \sup_b L_b(0,\vp(x,b)) \d\P + L_{b_0}(0,1)P(A_0) 
        \\
        &\leq \risk_P(\delta) + 1 - \E_P[E_\delta] =1,
    \end{align*}
    which proves that $\vphat$ is type-I risk safe. To complete the proof, apply Lemma~\ref{lem:phi>delta} to see that $\deltahat$ is strictly preferable to $\delta$ with respect to any $\Gamma$.

\subsection{Proof of Lemma~\ref{lem:representation}}
\label{proof:lem-representation}
Suppose there exists some $A$ satisfying $P(A),Q(A)>0$ such that for some $b\in\calB$ and all $X\in A$, $\delta(X,b) > \min\{1,E_\delta(X)/ L_b(0,1)\}$. That is, 
    \begin{equation}
    \label{eq:pf-constant-sup-1}
      L_b(0,1)\delta(X,b) > \min\{ L_b(0,1),E_\delta(X)\}, \text{~ for all } X\in A.  
    \end{equation}
    Observe that we must have $\min\{ L_b(0,1),E_\delta(X)\} = E_\delta(X)$. Otherwise \eqref{eq:pf-constant-sup-1} implies that $ L_b(0,1) \delta(X,b) >  L_b(0,1)$, which is impossible since 
    $\delta(X,b)\leq 1$ by definition. Therefore the minimum is achieved by $E_\delta(X)$ and we have $ L_b(0,1)\delta(X,b) > E_\delta(X)$, contradicting the definition of $E_\delta$. Thus no such $b$ can exist, and we conclude that the left hand side of~\eqref{eq:canonical} is at most the right hand side. 

    Next, suppose there exists some $b$ with $\delta(X,b) < \min\{1, E_\delta(x)/ L_b(0,1)\}$ for all $X\in A$, where again $P(A),Q(A)>0$. Consider a new decision family $\deltahat$ defined as $\deltahat(X,b) = \min\{1, E_\delta(X)/ L_b(0,1)\}$ for all $X\in A$ and which  acts as $\delta$ everywhere else. Note that for $X\in A$,  $ L_b(0,1)\deltahat(X,b) = \min\{ L_b(0,1), E_\delta(X)\}\leq E_\delta(X)$, hence $E_{\deltahat}(X) = \sup_b  L_b(0,1)\deltahat(X,b) = E_\delta(X)$ (using that $\deltahat(\cdot,c) = \delta(\cdot,c)$ for all $c\neq b$). Therefore, 
    \begin{equation}
        \risk_\P(\deltahat)=   \int_{\calX\setminus A} E_\delta(x)\d P + \int_A E_{\deltahat}(x)\d P= \risk_\P(\delta) \leq 1, 
    \end{equation}
    so $\deltahat$ remains type-I risk safe. Moreover, it strictly improves type-II risk by Lemma~\ref{lem:phi>delta} ($\deltahat(\cdot,b)> \delta(\cdot,b)$ on a set of positive measure under $\Q$), implying that $\delta$ is $\Gamma$-inadmissible.

\subsection{Proof of Lemma~\ref{lem:E<supb}}
\label{sec:proof-E<supb}
    Suppose there exists a set $A\subset\calX$ such that $E(X) > \sup_bL_b(0,1)$ for all $X\in A$ and $P(A)>0$. We will define a new e-variable $F$ for $P$ such that $F(X) \leq \sup_bL_b(0,1)$ for all $X\in A$ and $\delta_F$ is strictly preferable to $\delta$ with respect to any $\Gamma$. First we observe that there exists some $Y\subset\calX$ such that $E(X) < \sup_b L_b(0,1)$ for all $X\in Y$ and $P(Y)>0$. If not, then $E(X) \geq \sup_b L_b(0,1)$ $P$-almost surely, implying by definition (see~\eqref{eq:delta-from-e}) that $\delta_E$ rejects with probability 1 $P$-almost surely, contradicting our assumption that $L_b(0,1)>1$ for some $b$. For some fixed $\eps>0$ to be specified later, define $F$ as follows: 
    \begin{equation}
        F(X) = \begin{cases}
            \sup_b L_b(0,1), & \text{if } X\in A,\\
            E(X) + \eps,&\text{if } X\in Y, \\
            E(X),&\text{otherwise}.
        \end{cases}
    \end{equation}
    First notice that on $Y$ there exists some $b$ such that $\delta_F(X,b) = \min\{1, E_(X)/L_b(0,1)\} > E(X)/L_b(0,1) = \delta_E(X)$. Moreover, on $A$, $\delta_F(X) = 1 = \delta_E(X,b)$ for all $b$ since $E(X) > F(X) =\sup_bL_b(0,1)$. Finally, everywhere else we have $\delta_E(X,b) = \delta_F(X,b)$ for all $b$ since $F(X) = E(X)$. Therefore, it follows from Lemma~\ref{lem:phi>delta} that $\delta_E$ is $\Gamma$-inadmissible if $\delta_F$ is type-I risk safe. Compute 
    \begin{align*}
        \risk_P(\delta_F) &= \int_A F \d\P + \int_Y F \d\P + \int_{\calX\setminus( A\cup Y)} E \d\P \\ &= \int_A (\sup_bL_b(0,1) - E + E)\d\P + \int_E F \d\P + \eps P(Y) + \int_{\calX\setminus( A\cup Y)} E \d\P \\ 
        &= \int_\calX E \d\P + \int_A (\sup_b L_b(0,1) - E)\d\P + \eps P(Y).
    \end{align*}
    Taking 
    \begin{equation}
        \eps = \frac{1}{P(Y)}\int_A (\sup_bL_b(0,1) - E)\d\P,
    \end{equation}
    thus suffices to ensure that $\risk_P(\delta_F) = \risk_P(\delta_E)\leq 1$. We have thus proven that if $P(E(X)>\sup_b L_b(0,1))>0$ then $\delta$ is $\Gamma$-inadmissible. Then, if $P(E(X)>\sup_b L_b(0,1)) = 0$, $E_{\delta_E}=E$ $P$-almost surely by~\eqref{eq:EdeltaE_vs_E}.

\subsection{Proof of Lemma~\ref{lem:representation-binary}}
\label{proof:lem-representation-binary}
    The proof is similar to that of Lemma~\ref{lem:representation}. First note that we can't have $\vp(X,b) = 1$ if $L_b(0,1) > E_{\vp}(X)$ by definition of $E_{\vp}$. Therefore, $\vp(X,b) \leq  \ind{ L_b(0,1) \leq E_{\vp}(X)}$. Conversely, suppose that, for some $b^*$, $\vp(X,b^*) =0$ when $L_{b^*}(0,1) < E_{\vp}(X)$ for all $X\in A$ where $P(A), Q(A)>0$. Define a new rule $\vphat$ such that $\vphat(X,b^*) = 1$ for $X\in A$ and otherwise $\vphat$ is the same as $\vp$. Since $E_{\vp}= E_{\vphat}$, $\vphat$ is  type-I risk safe. Now apply Lemma~\ref{lem:phi>delta} with $\phi= \vphat$ to see that $\vp$ is $\Gamma$-inadmissible.

\subsection{Proof of Lemma~\ref{lem:compatible}}
\label{sec:proof-compatible}
    We prove the contrapositive. Suppose that $E$ is not compatible with $\calL_0$. Then there exists some set $Y\subset\calX$ with $P(Y)>0$ such that, for all $X\in Y$, either 
    \begin{enumerate}
        \item $L_{b_1}(0,1) < E(X) < L_{b_2}(0,1)$ for some $b_1<b_2$ such that there is no $b\in\calB$ with $b_1<b<b_2$, or 
        \item $E(X) > \sup_b L_b(0,1)$. Of course, in this case $\sup_bL_b(0,1)<\infty$. 
    \end{enumerate}
    We consider these two cases separately. First suppose that (1) holds. For some \emph{disjoint} subsets $Y_0,Y_1\subset Y$ to be determined, we define a new random variable as follows: 
    \begin{equation}
        F(X) = \begin{cases}
         L_{b_1}(0,1),&\text{if } X\in Y_0,\\ 
         L_{b_2}(0,1),&\text{if } X\in Y_1, \\ 
         E(X),&\text{otherwise}.
    \end{cases}   
    \end{equation}
    Notice that for all $X\in Y_0$ and all $b\in\calB$, $\vp_F(X,b) = \vp_E(X,b)$; the fact that $E(X) > L_{b_1}(0,1)$ makes no difference since it is strictly smaller than the next loss, $L_{b_2}(0,1)$ and thus cannot influence the value of $\vp_E$. Meanwhile, $\vp_F(X,b_2)  =1 > 0 = \vp_E(X,b_2)$ for $X\in Y_1$, and $\vp_F(X,b) = \vp_E(X,b)$ for all other $b$. For $X\notin Y_0\cup Y_1$, clearly $\vp_E(X,b) = \vp_F(X,b)$. Therefore, as long as $Y_1$ has positive measure under $P$ and $\vp_E$ is type-I risk safe, then $\vp_E$ is $\Gamma$-inadmissible by Lemma~\ref{lem:phi>delta}. Let us compute 
    \begin{align*}
        \risk_P(\delta_F) &= \int_{Y_0} F\d\P + \int_{Y_1}F \d\P + \int_{\calX \setminus (Y_0\cup Y_1)} F\d\P \\ 
        &= \int_{Y_0} (F-E+E)\d\P + \int_{Y_1}(F-E+E) \d\P + \int_{\calX \setminus (Y_0\cup Y_1)} E\d\P \\ 
        &= \int_\calX E \d\P + \int_{Y_0} (L_{b_1}(0,1) - E)\d\P + \int_{Y_1} (L_{b_2}(0,1) - E)\d\P \\ 
        &\leq 1 + \int_{Y_0}(L_{b_1}(0,1) - E)\d\P + (L_{b_2}(0,1) - L_{b_1}(0,1))P(Y_1).
    \end{align*}
    Thus, for a fixed $Y_0$ with positive measure, taking $Y_1$ small enough such that 
    \begin{equation}
    0<P(Y_1) = (L_{b_2}(0,1) - L_{b_1}(0,1))^{-1} \int_{Y_0} (E - L_{b_1}(0,1))\d\P, 
    \end{equation}
    ensures that $\risk_P(\delta_F)\leq 1$. Such choices of $Y_0$ and $Y_1$ are possible by continuity of $P$. This completes the proof of case (1). 

    Let us now consider case (2). Here the proof is similar to that of Lemma~\ref{lem:E<supb}, but it must be handled somewhat differently to account for binary tests. Since $E(X)>\sup_b L_b(0,1)$ on $Y$, there exists some $U$ such that $E(X) < \sup_b L_b(0,1)$ for all $X\in U$. The reasoning is the same as in previous proofs. 

    For some $U_0\subset U$, define  
    \begin{equation}
        F(X) = \begin{cases}
            \sup_b L_b(0,1),& \text{if } X \in Y \cup U_0, \\ 
            E(X),&\text{otherwise}.
        \end{cases}
    \end{equation}
    Following similar reasoning as above, $\vp_F(X,b)\geq \vp_E(X,b)$ with strict inequality for some $b$ when $x\in U_0$. We thus need only show that $\vp_F$ is type-I risk safe to complete the proof. Let $s = \sup_b L_b(0,1)$ and write 
    \begin{align*}
        \risk_P(\vp_F) &= \int_Y (s - E + E)\d\P + \int_{U_0} (s - E + E)\d\P + \int_{\calX\setminus(Y\cup U_0)} E\d\P \\ 
        &\leq  \int_\calX E\d\P + \int_Y (s - E)\d\P + \int_{U_0} s \d\P \\ 
        &\leq 1 + \int_Y (s - E)\d\P + sP(U_0). 
     \end{align*}
     Since $s<\infty$ and $P$ is continuous, we can choose $U_0$ such that $s<P(U_0) = \int_Y ( E - s) \d\P$, which gives $\risk_P(\vp_F)$, completing the argument.

\subsection{Proof of Theorem~\ref{thm:M-admissibility}}
\label{sec:proof-thm-M-admissibility} 
Lemma~\ref{lem:Ephi>Edelta} implies that if $E_\delta$ is not sharp then $\delta$ is $\Gamma_\all$-inadmissible. Meanwhile, Lemma~\ref{lem:representation} implies that if $\delta$ is not canonical then it is not $\Gamma_\all$-admissible. This proves one direction. 
To prove the other direction, suppose that $\delta$ is canonical and $E_\delta$ is sharp. Let $\phi$ be any other type-I risk safe test family. We may assume that $\phi$ is canonical and $E_\phi$ sharp, otherwise we may improve it and consider the resulting test family. 
  
  Define the following subsets of $\calX$. Let
\begin{enumerate}
    \item $R_1\subset\calX$ be the region where $E_\phi(X) = E_\delta(X)=0$,
    \item $R_2\subset\calX$ be the region where $0=E_\delta(X) < E_\phi(X)$,
    \item $R_3\subset \calX$ be the region where $0<E_\delta(X) \leq E_\phi(X)$,
    \item and $R_4\subset\calX$ be the region where $E_\delta(X) > E_\phi(X)$.
\end{enumerate}
Observe that $R_4$ cannot have $Q$-measure zero. Indeed, if so, then $E_\delta(X) \leq E_\phi(X)$ $Q$-almost surely. Since both $E_\delta$ and $E_\phi$ have the same expected value and both are nonnegative, this implies that $E_\phi(X) = E_\delta(X)$ $Q$-almost surely, implying by Lemma~\ref{lem:representation} that $\delta = \phi$ $Q$-almost surely.

It suffices to find a map $B$ such that $\E_\Q [L_B(1,\phi(X,B))] > \E_\Q [L_B(1,\delta(X,B))]$, thus demonstrating that $\phi$ cannot be strongly $\Gamma_\all$-preferable to $\delta$. To this end, write  
\begin{equation}
\E_\Q [L_B(1,\phi(X,B)) - L_B(1,\delta(X,B))] = \sum_{1\leq k\leq 4} \Delta_k, 
\end{equation}
where 
\begin{align*}
    \Delta_k &\equiv \int_{R_k} L_B(1,\phi(x,B)) - L_B(1,\delta(x,B)) \d\Q \\ 
    &= \int_{R_k} L_B(1,0)( \delta(x,B) - \phi(x,B)) \d Q.
\end{align*}
   Define the adversary $B$ as follows. 
   First consider $R_4$ which, as we argued above, has positive measure under $Q$. 
   For $X\in R_4$, we set $B(X)$ to be any $b$ such that $L_b(0,1)>E_\phi(X)$. Such a $b$ must exist otherwise, $E_\delta(X) > E_\phi(X) \geq \sup_b L_b(0,1)$, contradicting the definition of $E_\delta(X)$.  
   Therefore, on $R_4$, $\phi$ satisfies $\phi(X,B(X)) = E_\phi(X) / L_{B(X)}(0,1) < 1$, by Lemma~\ref{lem:representation}, so 
   \begin{equation*}
       \delta(X, B(X)) = \min\left\{1, \frac{E_\delta(X)}{L_{B(X)}(0,1)}\right\} > \frac{E_\phi(X)}{L_{B(X)}(0,1)} = \phi(X,B(X)), \quad Q\text{-almost surely}.
   \end{equation*}
   Therefore, 
   \begin{align*}
       \Delta_4 = \int_{R_4} L_{B}(1,0)\left(\frac{E_\delta(X)}{L_{B}(0,1)} - \frac{E_\phi(X)}{L_{B}(0,1)}\right)\d\Q > 0,
   \end{align*}
   Now let us choose how $B$ acts on the rest of $\calX$. For $X\in R_1$, the choice is arbitrary, since $\phi(X,b) = \delta(X,b) = 0$ so $\Delta_1=0$. For $R_2$, write 
   \begin{align*}
       \Delta_2 = -\int_{R_2} L_{B}(1,0) \min\left\{1, \frac{E_\phi(X)}{L_{B}(0,1)}\right\}\d\Q.
   \end{align*}
   For all $X\in R_2$, take $B(X) = b$ where $b$ is large enough such that 
   \begin{equation}
   \label{eq:pf-M-admissible-1}
     \frac{L_b(1,0)}{L_b(0,1)} \leq \frac{\Delta_4}{4 \int_{R_2} E_\phi(X)\d\Q} \text{~ and ~} L_b(0,1) \geq \sup_{X\in R_2} E_\phi(X).  
   \end{equation} 
   Such a $b$ must exist by condition~\eqref{eq:C1} and by definition of $E_\phi$. In this case, 
   \[\Delta_2 \geq - \int_{R_2} \frac{L_b(1,0)}{L_b(0,1)} E_\phi(X) \d\Q \geq - \frac{\Delta_4}{4}.\]
   For $R_3$ we make a similar argument. Write 
   \begin{equation*}
       \Delta_3 = - \int_{R_3} L_B(1,0)(\phi(x,B) - \delta(x,B))\d\Q \geq -\int_{R_3} L_B(1,0)\min\left\{1, \frac{E_\phi(X)}{L_B(0,1)}\right\}\d\Q.
   \end{equation*}
   Take $B(X) = b$ where $b$ defined as in~\eqref{eq:pf-M-admissible-1} but with $R_3$ in place of $R_2$. Then $\Delta_3 \geq -\Delta_4/4$. Therefore, overall we have $\sum_{k\leq 4} \Delta_k \geq -\Delta_4/4 - \Delta_4/4 + \Delta_4 >0$ since $\Delta_4>0$. This completes the first part of the argument. 

   For the second part, let $E$ be an e-variable for $P$. If $E\leq\sup_bL_b(0,1)$ $P$-almost surely, then $E_{\delta_E} = E$ $P$-almost surely by~\eqref{eq:EdeltaE_vs_E}. Thus if $E$ is sharp then $E_{\delta_E}$ is sharp, and we may apply the first part of the theorem to see that $\delta_E$ is $\Gamma_\all$-admissible. Conversely, if $\delta_E$ is $\Gamma_\all$-admissible then $E_{\delta_E}$ is sharp, again by applying the first part of the theorem. This in turn implies that $E$ is sharp, since $E_{\delta_E}(X) \leq E(X)$ by~\eqref{eq:EdeltaE_vs_E}.  
   Next we claim that $E = E_{\delta_E}$, which will complete the proof since $E_{\delta_E}(X)\leq \sup_bL_b(0,1)$ for all $X$ by construction. Let $W= \{X : E(X) \neq E_{\delta_E}(X)\}$ be the subset on which $E$ and $E_{\delta_E}$ disagree. By~\eqref{eq:EdeltaE_vs_E}, we know they disagree if and only if $E(X) > \sup_b L_b(0,1)$. Therefore, we can write 
   \[W = \{ X: E(X) > \sup_bL_b(0,1)\}.\]
   Since $E$ is sharp,
   \begin{align*}
       1 &= \E_P[E] = \E_P[E\ind{X\notin W}] + \E_P[E\ind{X\in W}] \\
       &= \E_P[E_{\delta_E}\ind{X\notin W}] + \E_P[E\ind{X\in W}] \\ 
       &> \E_P[E_{\delta_E}\ind{X\notin W}] + \sup_b L_b(0,1) P(W).
   \end{align*}
   Therefore, 
   \begin{align*}
       1 &= \E_P[E_{\delta_E}] = \E_P[E\ind{X\notin W}] + \E_P[E_{\delta_E}\ind{X\in W}] \\ 
       &< 1 - \sup_b L_b(0,1)P(W) \E_P[E_{\delta_E}\ind{X\in W}] \\ 
       &\leq 1 - \sup_b L_b(0,1)P(W) \sup_b L_b(0,1)P(W) = 1, 
   \end{align*}
   a contradiction. Hence we must have $P(W)=0$, meaning $E = E_{\delta_E}$ $P$-almost surely. This completes the proof.

\subsection{Proof of Theorem~\ref{thm:M-admissibility-binary}}
\label{sec:proof-M-admissibility-binary}
    The proof is similar to that of Theorem~\ref{thm:M-admissibility}, so we provide only a sketch. Lemma~\ref{lem:Ephi>Edelta} implies that $E_{\vp}$ must be sharp; Lemma~\ref{lem:representation-binary} implies that $\vp$ must be canonical. For the converse, we partition the sample space the same way as in the proof of Theorem~\ref{thm:M-admissibility}. On $R_4$ we have $E_\delta(X) > E_\phi(X)$. Because the losses are dense in $\Re\cap [M,\infty)$, we can find some $b$ such that $E_\phi(X) < L_b(0,1) \leq E_\delta(X)$. (Note that $E_\delta(X) \geq \inf b L_b(0,1)\geq M$ for binary $\delta$.) We take $B(X)$ to be this value of $b$. Then $\Delta_4 > 0$. For $R_1$, the choice of $B$ is again arbitrary. For $R_2$, choose $B(X)$ such that $L_{B(X)}(0,1) > E_\phi(X)$ which is possible $\sup_b L_b(0,1) = \infty$ ($E_\phi$ cannot be infinite on a set of positive measure under $P$ otherwise $\phi$ is not type-I risk safe). Then $\Delta_2 = 0$. A similar argument applies to $R_3$. This shows that $\E_Q[L_B(1,\phi(X,B))] - \E_Q[L_B(1,\delta(X,B))] = \sum_i \Delta_i >0$. 

    Now consider an e-variable $E$ for $P$. If $E$ is not sharp then neither is $E_{\delta_E}$ since $E_{\delta_E}(X) = \sup_b L_b(0,1)\ind{E(X) \geq L_b(0,1)} \leq E(X)$. Therefore $\delta_E$ is $\Gamma_\all$-inadmissible by the first part of the theorem. Further, if $E$ is not compatible then $\delta_E$ is $\Gamma_\all$-inadmissible by Lemma~\ref{lem:compatible}.  Now we prove the other direction. If $E$ is compatible, then $E_{\delta_E} = E$ $P$-almost surely. Therefore if $E$ is sharp so is $E_{\delta_E}$ so $\delta_E$ is $\Gamma_\all$-admissible, again by applying the first part of the theorem. 

\subsection{Proof of Lemma~\ref{lem:inc-in-lr}}
\label{sec:proof-inc-in-lr}
Suppose not, so there exists $A_1, A_2$ and $b$ such that 
\begin{equation*}
    \inf_{X_1\in A_1}E_\delta(X_1) > \sup_{X_2\in A_2} E_\delta(X_2),\text{~ and ~} \sup_{X_1\in A_1}\lr(X_1) < \sup_{X_2\in A_2} \lr(X_2). 
\end{equation*}
For convenience, define  
\begin{equation*}
    \gamma_1 = \inf_{X_1\in A_1}E_\delta(X_1), \quad \gamma_2 = \sup_{X_2\in A_2}E_\delta(X_2). 
\end{equation*}
We will define a new decision rule $\deltahat$ which is strictly preferable to $\delta$ with respect to $\Gamma_\const$. 
Pick any $\Delta\in\Re$ such that $0<\Delta<\gamma_1-\gamma_2$, which exists by the density of the reals. Define 
\begin{align*}
    \calB_0 &\equiv \{b:  L_b(0,1) \leq \gamma_2\}\\ 
    \calB_1 &\equiv  \{b:  \gamma_2< L_b(0,1) \leq \gamma_2+\Delta\} \\ 
    \calB_2 &\equiv \calB \setminus (\calB_0 \cup \calB_1).
\end{align*}
Notice that $\calB_0,\calB_1$, and $\calB_2$ are all disjoint and form a partition of $\calB$. 
It may be helpful to draw a plot with $L_b(0,1)$ on the y-axis and $\lr(X)$ on the x-axis. $\calB_0$ are those $b$ at the bottom, $\calB_2$ those at the top. 
Further notice that $\calB_2$ is non-empty. Otherwise 
\[\sup_{b\in\calB} L_b(0,1) \leq \gamma_2 + \Delta < \gamma_1\leq  \sup_{X\in A_1} \sup_{b\in\calB} L_b(0,1)\delta(X,b) \leq \sup_{b\in \calB} L_b(0,1),\] 
a contradiction. 
Finally, note that since $ L_b(0,1)$ is increasing in $b$, if $b_1 \notin \calB_0$ and $b_2 >b_1$, then $b_2\notin \calB_0$. Now, for all $b \in \calB_0\cup \calB_1$, we let $\deltahat(\cdot,b)=\delta(\cdot,b)$.  
For all $b \in \calB_2$, we set 
\begin{equation}
  \deltahat(X_1,b) = \delta(X_1,b) - \frac{\eps_1}{ L_b(0,1)}, \text{~ and ~} \deltahat(X_2,b) = \delta(X_2,b) + \frac{\eps_2}{ L_b(0,1)}, 
\end{equation}
for all $X_1\in A_1$ and $X_2 \in A_2$.  Outside of $A_1$ and $A_2$, $\deltahat$ will be defined as $\delta$. 
Throughout this proof, to save ourselves from constantly writing quantifiers, we will always assume that $X_1\in A_1$ and $X_2\in A_2$.

We must first ensure that $\deltahat$ is well-defined, i.e., $\deltahat(X_1,b)\geq 0$ and $\deltahat(X_2,b)\leq 1$ (it is clear that $\deltahat(X_1,b) \leq \delta(X_1, b) \leq 1$ and $\deltahat(A_2,b) \geq \delta(A_2,b) \geq 0$).  
By Lemma~\ref{lem:representation}, we may assume that $\delta(X_1,b) = \min\{1, E_\delta(X_1)/ L_b(0,1)\}$ and $\delta(X_2,b) = \min\{1, E_\delta(X_2)/ L_b(0,1)\}$. Therefore, to ensure that $\deltahat(X_1,b) = \min\{1, E_\delta(X_1)/L_b(0,1)\} - \eps_1/L_b(0,1)\geq 0$, it suffices that (a) $\eps_1\leq  L_b(0,1)$ and (b) $\eps_1\leq \inf_{X_1\in A_1} E_\delta(X_1) = \gamma_1$ (consider the two cases in the minimum).  Since $ L_b(0,1)$ is increasing in $b$ and $L_b(0,1)>\gamma_2+\Delta$ for all $b\in\calB_2$,  we can ensure that (a) and (b) are both  met by taking 
\begin{equation}
\label{eq:pf-inc-d1}
 \eps_1\leq  \gamma_2 + \Delta <\gamma_1.
\end{equation}
As for $\deltahat$ on $A_2$, notice that $\deltahat(X_2,b) = (E_\delta(X_2) + \eps_2)/ L_b(0,1)$ for all $b\in \calB_2$, which is at most 1 if $E_\delta(X_2)  +\eps_2\leq  L_b(0,1)$. Since $ \gamma_2 + \Delta \leq L_b(0,1)$ for $b\in \calB_2$ and $E_\delta(X_2) \leq \gamma_2$, it suffices to take $\eps_2 \leq \gamma_2 + \Delta - \gamma_2$, i.e., 
\begin{equation}
\label{eq:pf-inc-d2}
 0<\eps_2\leq \Delta.   
\end{equation} 
We have shown that there exist $\eps_1,\eps_2>0$ such that $\deltahat$ is well-defined. Next we show that they can be chosen to ensure that $\deltahat$ remains type-I risk safe and is $\Gamma_\const$-preferable to $\delta$. We begin with type-I safety. Since $\deltahat$ changes only on $A_1$ and $A_2$, it suffices to show that 
\begin{align}
\label{eq:pf-inc-in-lr-1}
    \int_{A_1\cup A_2} \sup_b  L_b(0,\deltahat(x,b))\d\P  \leq  \int_{A_1\cup A_2} \sup_b  L_b(0,\delta(x,b))\d\P.  
\end{align}
For all $b\in \calB_0\cup\calB_1$, we have $\delta(X_1,b) = \min\{1, E_\delta(X_1)/L_b(0,1)\} = 1$ since $L_b(0,1)\leq \gamma_2+\Delta<\gamma_1\leq E_\delta(X_1)$. Moreover, $\deltahat(X_1,b) = \delta(X_1,b)$ for such $b$. 
For $b\in \calB_2$ meanwhile, $\deltahat(X_1,b) = \min\{1, E_\delta(X_1)/L_b(0,1)\} - \eps_1/L_b(0,1) \leq (E_\delta(X_1) - \eps_1)/L_b(0,1)$ so 
\begin{align*}
    \sup_{b\in\calB_2}L_b(0,\deltahat(X_1,b)) &\leq  \sup_{b\in\calB_2}  L_b(0,1) \left(\frac{E_\delta(X_1)-\eps_1}{L_b(0,1)}\right) = E_\delta(X_1)- \eps_1,
\end{align*}
and 
\begin{align*}
    \sup_{b\in\calB} L_b(0,\deltahat(X_1,b)) &= \max\left\{\sup_{b\in\calB_0\cup B_1}  L_b(0,\delta(X_1,b)), \sup_{b\in\calB_2}  L_b(0,\deltahat(X_1,b)) \right\} \\
    &\leq \max\left\{\sup_{b\in \calB_0\cup\calB_1}  L_b(0,1), E_\delta(X_1) - \eps_1 \right\} \\ 
    &\leq \max\left\{\gamma_2 + \Delta, E_\delta(X_1) - \eps_1 \right\}.
\end{align*}
We want this final quantity to equal $E_\delta(X_1) - \eps_1$, so we choose $\eps_1$ such that 
\begin{equation}
\label{eq:pf-inc-d3}
 0 < \eps_1 \leq \gamma_1 - (\gamma_2 + \Delta) \leq E_\delta(X_1) - (\gamma_2 + \Delta),  
\end{equation}
which is possible by our choice of $\Delta$. 
This ensures the  maximum in the above display is achieved by $E_\delta(X_1)-\eps_1$ and we have 
\begin{equation}
\label{eq:pf-inc-d3.1}
    \sup_b L_b(0,\deltahat(X_1,b)) \leq E_\delta(X_1) - \eps_1. 
\end{equation}
Now let us consider $A_2$. For $b\in \calB_0$ we have $\deltahat(X_2,b) = \delta(X_2,b) \leq 1$ and for $b\in \calB_1$, $\deltahat(X_2,b) = \delta(X_2,b) = \min\{1, \gamma_2/L_b(0,1)\} = E_\delta(X_2)/L_b(0,1)$. For $b\in\calB_2$, $\deltahat(X_2,b) = \delta(X_2,b) = E_\delta(X_2) / L_b(0,1) + \eps_2/L_b(0,1)$, so 
\begin{align*}
    \sup_{b\in\calB}  L_b(0,\deltahat(X_2,b) &\leq \max\left\{ \sup_{b\in\calB_0}  L_b(0,1), \sup_{b\in\calB_1} L_b(0,\deltahat(X_2,b)), \sup_{b\in \calB_2} L_b(0,\deltahat(X_2,b))\right\} \\
    &= \max\left\{ \sup_{b\in\calB_0}  L_b(0,1), \sup_{b\in\calB_1} E_\delta(X_2), \sup_{b\in \calB_2} E_\delta(X_2) + \eps_2\right\} \\
    &\leq \max\{ \gamma_2, \gamma_2, \gamma_2 + \eps_2\} = \gamma_2 + \eps_2. 
\end{align*}
Combining this and~\eqref{eq:pf-inc-d3.1},  the left hand side of~\eqref{eq:pf-inc-in-lr-1} can be upper bounded as 
\begin{align*}
    \int_{A_1\cup A_2} \sup_{b\in\calB}  L_b(0,\deltahat(x,b))\d\P &\leq \int_{A_1}E_\delta(x)\d\P \d\P - \eps_1P(A_1) + \int_{A_2}E_\delta(x)\d\P - \eps_2P(A_2) \\ 
    &= \int_{A_1\cup A_2} \sup_b L_b(0,\delta(x,b)) \d\P - \eps_1 P(A_1) + \eps_2 P(A_2). 
\end{align*}
For~\eqref{eq:pf-inc-in-lr-1} to hold we thus require that $-\eps_1 P(A_1) + \eps_2 P(A_2) \leq 0$, i.e., 
\begin{equation}
\label{eq:pf-inc-in-lr-2}
  \frac{\eps_2}{\eps_1} \leq \frac{P(A_1)}{P(A_2)}.  
\end{equation}
Still leaving the precise choice of $\eps_1$ and $\eps_2$ unspecified for the moment, let us move on to type-II risk. We claim that for all $b\in\calB_2$, 
$\E_Q [\deltahat(X,b))] \geq  \E_\Q [\delta(X,b))]$. Since type-II risk remains unchanged for $b\in \calB_0\cup\calB_1$ (since $\deltahat(\cdot,b) = \delta(\cdot,b)$ for such $b$), this will show that $\deltahat$ is strictly preferable to $\delta$ with respect to $\Gamma_\const$. 

As with type-I risk, because $\deltahat$ is the same as $\delta$ outside of $A$, it suffices to focus on $A_1$ and $A_2$ and show that 
\begin{equation}
\label{eq:pf-inc-lr-21}
  \E_Q [\deltahat(X,b)\ind {X\in A}] \geq  \E_\Q [\delta(X,b)\ind{X\in A}],  
\end{equation}
for all $b\in\calB_2$. For these $b$, write 
\begin{align*}
    \int_{A_1} \deltahat(x,b)\d\Q &= \int_{A_1} \left(\delta(x,b) - \frac{\eps_1}{L_b(0,1)}\right)\d\Q  
    = \int_{A_1} \delta(x,b) \d\Q - \frac{\eps_1}{L_b(0,1)} Q(A_1).
\end{align*}
Similarly, 
\begin{align*}
    \int_{A_2} \deltahat(x,b)\d\Q =  \int_{A_2}  \delta(x,B) \d\Q + \frac{\eps_2}{L_b(0,1)}Q(A_2).
\end{align*}
Combining these two displays, we see that \eqref{eq:pf-inc-lr-21} holds iff 
\begin{equation}
\label{eq:pf-inc-in-lr-3}
    \frac{\eps_2}{\eps_1} > \frac{Q(A_1)}{Q(A_2)}. 
\end{equation}
We choose $\eps_1 = w P(A_1)$ and $\eps_2 = wP(A_2)$, where $w>0$ is small enough to ensure that conditions 
\eqref{eq:pf-inc-d1}, \eqref{eq:pf-inc-d2}, and \eqref{eq:pf-inc-d3} are met (note that none of these conditions are in conflict with one another; all are upper bounds on $\eps_1$ and $\eps_2$). Then, 
\[\frac{\eps_2}{\eps_1} = \frac{P(A_1)}{P(A_2)} > \frac{Q(A_1)}{Q(A_2)},\]
where the final inequality follows from Lemma~\ref{lem:QA>QB}. Inequalities~\eqref{eq:pf-inc-in-lr-2} and~\eqref{eq:pf-inc-in-lr-3} are thus met, demonstrating that $\deltahat$ is strongly preferable to $\delta$ with respect to $\Gamma_\const$. This contradicts that $\delta$ is $\Gamma_\const$-admissible, so we conclude that $E_\delta(X)$ is indeed increasing in the likelihood ratio $P$-almost surely. The second part of the lemma follows by applying Lemma~\ref{lem:test-evalue-inc}.

\subsection{Proof of Lemma~\ref{lem:compatible-inc}}
\label{sec:proof-compatible-inc}
Suppose not. By Lemma~\ref{lem:E>=Lb} we have that $E_\delta(X) \geq \inf_{b\in\calB}L_b(0,1)$ $Q$-almost surely. Morever, we cannot have $E_\delta(X) > \sup_{b\in \calB} L_b(0,1)$ by definition. Therefore, there exists some $A$ with $Q(A)>0$ and some partition of $\calB = \calB_0 \cup \calB_1$ such that 
\[
\sup_{b\in \calB_0}L_b(0,1) < E_\delta(X) < \inf_{b\in \calB_1}L_b(0,1). 
\]
For simplicity we will assume that there exist $b_1,b_2$ such that $b_1 = \sup_{b\in \calB_0} L_b(0,1)$ and $b_2 = \inf_{b\in\calB_1}L_b(0,1)$, but the proof can be amended if these supremum and infimum are not in $\calB$. We can thus write $\calB_1 = \{b\geq b_2\}$ since the losses are assumed to be increasing in $b$.  

Let $A = A_1\cup A_2$ where $P(A_1) = P(A_2)$ and $\lr(X_1)<\lr(X_2)$ for all $X_1\in A_1$ and $X_2\in A_2$ (this is possible by continuity). 
Define a new decision rule $\deltahat$ such that, for all $b\in\calB_0$,  $\deltahat(X,b)= \delta(X,b) - \eps/ L_b(0,1)$ for all $X\in A_1$ and $\deltahat(X,b) = \delta(X,b) + \eps/ L_b(0,1)$ for all $X\in A_2$. Otherwise $\deltahat$ is the same as $\delta$. We have $0<\delta(X,b)<1$ for all $X\in A$ and $b\in\calB_1$ (the lower bound of 0 is implied by Lemma~\ref{lem:representation}). Therefore, we may find $\eps$ small enough such that $0<\deltahat(X,b)\leq 1$ for all $X\in A$. In particular, let us ensure that $\eps$ is small enough such that 
\begin{equation}
\label{eq:pf-rich-d1}
  L_{b_1}(0,1) + \eps < E_\delta(X).  
\end{equation}
To see that $\deltahat$ is type-I risk safe, observe that for $X\in A_1$, 
\begin{align*}
    E_{\deltahat}(X) &= \sup_{b}L_b(0, \deltahat(X,b))  \\
    &= \max\left\{ \sup_{b<b_2} L_b(0,\delta(X,b)), \sup_{b\geq b_2} L_b(0,\deltahat(X,b))\right\} \\ 
    &= \max\left\{ L_{b_1}(0,1), \left(\sup_{b\geq b_2} L_b(0,1)\delta(X,b)\right) - \eps \right\} \\ 
    &= \max\left\{ L_{b_1}(0,1), E_\delta(X) - \eps \right\} \\
    &= E_\delta(X) -\eps,
\end{align*}
where the final equality follows by~\eqref{eq:pf-rich-d1}. 
Similarly, for $X\in A_2$, 
\begin{align*}
    E_{\deltahat}(X) 
    &= \max\left\{ L_{b_1}(0,1), \left(\sup_{b\geq b_2} L_b(0,1)\delta(X,b)\right) + \eps \right\} \\ 
    &= \max\left\{ L_{b_1}(0,1), E_\delta(X) + \eps \right\} \\ 
    &= E_\delta(X) + \eps.
\end{align*}
Therefore, 
\begin{align*}
    \int_A E_{\deltahat}(x)\d\P  &= \int_{A_1} (E_\delta(x) - \eps)\d\P  + \int_{A_2} (E_\delta(x) + \eps)\d\P\\ 
    &= \int_A E_\delta(x) - \eps P(A_1) + \eps P(A_2) = \int_A E_\delta(x)\d\P, 
\end{align*}
since $P(A_1) = P(A_2)$ by assumption. From this it follows that $\risk_P(\deltahat) \leq 1$ since 
\[
\E_P[E_{\deltahat}(X)] = \int_{A'}E_{\deltahat}(x) \d\P + \int_{A}E_{\delta}(x) \d\P = \int_{\calX} E_\delta(x) \d\P \leq 1.
\]
So much for type-I risk. For type-II risk, the argument is nearly identical to that in Lemma~\ref{lem:E>=Lb}, so we omit it here. 

\subsection{Proof of Lemma~\ref{lem:inc-in-lr-binary}}
\label{sec:proof-inc-in-lr-binary}

Throughout the proof we drop the superscript ``bin'' on $\delta$. 
By Lemma~\ref{lem:test-evalue-inc} $\delta(X,b)$ is increasing in $\lr$ iff $E_\delta$ is increasing in the $\lr$, so we focus on the test family itself. 

Suppose that for some $b^*\in\calB$, $\delta(\cdot, {b^*})$ is not an increasing function of $\lr$. Then we can find two sets, $A_1,A_2\subset\calX$ with positive measure under both $\P$ and $\Q$ such that, for all $X_1\in A_1$ and $X_2\in A_2$, 
    \begin{equation}
    \label{eq:decreasing_lr}
    \delta(X_1,b^*) = 1, \; \delta(X_2,b^*) = 0, \text{~ and ~} 
    \sup_{X\in A_1} \lr(X_1) < \inf_{X_2\in A_2} \lr(X_2) 
    \end{equation}
    Define a set $G = [t_-, t_+]\subset \calB \cup \{\inf(\calB)\}\cup \{\sup(\calB)\}$ as follows.\footnote{Here $\sup(\calB) = \sup_{b\in\calB} b$. Similarly for $\inf(\calB)$. }  
    Suppose we can find some $A_1'\subset A_1$ and $t>  b^*$ such that $\delta(X_1,b)=0$ for all $b>t$ and all $X_1\in A_1'$. Then take $t^+= t$ and redefine $A_1\gets A_1'$. If no such $A_1'$ and $t$ exist, then $\delta(X_1,b)=1$ for all $X_1\in A_1$ all $b\geq b^*$, since $\delta$ is nonincreasing in $b$ by Lemma~\ref{lem:decreasing_in_b} (if $\delta(X,b) = 0$ for \emph{some} $b\geq b^*$ then $\delta(X,{b'})=0$ for all $b'\geq b$). 
    In this case we set $t^+ = \sup(\calB)$. 

    As for $t_-$, suppose we can find some $A_2'\subset A_2$ and some $t\leq b^*$ such that $\delta(X,b) = 1$ for all $x\in A_2'$ and $b<t$. Then we take $t = t_-$ and redefine $A_2 \gets A_2'$. Otherwise, $\delta(X,b)=0$ for all $b\leq b^*$ for all $X\in A_2$. 
    In this case take $t_-=\inf(\calB)$. 

    Finally, if $\P(A_2)>\P(A_1)$, then take a subset of $A_2'\subset A_2$ such that $\P(A_2') = \P(A_1)$ (which is possible by continuity). We may therefore assume that $\P(A_1) = \P(A_2)$. 

    The upshot of this construction is contained in the following three facts: 
    \begin{enumerate}
        \item[F1.] $\P(A_1)= \P(A_2)$
        \item[F2.] If $b>t_+$ then $\delta(X_1,b)$ for all $X_1\in A_1$. 
        \item[F3.] If $b<t_-$ then $\delta(X_2,b) = 1$ for all $X_2\in A_2$. 
    \end{enumerate}
    As we did in the proof of Lemma~\ref{lem:inc-in-lr}, we will always assume that $X_1\in A_1$ and $X_2\in A_2$ to save ourselves from constantly writing quantifiers. 
    
    Now, let $G^+ :=\{ b\in\calB: b>t_+\}$ and $G^- = \{b\in\calB: b<t_-\}$. Either (or both) of these sets may be empty. Note that $G \cup G^+\cup G^-= \calB$. If $G^+$ and $G^-$ are non-empty, facts F1 and F2 above can be rephrased as 
    \begin{equation}
    \label{eq:all-0-or-1}
      \sup_{X_1\in A_1} \sup_{b\in G^+} \delta(X_1,b)= 0, \quad \inf_{X_2\in A_2} \inf_{b\in G^-} \delta(X_2,b) =1.
    \end{equation}
    We define a new decision family, $\deltahat$, where 
    \begin{equation}
        \deltahat(X,b) = \begin{cases}
            \delta(X,b) & X\notin A_1\cup A_2\text{ or }b\notin G, \\
            0, & X\in A_1\text{ and }b\in G, \\
            1,& X\in A_2\text{ and }b\in G.
        \end{cases}
    \end{equation}
    We claim that $\deltahat$ remains type-I risk safe and has (sometimes strictly) lower type-II risk than $\delta$, thus demonstrating that $\delta$ is $\Gamma_\const$-inadmissible. 
    To see type-I risk safety, write
    \begin{equation}
    \label{eq:proof-increasing-delta_risk}
        \risk_\P(\deltahat) = \int_{\calX\setminus A_1\cup A_2} \sup_b L_b(0,\delta(x,b))\d\P + \int_{A_1 \cup A_2} \sup_bL_b(0,\deltahat(x,b))\d\P,   
    \end{equation}
    where we've used that on $\calX\setminus A_1\cup A_2$, $\deltahat(\cdot,b) = \delta(\cdot,b)$ for all $b$. Meanwhile for all $X_1\in A_1$, $L_b(0,1) \deltahat(X_1,b)=0$ for all $b\in G\cup G^+$ by F2 and definition of $\deltahat$ (since $\deltahat(X_1,b) \equiv 0$ by construction for $b\in G$ and $\deltahat(X_1,b) = \delta(X_1,b)\equiv 0$ for $b\in G^+$ by F2. If $G^+=\emptyset$ this holds trivially). Hence, for any $X_1\in A_1$, 
    \begin{align*}
     \sup_{b\in\calB} L_b(0,1)\deltahat(x_1,b) = \sup_{b\in G^-} L_b(0,1) \deltahat(x_1,b) = \sup_{b\in G^-} L_b(0,1) \delta(x_1,b) = \sup_{b\in G^-} L_b(0,1),
    \end{align*}
    where we interpret the supremum as 0 if $G^-=\emptyset$. 
    Therefore, 
    \begin{equation*}
        \int_{A_1} \sup_b L_b(0,1)\deltahat(x,b) \d\P \leq \int_{A_1} \sup_{b\in G^-}L_b(0,1) \d\P = \int_{A_2}\sup_{b\in G^-}L_b(0,1) \d\P, 
    \end{equation*}
    since $\P(A_1) = \P(A_2)$. Moreover, 
    \begin{equation*}
        \int_{A_2}\sup_{b\in G^-}L_b(0,1)  \d\P = \int_{A_2}\sup_{b\in G^-}L_b(0,1) \delta(x,b) \d\P \leq \int_{A_2}\sup_{b\in\calB}L_b(0,1)\delta(x,b) \d\P. 
    \end{equation*}
    To summarize, the previous two displays imply that:  
    \begin{equation}
    \label{eq:proof-increasing-A1_bound}
        \int_{A_1} \sup_{b\in\calB} L_b(0,1)\deltahat(x,b)\d\P \leq    \int_{A_2}\sup_{b\in\calB}L_b(0,1)\delta(x,b) \d\P. 
    \end{equation}
    Next, note that for all $X_2\in A_2$, $\deltahat(X_2,b) = 0$ for any $b\in G^+$, 
    $\deltahat(X_2,b) = 1$ for $b\in G$ by definition, and $\deltahat(X_2,b) = \delta(X_2,b) = 1$ for $b\in G^-$ by F3. 
    Therefore, since $L_b(0,1)$ is increasing in $b$, 
    \[\sup_{b\in\calB} L_b(0,1) \deltahat(X_2,b) \leq \sup_{b\in G\cup G^-} L_b(0,1) = \sup_{b\in G} L_b(0,1),\]
    and so, again using that $\P(A_1) = \P(A_2)$, 
    \begin{align}
        \int_{A_2} \sup_{b\in\calB}L_b(0,1)\deltahat_b(x) \d\P &\leq \int_{A_2} \sup_{b\in G}L_b(0,1) \d\P \notag \\
        &= \int_{A_1} \sup_{b\in G}L_b(0,1) \d\P 
        \leq \int_{A_1}\sup_{b\in \calB} L_b(0,1)\delta_b(x) \d \P. \label{eq:proof-increasing-A2_bound}
    \end{align}
    Here the final inequality uses that 
    \[\sup_{b\in G}L_b(0,1) = \sup_{b\in G}L_b(0,1) \delta(X_1,b) \leq \sup_{b\in\calB} L_b(0,1)\delta(X_1,b),\] 
    for all $X_1\in A_1$. 
    Combining \eqref{eq:proof-increasing-A1_bound} and \eqref{eq:proof-increasing-A2_bound} we have 
    \begin{equation}
        \int_{A_1 \cup A_2} \sup_{b\in\calB}L_b(0,1)\deltahat(x,b)\d\P \leq \int_{A_1\cup A_2}\sup_{b\in \calB} L_b(0,1)\delta(x,b)\d\P. 
    \end{equation}
    Hence, by \eqref{eq:proof-increasing-delta_risk}, 
    $\risk_\P(\deltahat)\leq \int_\calX \sup_{b\in\calB} L_b(0,1)\delta_b(X)\d\P = \risk_\P(\delta)\leq 1$.  
    
    Next we show that $\widehat{\delta}$ has strictly lower type-II risk than $\delta$ for all $b\in G$. Since  the type-II risk of $\deltahat$ is the same as that of $\delta$ for all $b\notin G$ (by construction), this will show that $\delta$ is inadmissible.  Since $P(A_1) = P(A_2)$ by construction, we have $Q(A_1) < Q(A_2)$ by Lemma~\ref{lem:QA>QB}. 
    Therefore,  for any $b\in G$, since $\deltahat(X_2,b) = 1$ for all $X_2\in A_2$ and $\deltahat(X_1,b)=0$ for all $X_1\in A_1$, we have 
    \begin{align*}
       \E_{\Q}[\deltahat_b(X)] 
       &= \int_{\calX \setminus A_1\cup A_2} \deltahat(x,b)\d\Q + \int_{A_1\cup A_2}\deltahat(x,b)\Q(A_2) \\ 
       &= \int_{\calX \setminus A_1\cup A_2} \delta(x,b)
       d\Q  + \Q(A_2) \\
        &> \int_{\calX \setminus A_1\cup A_2} \delta(x,b)\d\Q + \Q(A_1) = \E_{\Q}[\delta(X,b)],
    \end{align*}
    which proves that $\delta$ is $\Gamma_\const$-inadmissible, completing the argument. 

\subsection{Proof of Corollary~\ref{cor:decision-curve}}
\label{proof:cor-decision-curve}
    Lemma~\ref{lem:inc-in-lr-binary} implies that a $\consta$-admissible binary test $\vp$ is increasing in $\lr$, hence so too is $E_\psi$ by  Lemma~\ref{lem:test-evalue-inc} in the appendix. Therefore, there exists an increasing function $h$ such that $E_\phi(X) = h(\lr(X))$ and we may write $\vp(X,b) = \ind{E_\phi(X) \geq L_b(0,1)} = \ind{h(\lr(X)) \geq L_b(0,1)} = \ind{\lr(X) \geq h^-(L_b(0,1))}$, where $h^-$ is the generalized inverse of $h$ and the final equality holds since $h$ is increasing. The function $t(b) = h^-(L_b(0,1))$ defines the decision curve.

\subsection{Proof of Theorem~\ref{thm:const-admissibility-binary}}
\label{sec:proof-const-admissibility-binary}

    If $\vp$ is $\Gamma_\const$-admissible then $E_\delta$ is sharp by Lemma~\ref{lem:Ephi>Edelta}, increasing in $\lr$ by Lemma~\ref{lem:inc-in-lr-binary}, and has a minimal decision curve by Lemma~\ref{lem:minimal_decision_curve}. This proves the forward direction. 
    Conversely, 
    suppose that $E_{\vp}$ is sharp and increasing in $\lr$ and that $\vp$ has a minimal decision curve but that $\vp$ is $\Gamma_\const$-inadmissible. Then, by Lemma~\ref{lem:admissibility_by_t} there exists some $\psi$ such that $t_\phi(b)\leq t_\delta(b)$ for all $b\in\calB$ and $t_\phi(b^*)<t_\delta(b^*)$ for some $b^*$, with $Q(t_\phi(b^*)\leq \lr(X) < t_\delta(b^*)) > 0$. Let $I(b^*) = \{X: t_\phi(b^*) \leq \lr(X) < t_\delta(b^*)\}$. We claim that 
    \begin{equation}
    \label{eq:pf-const-admissibility-1}
        \int_{I(b^*)} \sup_b L_b(0,\vp(x,b)) \d\P < \int_{I(b^*)} \sup_b L_b(0,\psi(x,b)) \d\P. 
    \end{equation}
    For $X\in I(b^*)$, $\phi(X,b^*) = \ind{\lr(X) \geq t_\phi(b^*)} = 1$, so $\sup_b L_b(0,\psi(X,b)) \geq L_{b^*}(0,1)$ and the right hand side of~\eqref{eq:pf-const-admissibility-1} is at least $L_{b^*}(0,1) P(I(b^*))$. We claim that the integrand on the left hand side is less than $L_{b^*}(0,1)$. 
    We consider two cases. 
    \begin{enumerate}
        \item Suppose that we can find some $b_0$  such that $t_\phi(b^*) < t_\delta(b_0) < t_\delta(b^*)$ and $\P(t_\phi(b^*) \leq \lr(X) < t_\delta(b_0)) > 0$. Let $I(b^*) = I_1 \cup I_2$ where $I_1 = \{X: t_\phi(b^*)\leq \lr(X) < t_\delta(b_0)\}$ and $I_2 = \{X: t_\delta(b_0)\leq \lr(X) < t_\delta(b^*)\}$. Since $t_\delta$ is increasing, for $X\in I_1$, we have $\ind{\lr(X) \geq t_\delta(b)} = 0$ for all $b>b_0$ 
        so 
        \begin{align*}
         \sup_b L_b(0,\vp(X,b)) = \sup_b L_b(0,1)\ind{\lr(X) \geq t_\delta(b)} \leq L_{b_0}(0,1) < L_{b^*}(0,1). 
        \end{align*}
        Similarly, for $X\in I_2$, $\sup_b L_b(0,\vp(X,b)) \leq L_{b^*}(0,1)$ and it follows that 
        \begin{align*}
            \int_{I(b^*)} \sup_b L_b(0,\vp(x,b)) \d\P &= L_{b_0}(0,1)P(I_1) + L_{b^*}(0,1)P(I_2) < L_{b^*}(0,1) P(I(b^*)),   
        \end{align*}
        using since $b_0 < b^*$ (again since $t_\delta$ in increasing). Thus, in this case, \eqref{eq:pf-const-admissibility-1} holds. 
        \item Next suppose that no such $b_0$ exists. Therefore $\sup_{b<b^*}t_\delta(b) \leq t_\phi(b^*)<t_\delta(b^*)$. 
        Since $Q(\sup_{b<b^*}t_\delta(b) \leq \lr(X) < t_\delta(b^*)) \geq Q(t_\phi(b^*)\leq \lr(X) < t_\delta(b^*))>0$ by assumption of $\phi$, it follows that $\sup_{b<b^*}L_b(0,1) < L_{b^*}(0,1)$ since $\vp$ has a minimal decision curve. Therefore, since $\vp(X,b) = \ind{\lr(X) \geq t_\delta(b)}$ for all $b>b^*$ for $X\in I(b^*)$, we have 
        \begin{align*}
            \int_{I_{b^*}} \sup_{b}L_b(0,\vp(x,b))\d\P &= \int_{I(b^*)} \sup_{b<b^*} L_b(0,1)\ind{\lr(x) \geq t_\delta(b)} \d\P \\
            &\leq  \int_{I(b^*)} \sup_{b<b^*} L_b(0,1) \d\P < L_{b^*}(0,1)P(I(b^*)), 
        \end{align*}
        and the final quantity is a lower bound on the right hand side of~\eqref{eq:pf-const-admissibility-1}. 
    \end{enumerate}
    We have thus proved~\eqref{eq:pf-const-admissibility-1}. 
    Note all for all $X\in\calX$, $\phi(X,b) \geq \delta(X,b)$ since $t_\phi(b) \leq t_\delta(b)$. Therefore, 
    \begin{align*}
        \risk_P(\vp) &= \int_{\calX\setminus I(b^*)} \sup_b L_b(0,\vp(X,b)) \d\P +  \int_{I(b^*)} \sup_b L_b(0,\vp(X,b)) \d\P \\ 
        &< \int_{\calX\setminus I(b^*)} \sup_b L_b(0,\psi(X,b)) \d\P +  \int_{I(b^*)} \sup_b L_b(0,\psi(X,b)) \d\P \\
        &= \risk_P(\psi)\leq 1, 
    \end{align*}
    contradicting that $E_{\vp}$ is sharp. This proves the first part of theorem. 

    For the second half of the theorem, let $E$ be an e-variable for $P$. Suppose $\vp_E$ is $\Gamma_\const$-admissible. By Lemma~\ref{lem:compatible}, $E = E_{\vp_E}$. 
    Then, by the first past of the theorem, $E$ is sharp, compatible, and increasing in $\lr$. 
Conversely, suppose $E$ is sharp, compatible, and increasing in $\lr$. By compatibility, $E = E_{\vp_E}$ so it remains only to show that $\vp_E$ has a minimal decision-curve, in which case $\Gamma_\const$-admissibility will again follow from applying the first part of the theorem. 
Let $h$ be an increasing function such that $E(X) + h(\lr(X))$ and let $h^-(y)=\inf\{x\in\Re: h(x)\geq y\}$ by its generalized inverse. Since $h$ is increasing, $h^-$ is left continuous. Further, 
\begin{equation*}
    \vp_E(X,b) = \ind{E(X) \geq L_b(0,1)} = \ind{\lr(X) \geq h^-(L_b(0,1))},
\end{equation*}
so $h^-(L_b(0,1))$ is equal to $t_\delta(b)$ $P$-almost surely. Therefore, if $\sup_{b<b^*}L_b(0,1) = L_{b^*}(0,1)$ then $h^-(\sup_{b<b^*}L_b(0,1)) = h^-(L_{b^*}(0,1))$, hence $t_\delta(b)$ is a minimal decision curve. This completes the proof.

\section{On the definition of admissibility} 
\label{app:admissibility}

First let us discuss the notion of admissibility used by \citet{grunwald2024beyond}, which we henceforth refer to as G-admissibility. We say a test family $\delta$ is G-\emph{in}admissible if there exists a type-I risk safe test family $\phi$ such that, for all $b\in\calB$ and all $Q\in\calQ$, 
\begin{equation}
    Q(L_b(1,\phi(X,b)) > L_b(1,\delta(X,b)) = 0,
\end{equation}
and there exists some $b\in\calB$ and some $Q\in\calQ$ such that 
\begin{equation}
    Q(L_b(1,\phi(X,b)) < L_b(1,\delta(X,b))) > 0. 
\end{equation}
In other words, $\phi$ has loss at most that of $\delta$ with $Q$-probability 1 on all losses, and has lower loss than $\delta$ with positive probability. If no such $\phi$ exists then we say that $\delta$ is G-admissible. We note that \citet{grunwald2024beyond} defines G-admissibility in terms of type-I losses instead of type-II, after preprocessing the losses to remove any such that $L_b(1,a) > L_b(1,a)$ and $L_b(0,a) > L_b(0,a)$ for some $a$, which can be done without loss of generality. 

Now, if we translate Gr\"{u}nwald's results to our setting---in particular, we consider binary actions only---then his main result can be formulated as follows. Consider testing $\calP$ against $\calQ$.  Then: 
\begin{enumerate}
    \item If $\delta$ is G-admissible, then there exists some e-variable $E$ for $\calP$ such that $\delta(X,b) = \ind{E(X) \geq L_b(0,1)}$. 
    \item If $\delta = \ind{E(X)\geq L_b(0,1)}$ for some some e-variable $E$ for $\calP$ that is sharp, compatible with $\calL_0$, and all $P\in \calP$ are absolutely continuous with one another, then $\delta$ is G-admissible. 
\end{enumerate}

\edit{
Let us now discuss some relationships between $\Gamma$-admissibility and G-admissibility. 
First, $\consta$-admissibility implies G-admissibility. This is intuitive: If a test family $\delta$ cannot be beat in expectation for fixed $b$, then it cannot be beat with probability 1 either. 

\begin{lemma}
    $\consta$-admissibility implies G-admissibility. 
\end{lemma}
\begin{proof}
We prove the contrapositive. Suppose that $\delta$ is G-inadmissible. So there exists some $\phi$ such that, for all $b\in\calB$, 
\[\phi(X,b) \geq \delta(X,b)\quad Q\text{-almost everywhere for all }Q\in\calQ,\] 
and $\phi(X,b_0) > \delta(X,b_0)$ for some $b_0$ and all $X\in A$ with $Q_0(A)>0$ for some $Q_0$. Therefore, $\E_{Q_0} [\delta(X,b_0)\ind{A}] < \E_{Q_0}[\phi(X,b_0)\ind{A}]$  and 
$\E_{Q_0}[\delta(X,b_0) \ind{A^c}]\leq \E_{Q_0}[\phi(X,b_0)\ind{A^c}]$. Thus $\E_{Q_0}[\delta(X,b_0)] < \E_{Q_0}[\delta(X,b_0)]$. 
Moreover, since $\delta(X,b)\geq \phi(X,b)$ $\calQ$-almost surely for all $b$, we have $\E_Q[\delta(X,b)]\leq  \E_Q[\delta(X,b)]$ for all $b$, hence $\phi$ is strictly preferable to $\delta$ with respect to $\alla$. Thus G-inadmissibility implies $\consta$-inadmissibility.     
\end{proof}

Note, however, that this proof does not work if we replace $\consta$-admissibility by $\alla$-admissibility. Consider the following example: Set $\calB=[1,\infty)$ and let $X$ take values on $[1,\infty)$. Suppose that $\delta(X,b) > \phi(X,b)$ for all $X\neq b$, and for $X=b$ we have $\delta(X,b) <\phi(X,b)$. Since $X=b$ has measure zero, $\delta$ is strictly preferable to $\phi$ with respect to G-admissibility. But for $B(X) = x$ we have $\E_\Q[L_B(1,\delta(X,B))] = \E_\Q[L_X(1,0)(1 - \delta(X,X))] > \E_\Q[L_X(1,0) (1 - \phi(X,X))] = \E_Q[L_B(1,\phi(X,B))]$, so $\delta$ is not preferable to $\phi$ with respect to $\alla$-admissibility. This shows that we cannot conclude $\alla$-inadmissibility from G-inadmissibility. 

Of course, this does not provide a counterexample to the claim that $\alla$-admissibility implies G-admissibility per se. But it sheds light on the fact that one would need a proof totally unlike the one above. This makes us doubt that the claim is true.  
}

Next let us return to a question raised in Section~\ref{sec:testing_as_decision_theory}: Why not define admissibility with a supremum over losses inside the expectation? That is, we could say that $\delta$ is admissible if it is type-I risk safe and for all other type-I risk safe tests $\phi$, 
\begin{equation}
\label{eq:stringent_admissibility}
  \E\bigg[\sup_{b\in\calB} ( L_b(1,\delta_b(X)) -  L_b(1,\phi_b(X))\bigg] \leq 0. 
\end{equation}
Such a definition, however, leads to an empty set of admissible test families. Intuitively, no test can be as good as any other on every data-dependent loss. To do so, it must be the most powerful test on every loss. That is, for a point null versus a point alternative, it must be the Neyman-Pearson test. But it is invalid to play the Neyman-Pearson on each loss, since such a test family is not type-I risk safe. 
This is formalized in the following argument. 

Suppose we are testing a simple null vs simple alternative, $\calP = \{\P\}$ vs $\calQ = \{\Q\}$. 
Consider two loss functions, $\calB = \{0,1\}$. Suppose $\delta$ is admissible according to~\eqref{eq:stringent_admissibility} and consider the Neyman-Pearson test on each loss. That is, for $i\in\{0,1\}$, define 
\[\phi^{(i)}(X,b) = \begin{cases}
    \ind{\lr(X) \geq \kappa_{i}}, & \text{if }b=i \\ 
    0, & \text{otherwise},
\end{cases}\]
where, as usual,  $\kappa_i$ is chosen such that $\E_{\P}[\phi^{(i)}(X,b)] = 1/L_{i}(0,1)$. 
Note that $\phi^{(i)}$ is type-I risk safe. 
Also, 
\begin{align*}
    & \quad \E_{Q} \left[\sup_b\left( L_b(1,\delta(X,b)) -  L_b(1,\phi^{(i)}(X,b))\right)\right] \\
    &\geq \E_{Q} [L_i(1,\delta(X,i)) - L_i(1,\phi^{(i)}(X,i)] \\ 
    &= L_i(1,0) \E_{Q} [\phi^{(i)}(X,i) - \delta(X,i)] \geq 0, 
\end{align*}
where the last inequality follows 
since $\phi^{(i)}$ is the uniformly most powerful test  by the Neyman-Pearson lemma. Therefore, by admissibility we conclude that $\E_{Q}[\phi^{(i)}(X,i) - \delta(X,i)]=0$. That is, $\delta$ has the same power as $\phi^{(i)}$ under loss $b_i$ and thus, by the uniqueness of the Neyman-Pearson lemma, we have $\delta = \phi^{(i)}$ almost surely. But $i$ was arbitrary, so $\phi^{(1)} = \delta = \phi^{(2)}$,  a contradiction. 

\subsection{Type-I risk as the supremum over adversaries}
\label{sec:risk-as-adversary}

Let us now discuss~\eqref{eq:risk-as-adversary}, i.e., the fact that we can write type-I risk as the supremum over maps from the data to losses. 

Recall that $(\calA,\Sigma)$ is a standard Borel space if there exists a complete, separable metric $\rho$ on $X$ such that $\Sigma$ becomes the Borel $\sigma$-field under $\rho$. A probability space $(\Omega, \calF, P)$ is Borel if $(\Omega, \calF)$ is Borel. If we call a function $Y:(\calA,\Sigma_\calA)\to(\Re,\Sigma_\Re)$ $\Sigma_\calA$-measurable, we mean that it is  $\Sigma_\calA/\Sigma_\Re$-measurable where $\Sigma_\Re$ is understood to the the Borel $\sigma$-field; likewise if $\Re$ is replaced with any subset of $\Re$.

\begin{lemma}
\label{lem:risk-as-adversary}
Let $(\Omega, \calF,P)$ be a Borel probability space. 
Equip $\calX$ and $\calB$ with $\sigma$-fields $\Sigma_X$ and $\Sigma_\calB$ such that $(\calX,\Sigma_\calX)$ and $(\calB,\Sigma_\calB)$ are standard Borel spaces. 
Consider any set of type-I loss functions $\{L_b(0,1):b\in\calB\}$ such that $b\mapsto L_b(0,1)$ is $\Sigma_\calB$-measurable. 
Let $\delta$ be a test family that is $(\Sigma_\calX\otimes \Sigma_\calB)$-measurable. 
If $\delta$ has finite type-I risk, then: 
\begin{equation}
\label{eq:risk-as-adversary-2}
    \E_{X\sim \P}\sup_{b\in\calB} L_b(0,\delta(X,b))= \sup_{B:\calX \to \calB} \E_{X\sim \P} L_{B(X)}(0,\delta(X,B(X))),
\end{equation}    
where the supremum on the right is over all maps $B$ that are $\Sigma_\calX/\Sigma_\calB$-measurable. 
\end{lemma}
\begin{proof}
First we claim the left hand side is well-defined. 
Define the function $g: \Omega\to \Re$ by  $g(\omega) = \sup_b L_b(0,\delta(X(\omega),b))$. We claim that $g$ is $\calF$-measurable. Since $L_b= L_b(0,1)$ is $\Sigma_\calB$-measurable and $\delta$ is $(\Sigma_\calX\otimes\Sigma_\calB)$-measurable, $L_b(0,\delta(x,b))=L_b(0,1)\delta(x,b)$ is $(\Sigma_\calX\otimes\Sigma_\calB)$-measurable. Consequently, the map $(\omega,b)\mapsto L_b(0,\delta(X(\omega),b))$ is $(\calF\otimes \Sigma_\calB)$-measurable. Consider the super-level sets 
\begin{equation*}
 \{\omega: g(\omega)>c\} = \Gamma_\Omega(\underbrace{\{(\omega, b): L_b(0,1)\delta(X(\omega),b) > c\}}_{=:A_c}),   
\end{equation*}
where $\Gamma_\Omega$ represents the projection of a set in $\Omega\times \calB$ onto $\Omega$. The set $A_c$ is Borel (i.e., lies in $\Sigma_\calX\otimes\calB$) by the measurability of $L_b(0,1)\delta(X(\omega),b)$. The projection of a Borel set on a product space is analytic by Suslin's theorem, hence $\{\omega:g(\omega)>c\}$ is analytic~\citep[Exercise 14.3]{kechris2012classical}. Moreover, every analytic set is universally measurable~\citep[Theorem 21.10]{kechris2012classical}, meaning it is measurable for any finite Borel measure on $\Omega$, and $P$ is finite and Borel by assumption. Finally, if the super-level sets of a real-valued function are measurable, then the function itself measurable. We conclude that $g$ is $\calF$-measurable, so the left hand side of~\eqref{eq:risk-as-adversary} exists. Now, for any $X$ and $B$, $L_{B(X)}(0,\delta(X,B(X))) \leq \sup_b L_b(0,\delta(X,b))$. Taking expectations and then taking the supremum over all maps $B$, we see that the right hand side of~\eqref{eq:risk-as-adversary} is at most the left hand side. 

To prove the converse fix $\eps>0$ and let $\omega$ be given. Define the multifunction $\Xi_\eps:\Omega \rightrightarrows \calB$ 
\begin{equation*}
    \Xi_\eps(\omega) = \left\{ b^*\in\calB: L_{b^*}(0,\delta(X(\omega), b^*)) > \sup_b L_b(0,\delta(X(\omega),b)) - \eps\right\},
\end{equation*}
which is $\calF$-measurable since $(\omega,b)\mapsto L_b(0,\delta(X(\omega),b))$ is measurable. The Kuratowski–Ryll-Nardzewski measurable selection theorem~\citep{kuratowski1965general} (see also \citealt[Theorem 12.13]{kechris2012classical}) guarantees the existence of a measurable selector $B^*_\eps(X(\omega))$ such that $B_\eps(X(\omega)) \in \Xi_\eps(\omega)$, hence $L_{B^*_\eps(X)}(0,\delta(X(\omega),B_\eps^*(X)) + \eps > \sup_b L_b(0,\delta(X,b))$. Taking expectations and letting $\eps\downarrow0$ shows that the 
left hand side of~\eqref{eq:risk-as-adversary} is at most the right, which proves the result. 
\end{proof}

The following corollary
extends the result to composite null hypotheses. The proof follows by taking the supremum over $P\in\calP$ on both sides of~\eqref{eq:risk-as-adversary-2}. 

\begin{corollary}
    \label{cor:risk_as_adversary_composite}
    Let $(\Omega, \calF)$ be a measurable space and let $\calP$ be a collection of probability measures such that $(\Omega,\calF,P)$ is a Borel probability space for for each $P\in\calP$. Let $(\calX,\Sigma_\calX)$, $(\calB,\Sigma_\calB)$, and $\{L_b(0,1):b\in\calB\}$ be as in Lemma~\ref{lem:risk-as-adversary}. Then 
    \begin{equation}
    \sup_{P\in\calP}\E_{X\sim \P}\sup_{b\in\calB} L_b(0,\delta(X,b))= \sup_{P\in\calP}\sup_{B:\calX \to \calB} \E_{X\sim \P} L_{B(X)}(0,\delta(X,B(X))),
\end{equation} 
where the rightmost supremum is over all maps $B$ that are $\Sigma_\calX/\Sigma_\calB$-measurable. 
\end{corollary}

\begin{remark}
    \label{rem:measurability}
    There exist pathological choices of $\{L_b\}_{b\in\calB}$ which ensure that $\sup_b L_b(0,\delta(X,b))$ is not measurable, in which case Definition~\ref{def:risk} breaks down and Lemma~\ref{lem:risk-as-adversary} doesn't hold. Instead of making various measurability assumptions as we do in this work, one could circumvent this problem by following \citet{grunwald2024beyond} and calling $\delta$ type-I risk safe if there exists some measurable function $\xi:\calX\to\Re_{\geq 0}$ such that 
    \begin{equation}
        \sup_{\P\in \calP} \E_{X\sim \P} \xi(X) \leq 1\text{~ and ~} \sup_{b\in \calB} L_b(0,\delta(X,b))\leq \xi(X)\quad  \calP\text{-almost surely.}
    \end{equation}
\end{remark}

\section{On the relationship to Bayesian decision theory}
\label{app:bayesian}
A common response to the drawbacks of classical Neyman-Pearson theory that we've highlighted is to pivot to Bayesian decision theory~\citep{bernardo1994bayesian}. While the Bayesian often/typically omits the notion of type-I error control and focuses solely on loss minimization, we show below that he gains the ability to handle post-hoc loss functions for free. We also discuss what a Bayesian notion of type-I risk control might look like, and how to achieve it. 

Moving away from the setting considered in the main paper for a moment, let $\Theta$ be a parameter space, $\calA$ a set of actions, and $L:\Theta\times \calA\to\Re_{\geq 0}$ a loss function. Associated to each $\theta\in\Theta$ is a distribution $P_\theta$. 
Let $\pi$ be a prior on $\Theta$. Given observations $X$, Bayesian decision theory tells us to minimize the expected posterior loss. That is, we take action 
\begin{equation}
\label{eq:bayes_estimator}
    \delta_\pi(X) \equiv \argmin_{a\in\calA} \E_{\theta \sim \pi(\cdot|X)} L(\theta, a).
\end{equation}
A fundamental result is that the \emph{Bayes optimal decision rule}~\citep[Theorem 4.1.1]{lehmann1998theory}, henceforth `Bayes decision' for brevity, is always given by $\delta_\pi$, meaning that it satisfies $B_\pi(\delta_\pi) = \inf_\delta B_\pi(\delta)$ where 
\begin{equation}
\label{eq:bayes_risk}
    B_\pi(\delta) = \E_{\theta\sim \pi} \E_{X\sim P_\theta}[L(\theta, \delta(X))]. 
\end{equation}
(Note that $\theta$ is drawn from the prior $\pi$ in~\eqref{eq:bayes_risk}, not the posterior $\pi(\cdot|X)$ as in~\eqref{eq:bayes_estimator}.) Bayes estimators are, as the name suggests, the Bayesian equivalent to frequentist minimax estimators, so a unified solution for the Bayes decision regardless of the loss function is a big benefit of Bayesian decision theory. 

To apply this to our setting, we take $\calA = \{0,1\}$ and we assume that $\Theta$ is naturally partitioned into $\Theta_0$ and $\Theta_1$, where $\Theta_0$ corresponds to the null and $\Theta_1$ to the alternative. This induces the conditional prior distributions $\pi_i(\theta) = \pi_i(\theta|\theta\in \Theta_i)$ and conditional posterior distributions $\pi_i(\theta|X) = \pi_i(\theta|X,\theta\in\Theta_i)$.  

Assuming as usual that $L(\theta, i) = 0$ for $\theta\in \Theta_i$ and $i\in\{0,1\}$, the Bayes decision becomes  
\begin{align}
    \delta_\pi(X) 
    &= \argmin\big\{ \E_{\theta\sim \pi(\cdot|X)} [L(\theta,0)], \E_{\theta\sim\pi(\cdot|X)}[L(\theta,1)]\big\} \nonumber \\
    &= \argmin\big\{ \E_{\theta\sim \pi_1(\cdot|X)} [L(\theta,0)]\pi(\Theta_1|X), \E_{\theta\sim\pi_0(\cdot|X)}[L(\theta,1)]\pi(\Theta_0|X)\big\} \nonumber \\
    &= \ind{\E_{\theta\sim \pi_1(\cdot|X)} [L(\theta,0)]\pi(\Theta_1|X) \geq \E_{\theta\sim\pi_0(\cdot|X)}[L(\theta,1)]\pi(\Theta_0|X)}.
\end{align}
In words, we play whichever action has the lowest cost if wrong (i.e., results in a false positive or negative), weighted by the probability under the posterior.

To move to a post-hoc setting, let's consider a set of losses $\{L_b\}_{b\in\calB}$. As before, we allow an adversary to choose $b$ after seeing the data and we encode the adversary as a measurable map $B:\calX\to\calB$. As usual, we let the decision rule see $B(X)$ before making a decision. 

Can Bayesian decision theory accommodate a post-hoc selection of the loss? Yes, in the following sense: If we minimize the expected posterior loss as before but on loss $L_{B(X)}$ (which we observe), then the resulting decision rule minimizes Bayes risk, where the Bayes risk is now defined in terms of the loss selected by the adversary. More formally: 

\begin{proposition}
\label{prop:bayesian-post-hoc-valid}
The decision rule 
    \begin{equation*}
\phi_\pi(X) = \argmin_{a\in\calA} \E_{\theta\sim \pi(\cdot|X)} L_{B(X)} (\theta, a),
\end{equation*}
satisfies 
\begin{equation*}
    \E_{\theta\sim\pi} \E_{X\sim P_\theta} L_{B(X)}(\theta, \phi_\pi(X)) = \inf_{\phi} \E_{\theta\sim\pi} \E_{X\sim P_\theta} L_{B(X)}(\theta, \phi(X)), 
\end{equation*}
for all $B:\calX\to\calB$. 
\end{proposition}
We emphasize that $\phi_\pi$ is well-defined since, as mentioned earlier, $\phi$ is allowed to see $B(X)$. We also note that this result is not new. For example, \citet{grunwald2023posterior} notes that that Bayesian decision can accommodate post-hoc loss functions, but we hope the formal treatment above clarifies in what sense this is true, and allows for easier comparison to our frequentist results.  
\begin{proof}
For any decision rule $\phi$, let $\ell(\phi(x) |x) = \E_{\theta\sim \pi(\cdot|x)} L_{B(x)}(\theta, \phi(x))$ be its expected posterior loss at a given $x$ on loss $L_{B(x)}$.  For simplicity, assume that $P_\theta$ has density $p(\cdot|\theta)$. Let $p(x) = \int p(x|\theta)\pi(\theta)\d\theta$ be the marginal and note that $p(x|\theta)\pi(\theta) = \pi(\theta|x)p(x)$. Then, by Fubini's theorem, 
\begin{align*}
    \E_{\theta\sim\pi} \E_{X\sim P_\theta} L_{B(X)}(\theta, \phi(X)) &= \int_\Theta \int_\calX L_{B(x)} (\theta, \phi(x)) p(x|\theta)\pi(\theta)\d x\d\theta \\ 
    &= \int_\calX \int_\Theta L_{B(x)} (\theta, \phi(x)) \pi(\theta|x)p(x)\d \theta\d x \\ 
    &= \int_\calX \ell(\phi(x)|x) p(x) \d\theta.
\end{align*}
By minimizing the expected posterior $\ell(\phi(x)|x)$ we're minimizing the above integrand at every $x$, thus minimizing the value of the integral. 
\end{proof}

The Bayesian paradigm is thus post-hoc valid in the sense of Proposition~\ref{prop:bayesian-post-hoc-valid}. But it does not take into account the notion of type-I error---it is solely concerned with minimizing the loss. In other words, Bayesian theory treats type-I and type-II loss as symmetric, whereas Neyman-Pearson theory, and our generalized version thereof, views type-I risk as taking precedence. A natural question is whether one can incorporate a notion of type-I risk safety into Bayesian decision theory; we provide one answer  below. 

We might consider the following definition of type-I risk in the Bayesian context: 
\begin{equation}
\label{eq:bayesian-type-I-risk}
  S_\pi(\delta)= \sup_{B:\calX\to\calB}\E_{\theta\sim \pi_0} \E_{X\sim P_\theta} L_{B(X)} (\theta, \delta(X)), 
\end{equation}
where we emphasize that $\theta$ is being drawn from the induced prior $\pi_0$ over $\Theta_0$, not from $\pi$ itself. It's helpful to compare~\eqref{eq:bayesian-type-I-risk} to the frequentist counterpart in the main paper, Definition~\ref{def:risk}. First, following the Bayesian philosophy, we've replaced the supremum over $\Theta_0$ with an expectation (the supremum over $\Theta_0$ appears as a supremum over $\calP$ in Definition~\ref{def:risk}). 
But we have also moved the supremum over $B$ to the outside;
we leave moving the supremum inside $\E_{\theta\sim\pi_0}$ as a possible future direction.  

It turns out that we can solve a constrained optimization problem in order provide an alternative estimator which satisfies $S_\pi(\delta) \leq 1$. This alternative estimator is the Bayes decision for a new problem with a different loss function. Only the general form of this estimator can be given, however. It involves a parameter $\lambda$ which, to calculate explicitly, would require knowledge of the particular  distributions involved (and thus can be calculated for any given problem, just not in a closed form manner that applies to all problems). 

\begin{proposition}
    There exists some $\lambda\geq 0$ such that the Bayes decision $\delta_\pi$ with loss
    \begin{equation}
    \widehat{L}(\theta, a) = \begin{cases}
        L(\theta, a), &\theta\in\Theta_1, \\
        \frac{\pi(\Theta_0) + \lambda}{\pi(\Theta_0)}L(\theta, a), & \theta\in\Theta_0,
    \end{cases}
\end{equation}
satisfies $S_\pi(\delta_\pi)\leq 1$ on the original problem.
\end{proposition}

\begin{proof}
    In order to enforce the constraint $S_\pi(\delta_\pi) \leq 1$ we introduce the Lagrange multiplier $\lambda$ and solve the optimization problem 
\begin{equation}
    \min_\delta \big\{ B_\pi(\delta) + \lambda(S_\pi(\delta) - 1)\big\}.
\end{equation}
Expanding, we can collect terms and obtain 
\begin{align*}
    B_\pi(\delta) + \lambda(S_\pi(\delta) - 1) &= (\pi(\Theta_0) + \lambda)\E_{\theta\sim \pi_0}\E_{X\sim P_\theta} L(\theta, \delta(X))  \\
    &\qquad + \E_{\theta\sim\pi_1}\E_{X\sim P_\theta} L(\theta,\delta(X))\pi(\Theta_1) - \lambda \\ 
    &= \E_{\theta\sim\pi} \E_{X\sim P_\theta} \widehat{L}(\theta, \delta(X)) - \lambda. 
\end{align*}
This is minimized by the Bayes estimator in~\eqref{eq:bayes_estimator} with $L$ replaced by $\widehat{L}$. The value of $\lambda$ can then be adjusted to ensure that $S_\pi(\delta)\leq 1$.  
\end{proof}

\section{On risk control vs error probabilities}
\label{app:risk-vs-error}

The formulation of post-hoc hypothesis testing we've advanced here changes the kinds of guarantees that an analyst can give concerning the performance of their tests. The currency of our framework is risk control (Definition~\ref{def:risk}) instead of type-I error control---one obtains a bound on the expected loss instead of the error probability. Given that modern hypothesis testing uses error probabilities, it is reasonable to be skeptical of changing the metric. Here we consider four brief arguments in favor of risk control. To provide a full-throated defense would require its own paper, but it's worth at least sketching these considerations here. 

The first argument is historical. As we've discussed, \citet{wald1939contributions} noted that standard Neyman-Pearson hypothesis testing is a part of classical statistical decision theory once it is reformulated in terms of loss functions. And decision theory has always been concerned with expected loss as the error metric, due to compelling arguments such as representation theorems~\citep{savage1954foundations}, coherence and Dutch books arguments~\citep{de1937prevision}, complete class theorems~\citep{wald1950statistical,blackwell1954theory}, amenability to central limit theorems and laws of large numbers, and so on. Thus, our emphasis on risk control can be seen as a return to the roots of hypothesis testing as decision theory, where expected risk is the norm.

The second is a pragmatic argument. As we mentioned in the introduction, there are alternatives to Neyman-Pearson hypothesis testing, the most successful of which is the Bayesian agenda. But if one is fond of the tools of the Neyman-Pearson paradigm---such as significance levels, power, and accept/reject procedures---and seeks to generalize those tools to handle post-hoc analysis, then the formulation of post-hoc hypothesis testing studied here (and in \citealp{grunwald2024beyond}) appears to be the only option currently on offer. (See Appendix~\ref{app:bayesian} for a longer discussion on Bayesian decision theory and how it relates to our framework.)  

The third argument is pedagogical. It is well known that p-values and error probabilities are misunderstood by practitioners, necessitating a near-constant stream of mildly condescending educational articles on their correct use and interpretation. Changing the metric to risk control (and reporting e-values instead of p-values) may be prophylactic, allowing us to leave some of this confusion behind. 

Last is a normative argument. Traditional hypothesis testing, with its focus on controlling error probabilities, can mislead practitioners into treating the results as definitive solely based on whether $p\leq \alpha$. This obscures the fundamental uncertainty inherent in any single study. Risk control, by contrast, emphasizes expected loss across many applications of a testing procedure. The guarantee is only that the average loss will be controlled in the long run. It emphasizes that any individual study may be wrong, and one should remain humble when asserting conclusions.

\end{document}